\documentclass[twoside]{article}

\usepackage{jmlr2e} 
\usepackage{graphicx} 
\usepackage{subfigure} 

\usepackage{algorithm}
\usepackage{algorithmic}

\usepackage{hyperref}

\usepackage{amsmath}
\usepackage{pgfplots}
\usepackage{color}
\usepackage{needspace}

\usepackage{flushend}
\usepackage{multirow}
\usepackage{rotating}
\usepackage{appendix}

\input{mysymbol.sty}

\newtheorem{assumption}{\hspace{0pt}\bf Assumption}

\begin{document}


\jmlrheading{1}{2000}{1-48}{4/00}{10/00}{Aryan Mokhtari and Alejandro Ribeiro}


\ShortHeadings{Global Convergence of Online Limited Memory BFGS}{Mokhtari and Ribeiro}
\firstpageno{1}



\title{Global Convergence of Online Limited Memory BFGS}

\author{\name Aryan Mokhtari  \email aryanm@seas.upenn.edu \\
       \name Alejandro Ribeiro \email aribeiro@seas.upenn.edu  \\
       \addr Department of Electrical and Systems Engineering\\
       University of Pennsylvania\\
       Philadelphia, PA 19104, USA}

\editor{}

\maketitle


\thispagestyle{empty}

\begin{abstract}

Global convergence of an online (stochastic) limited memory version of the Broyden-Fletcher-Goldfarb-Shanno (BFGS) quasi-Newton method for solving optimization problems with stochastic objectives that arise in large scale machine learning is established. Lower and upper bounds on the Hessian eigenvalues of the sample functions are shown to suffice to guarantee that the curvature approximation matrices have bounded determinants and traces, which, in turn, permits establishing convergence to optimal arguments with probability 1. Numerical experiments on support vector machines with synthetic data showcase reductions in convergence time relative to stochastic gradient descent algorithms as well as reductions in storage and computation relative to other online quasi-Newton methods. Experimental evaluation on a search engine advertising problem corroborates that these advantages also manifest in practical applications.

\end{abstract}



\section{Introduction}\label{sec_Introduction}

Many problems in Machine Learning can be reduced to the minimization of a stochastic objective defined as an expectation over a set of random functions (\cite{BottouCun, Bottou, SS, cMokhtariRibeiro14}).
Specifically, consider an optimization variable $\bbw \in \reals^{n}$ and a random variable $\bbtheta \in \Theta\subseteq\reals^p$ that determines the choice of a function $f(\bbw,{\bbtheta}):\reals^{n\times p}\to\reals$. Stochastic optimization problems entail determination of the argument $\bbw^*$ that minimizes the expected value $F(\bbw):=\mbE_{\bbtheta}[f(\bbw,{\bbtheta})]$,
\begin{equation}\label{optimization_problem}
   \bbw^* \ :=\ \argmin_\bbw \mbE_{\bbtheta}[f(\bbw,{\bbtheta})]
          \ :=\ \argmin_\bbw {F(\bbw)}.
\end{equation} 
We refer to $f(\bbw,{\bbtheta})$ as the random or instantaneous functions and to $F(\bbw):=\mbE_{\bbtheta}[f(\bbw,{\bbtheta})]$ as the average function. A canonical class of problems having this form are support vector machines (SVMs) that reduce binary classification to the determination of a hyperplane that separates points in a given training set; see, e.g., (\cite{Vapnik, Bottou, BGV}). In that case $\bbtheta$ denotes individual training samples, $f(\bbw,{\bbtheta})$ the loss of choosing the hyperplane defined by $\bbw$, and $F(\bbw):=\mbE_{\bbtheta}[f(\bbw,{\bbtheta})]$ the mean loss across all elements of the training set. The optimal argument $\bbw^*$ is the optimal linear classifier. 

Numerical evaluation of objective function gradients $\nabla_{\bbw} F(\bbw)=\mbE_{\bbtheta}[\nabla_{\bbw} f(\bbw,{\bbtheta})]$ is intractable when the cardinality of $\Theta$ is large, as is the case, e.g., when SVMs are trained on large sets. This motivates the use of algorithms relying on stochastic gradients that provide gradient estimates based on small data subsamples. For the purpose of this paper stochastic optimization algorithms can be divided into three categories: Stochastic gradient descent (SGD) and related first order methods, stochastic Newton methods, and stochastic quasi-Newton methods.

SGD is the most popular method used to solve stochastic optimization problems (\cite{Bottou,Shwartz,Zhang,Roux}). However, as we consider problems of ever larger dimension their slow convergence times have limited their practical appeal and fostered the search for alternatives. In this regard, it has to be noted that SGD is slow because of both, the use of gradients as descent directions and their replacement by random estimates. Several alternatives have been proposed to deal with randomness in an effort to render the convergence times of SGD closer to the faster convergence times of gradient descent (\cite{Ruszczynski, Konency, Mahdavi}). These SGD variants succeed in reducing randomness and end up exhibiting the asymptotic convergence rate of gradient descent. Although they improve asymptotic convergence rates, the latter methods are still often slow in practice. This is not unexpected. Reducing randomness is of no use when the function $F(\bbw)$ has a challenging curvature profile. In these ill-conditioned functions SGD is limited by the already slow convergence times of deterministic gradient descent. The golden standard to deal with ill-conditioned functions in a deterministic setting is Newton's method. However, unbiased stochastic estimates of Newton steps can't be computed in general. This fact limits the application of stochastic Newton methods to problems with specific structure (\cite{Birge, ZarghamEtal14}).

If SGD is slow to converge and stochastic Newton can't be used in general, the remaining alternative is to modify deterministic quasi-Newton methods that speed up convergence times relative to gradient descent without using Hessian evaluations (\cite{Dennis,Powell,Byrd,Nocedal}). This has resulted in the development of the stochastic quasi-Newton methods known as online (o) Broyden-Fletcher-Goldfarb-Shanno (BFGS) (\cite{Schraudolph}), regularized stochastic BFGS (RES) (\cite{AryanAleTSP}), and online limited memory (oL)BFGS  (\cite{Schraudolph}) which occupy the middle ground of broad applicability irrespective of problem structure and conditioning. All three of these algorithms extend BFGS by using stochastic gradients both as descent directions and constituents of Hessian estimates. The oBFGS algorithm is a direct generalization of BFGS that uses stochastic gradients in lieu of deterministic gradients. RES differs in that it further modifies BFGS to yield an algorithm that retains its convergence advantages while improving theoretical convergence guarantees and numerical behavior. The oLBFGS method uses a modification of BFGS to reduce the computational cost of each iteration.

An important observation here is that in trying to adapt to the changing curvature of the objective, stochastic quasi-Newton methods may end up exacerbating the problem. Indeed, since Hessian estimates are stochastic, it is possible to end up with almost singular Hessian estimates. The corresponding small eigenvalues then result in a catastrophic amplification of the noise which nullifies progress made towards convergence. This is not a minor problem. In oBFGS this possibility precludes convergence analyses (\cite{Bordes, Schraudolph}) and may result in erratic numerical behavior (\cite{AryanAleTSP}). As a matter of fact, the main motivation for the introduction of RES is to avoid this catastrophic noise amplification so as to retain smaller convergence times while ensuring that optimal arguments are found with probability 1 (\cite{AryanAleTSP}). However valuable, the convergence guarantees of RES and the convergence time advantages of oBFGS and RES are tainted by an iteration cost of order $O(n^2)$ and $O(n^3)$, respectively, which precludes their use in problems where $n$ is very large. In deterministic settings this problem is addressed by limited memory (L)BFGS (\cite{DingNocedal}) which can be easily generalized to develop the oLBFGS algorithm (\cite{Schraudolph}). Numerical tests of oLBFGS are promising but theoretical convergence characterizations are still lacking. The main contribution of this paper is to show that oLBFGS converges with probability 1 to optimal arguments across realizations of the random variables $\theta$. This is the same convergence guarantee provided for RES and is in marked contrast with oBFGS, which fails to converge if not properly regularized. Convergence guarantees for oLBFGS do not require such measures.

We begin the paper with brief discussions of deterministic BFGS (Section \ref{sec:problem}) and  LBFGS (Section \ref{sec_limmem_bfgs}) and the introduction of oLBFGS (Section \ref{sec_lim_mem_stochastic_bfgs}). The fundamental idea in BFGS and oLBFGS is to continuously satisfy a secant condition while staying close to previous curvature estimates. They differ in that BFGS uses all past gradients to estimate curvature while oLBFGS uses a fixed moving window of past gradients. The use of this window reduces memory and computational cost (Appendix \ref{apx_lbfgs_reduced_cost}). The difference between LBFGS and oLBFGS is the use of stochastic gradients in lieu of their deterministic counterparts.

Convergence properties of oLBFGS are then analyzed (Section \ref{sec_convergence}). Under the assumption that the sample functions $f(\bbw,{\bbtheta})$ are strongly convex we show that the trace and determinant of the Hessian approximations computed by oLBFGS are upper and lower bounded, respectively {(Lemma \ref{lemma_determinant_and_trace_bounds}). These bounds are then used to limit the range of variation of the ratio between the Hessian approximations' largest and smallest eigenvalues (Lemma \ref{norooz}). In turn, this condition number limit is shown to be sufficient to prove convergence to the optimal argument $\bbw^*$ with probability 1 over realizations of the sample functions (Theorem \ref{convg}). This is an important result because it ensures that oLBFGS doesn't suffer from the numerical problems that hinder oBFGS. We complement this almost sure convergence result with a characterization of the convergence rate which is shown to be at least $O(1/t)$ in expectation (Theorem \ref{theo_convergence_rate}). It is fair to emphasize that, different from the deterministic case, the convergence rate of oLBFGS is not better than the convergence rate of SGD. This is not a limitation of our analysis. The difference between stochastic and regular gradients introduces a noise term that dominates convergence once we are close to the optimum, which is where superlinear convergence rates manifest. In fact, the same convergence rate would be observed if exact Hessians were available. The best that can be proven of oLBFGS is that the convergence rate is not worse than that of SGD. Given that theoretical guarantees only state that the curvature correction does not exacerbate the problem's condition it is perhaps fairer to describe oLBFGS as an adaptive reconditioning strategy instead of a stochastic quasi-Newton method. The latter description refers to the genesis of the algorithm. The former is a more accurate description of its actual behavior.

To show the advantage of using oLBFGS as an adaptive reconditioning strategy we develop its application to SVM  problems (Section \ref{sec:SVMproblem}) and perform a comparative numerical analysis with synthetic data. The conclusions of this numerical analysis are that oLBFGS performs as well as oBFGS and RES while outperforming SGD when convergence is measured with respect to the number of feature vectors processed. In terms of computation time, oLBFGS outperforms all three methods, SGD, oBFGS, and RES. The advantages of oLBFGS grow with the dimension of the feature vector and can be made arbitrarily large (Section \ref{sec:SVM2}). To further substantiate numerical claims we use oLBFGS to train a logistic regressor to predict the click through rate in a search engine advertising problem (Section \ref{sec:ad_application}). The logistic regression uses a heterogeneous feature vector with 174,026 binary entries that describe the user, the search, and the advertisement (Section \ref{sec:dataset}). Being a large scale problem with heterogeneous data, the condition number of the logistic log likelihood objective is large and we expect to see significant advantages of oLBFGS relative to SGD. This expectation is fulfilled. The oLBFGS algorithm trains the regressor using less than $1\%$ of the data required by SGD to obtain similar classification accuracy. (Section \ref{sec_ad_numerical_results}). We close the paper with concluding remarks (Section \ref{sec_conclusions}).

\medskip\noindent{\bf Notation\quad} Lowercase boldface $\bbv$ denotes a vector and uppercase boldface $\bbA$ a matrix. We use $\|\bbv\|$ to denote the Euclidean norm of vector $\bbv$ and $\|\bbA\|$ to denote the Euclidean norm of matrix $\bbA$. The trace of $\bbA$ is written as $\tr(\bbA)$ and the determinant as $\det(\bbA)$. We use $\bbI$ for the identity matrix of appropriate dimension. The notation $\bbA\succeq\bbB$ implies that the matrix $\bbA- \bbB$ is positive semidefinite. The operator $\mbE_{\bbx}[\cdot]$ stands in for expectation over random variable $\bbx$ and  $\mbE[\cdot]$ for expectation with respect to the distribution of a stochastic process.

%
\section{Algorithm definition}\label{sec:problem}

Recall the definitions of the sample functions $f(\bbw,\bbtheta)$ and the average function $F(\bbw):=\mbE_{\bbtheta}[f(\bbw,{\bbtheta})]$. We assume the sample functions $f(\bbw,\bbtheta)$ are strongly convex for all $\bbtheta$. This implies the objective function $F(\bbw):=\mbE_{\bbtheta}[f(\bbw,{\bbtheta})]$, being an average of the strongly convex sample functions, is also strongly convex. We define the gradient $\bbs(\bbw) := \nabla F(\bbw)$  of the average function $F(\bbw)$ and assume that it can be computed as 
\begin{align}\label{gradient}
    \bbs(\bbw) := \nabla F(\bbw) 
                  = \mbE_{\bbtheta}[\nabla f(\bbw,{\bbtheta})].
\end{align}
Since the function $F(\bbw)$ is strongly convex, gradients $\bbs(\bbw)$ are descent directions that can be used to find the optimal argument $\bbw^{*}$ in \eqref{optimization_problem}. Introduce then a time index $t$, a step size $\eps_t$, and a positive definite matrix $\bbB_t^{-1}\succ0$ to define a generic descent algorithm through the iteration
\begin{equation}\label{eqn_bfgs_descent}
   \bbw_{t+1} \ =\ \bbw_{t}-\epsilon_{t}\ \bbB_t^{-1}\bbs(\bbw_t) 
              \ =\ \bbw_{t}-\epsilon_{t}\ \bbd_t.
\end{equation}
where we have also defined the descent step $\bbd_t=\bbB_t^{-1}\bbs(\bbw_t)$. When $\bbB_t^{-1}=\bbI$ is the identity matrix, \eqref{eqn_bfgs_descent} reduces to gradient descent. When $\bbB_t=\bbH(\bbw_t):=\nabla^2 F(\bbw_t)$ is the Hessian of the objective function, \eqref{eqn_bfgs_descent} defines Newton's algorithm. In this paper we focus on quasi-Newton methods whereby we attempt to select matrices $\bbB_t$ close to the Hessian $\bbH(\bbw_t)$. Various methods are known to select matrices $\bbB_t$, including those by Broyden e.g., \cite{Broyden}; {Davidon, Fletcher, and Powell (DFP) e.g., \cite{DFP}; and Broyden, Fletcher, Goldfarb, and Shanno (BFGS) e.g., \cite{Byrd,Powell}. We work with the matrices $\bbB_t$ used in BFGS since they have been observed to work best in practice (see \cite{Byrd}). 

In BFGS, the function's curvature $\bbB_t$ is approximated by a finite difference. Let $\bbv_{t}$ denote the variable variation at time $t$ and $\bbr_{t}$ the gradient variation at time $t$ which are respectively defined as
\begin{equation}\label{ball}
      \bbv_{t} := \bbw_{t+1}-\bbw_{t}, \qquad
      \bbr_{t}  := \bbs(\bbw_{t+1})-\bbs(\bbw_t).
\end{equation}
We select the matrix $\bbB_{t+1}$ to be used in the next time step so that it satisfies the secant condition $ \bbB_{t+1} \bbv_{t} = \bbr_{t}$. The rationale for this selection is that {the Hessian $\bbH(\bbw_t)$ satisfies this condition for $\bbw_{t+1}$ tending to $\bbw_{t}$}. Notice however that the secant condition $ \bbB_{t+1} \bbv_{t} = \bbr_{t}$ is not enough to completely specify $\bbB_{t+1}$. 
To resolve this indeterminacy, matrices $\bbB_{t+1}$ in BFGS are also required to be as close as possible to the previous Hessian approximation $\bbB_{t}$ in terms of differential entropy. 
These conditions can be resolved in closed form leading to the explicit expression -- see, e.g., \cite{Nocedal} --,
\begin{equation}\label{answer2}
   \bbB_{t+1} = \bbB_{t}
              + \frac{\bbr_{t}\bbr_{t}^{T}}
                     {\bbv_{t}^{T }\bbr_{t}}
              - \frac{\bbB_{t}\bbv_{t}\bbv_{t}^{T}\bbB_{t}}
                     {\bbv_{t}^{T} \bbB_{t}  \bbv_{t}}.
\end{equation}
While the expression in \eqref{answer2} permits updating the Hessian approximations $\bbB_{t+1}$, implementation of the descent step in \eqref{eqn_bfgs_descent} requires its inversion. This can be avoided by using the Sherman-Morrison formula in \eqref{answer2} to write
\begin{equation}\label{answer3}
   \bbB_{t+1}^{-1}\ =\ \bbZ_{t}^{T}\ \bbB_{t}^{-1}\ \bbZ_{t}
		 	           \ +\ \rho_{t}\ {\bbv_{t}\bbv_{t}^{T}},
\end{equation}
where we defined the scalar $\rho_t$ and the matrix $\bbZ_t$ as
\begin{equation}\label{parameteres}
 	 \rho_{t}\ :=\ \frac{1}{\bbv_{t}^{T}\bbr_{t}} \ ,         \qquad
	 \bbZ_{t}\ :=\ \bbI - \rho_{t} \bbr_{t} \bbv_{t}^{T}.
\end{equation}
The updates in \eqref{answer2} and \eqref{answer3} require the inner product of the gradient and variable variations to be positive, i.e., $\bbv_{t}^T\bbr_{t}>0$. This is always true if the objective $F(\bbw)$ is strongly convex and further implies that $\bbB_{t+1}^{-1}$ stays positive definite if $\bbB_{t}^{-1}\succ \bb0$, \cite{Nocedal}.

Each BFGS iteration has a cost of $O(n^{2})$ arithmetic operations. This is less than the  $O(n^{3})$ of each step in Newton's method but more than the $O(n)$ cost of each gradient descent iteration. In general, the relative convergence rates are such that the total computational cost of BFGS to achieve a target accuracy is smaller than the corresponding cost of gradient descent. Still, alternatives to reduce the computational cost of each iteration are of interest for large scale problems. Likewise, BFGS requires storage and propagation of the $O(n^{2})$ elements of $\bbB_{t}^{-1}$, whereas gradient descent requires storage of $O(n)$ gradient elements only. This motivates alternatives that have smaller memory footprints. Both of these objectives are accomplished by the limited memory (L)BFGS algorithm that we describe in the following section.

%
\subsection{LBFGS: Limited memory BFGS}\label{sec_limmem_bfgs}

As it follows from \eqref{answer3}, the updated Hessian inverse approximation $\bbB_{t}^{-1}$ depends on $\bbB_{t-1}^{-1}$ and the curvature information pairs $\{\bbv_{t-1}$, $\bbr_{t-1}\}$. In turn, to compute $\bbB_{t-1}^{-1}$, the estimate $\bbB_{t-2}^{-1}$ and the curvature pair $\{\bbv_{t-2}, \bbr_{t-2}\}$ are used. Proceeding recursively, it follows that $\bbB_{t}^{-1}$ is a function of the initial approximation $\bbB_{0}^{-1}$ and all previous $t$ curvature information pairs $\{\bbv_{u}$, $\bbr_{u}\}_{u=0}^{t-1}$. The idea in LBFGS is to restrict the use of past curvature information to the last $\tau$ pairs $\{\bbv_{u}, \bbr_{u}\}_{u=t-\tau}^{t-1}$. Since earlier iterates $\{\bbv_{u}, \bbr_{u}\}$ with $u<t-\tau$ are likely to carry little information about the curvature at the current iterate $\bbw_t$, this restriction is expected to result in a minimal performance penalty.

For a precise definition, pick a positive definite matrix $\bbB_{t,0}^{-1}$ as the initial Hessian inverse approximation at step $t$. Proceed then to perform $\tau$ updates of the form in \eqref{answer3} using the last $\tau$ curvature information pairs $\{\bbv_{u},\bbr_{u}\}_{u=t-\tau}^{t-1}$. Denoting as $\bbB_{t,u}^{-1}$ the curvature approximation after $u$ updates are performed we have that the refined matrix approximation $\bbB_{t,u+1}^{-1}$ is given by [cf. \eqref{answer3}]
\begin{align}\label{LBFGS_update}
   \bbB_{t,u+1}^{-1} \
        = \  & \bbZ_{t-\tau+u}^{T}\ \bbB_{t,u}^{-1}\   \bbZ_{t-\tau+u} 
		 	  \ +\  \rho_{t-\tau+u}\ {\bbv_{t-\tau+u}\ \! \bbv_{t-\tau+u}^{T}},
\end{align}
where $u=0,\dots,\tau-1$ and the constants $\rho_{t-\tau+u}$ and rank-one plus identity matrices $\bbZ_{t-\tau+u}$ are as given in \eqref{parameteres}. The inverse Hessian approximation $\bbB_{t}^{-1}$ to be used in \eqref{eqn_bfgs_descent} is the one yielded after completing the $\tau$ updates in \eqref{LBFGS_update}, i.e., $\bbB_{t}^{-1}=\bbB_{t,\tau}^{-1}$. Observe that when $t<\tau$ there are not enough pairs $\{\bbv_{u},\bbr_{u}\}$ to perform $\tau$ updates. In such case we just redefine $\tau=t$ and proceed to use the $t=\tau$ available pairs $\{\bbv_{u},\bbr_{u}\}_{u=0}^{t-1}$ .

Implementation of the product $\bbB_t^{-1}\bbs(\bbw_t)$ in \eqref{eqn_bfgs_descent} for matrices $\bbB_{t}^{-1}=\bbB_{t,\tau}^{-1}$ obtained from the recursion in \eqref{LBFGS_update} does not need explicit computation of the matrix $\bbB_{t,\tau}^{-1}$. Although the details are not straightforward, observe that each iteration in \eqref{LBFGS_update} is similar to a rank-one update and that as such it is not unreasonable to expect that the product $\bbB_t^{-1}\bbs(\bbw_t)=\bbB_{t,\tau}^{-1}\bbs(\bbw_t)$ can be computed using $\tau$ recursive inner products. Assuming that this is possible, the implementation of the recursion in \eqref{LBFGS_update} doesn't need computation and storage of prior matrices $\bbB_{t-1}^{-1}$. Rather, it suffices to keep the $\tau$ most recent curvature information pairs $\{\bbv_{u},\bbr_{u}\}_{u=t-\tau}^{t-1}$, thus reducing storage requirements from $O(n^2)$ to $O(\tau n)$. Furthermore, each of these inner products can be computed at a cost of $n$ operations yielding a total computational cost of $O(\tau n)$ per LBFGS iteration. Hence, LBFGS decreases both the memory requirements and the computational cost of each iteration from the $O(n^2)$ required by regular BFGS to $O(\tau n)$. We present the details of this iteration in the context of the online (stochastic) LBFGS that we introduce in the following section.

%
\subsection{Online (Stochastic) Limited memory BFGS}\label{sec_lim_mem_stochastic_bfgs}

To implement \eqref{eqn_bfgs_descent} and \eqref{LBFGS_update} we need to compute gradients $\bbs(\bbw_t)$. This is impractical when the number of functions $f(\bbw,\bbtheta)$ is large, as is the case in most stochastic problems of practical interest and motivates the use of stochastic gradients in lieu of actual gradients. Consider a given set of $L$ realizations $\tbtheta=[\bbtheta_{1};...;\bbtheta_{L}]$ and define the stochastic gradient of $F(\bbw)$ at $\bbw$ given samples $\tbtheta$ as
\begin{equation}\label{stochastic_gradient}
   \hbs(\bbw,\tbtheta) 
          := \frac{1}{L}\sum_{l=1}^{L}  \nabla f(\bbw,{\bbtheta_{l}}).
\end{equation}
In oLBFGS we use stochastic gradients $\hbs(\bbw,\tbtheta)$ for descent directions and curvature estimators. In particular, the descent iteration in \eqref{eqn_bfgs_descent} is replaced by the descent iteration 
\begin{equation}\label{eqn_olbfgs_descent}
   \bbw_{t+1}\ =\ \bbw_{t}-\epsilon_{t}\ \hbB_t^{-1}\hbs(\bbw_t,\tbtheta_t)
             \ =\ \bbw_{t} - \epsilon_{t}\hbd_{t} ,
\end{equation}
where $\tbtheta_t=[\bbtheta_{t1};...;\bbtheta_{tL}]$ is the set of samples used at step $t$ to compute the stochastic gradient $\hbs(\bbw_t,\tbtheta_t)$ as per \eqref{stochastic_gradient} and the matrix $\hbB_t^{-1}$ is a function of past stochastic gradients $\hbs(\bbw_u,\tbtheta_u)$ with $u\leq t$ instead of a function of past gradients $\bbs(\bbw_u)$ as in \eqref{eqn_bfgs_descent}. As we also did in \eqref{eqn_bfgs_descent} we have defined the stochastic step $\hbd_{t}:=\hbB_t^{-1}\hbs(\bbw_t,\tbtheta_t)$ to simplify upcoming discussions.

To properly specify $\hbB_t^{-1}$ we define the stochastic gradient variation $\hbr_{t}$ at time $t$ as the difference between the stochastic gradients $\hbs(\bbw_{t+1},\tbtheta_{t})$ and $\hbs(\bbw_{t},\tbtheta_{t})$ associated with subsequent iterates $\bbw_{t+1}$ and $\bbw_{t}$ and the {\it common} set of samples $\tbtheta_t$ [cf. \eqref{ball}],
\begin{equation}\label{chris}
   \hbr_{t} :=\hbs(\bbw_{t+1},\tbtheta_{t})-\hbs(\bbw_{t},\tbtheta_{t}).
\end{equation}
Observe that $\hbs(\bbw_{t},\tbtheta_{t})$ is the stochastic gradient used at time $t$ in  \eqref{eqn_olbfgs_descent} but that $\hbs(\bbw_{t+1},\tbtheta_{t})$ is computed solely for the purpose of determining the stochastic gradient variation. The perhaps more natural definition $\hbs(\bbw_{t+1},\tbtheta_{t+1})-\hbs(\bbw_{t},\tbtheta_{t})$ for the stochastic gradient variation, which relies on the stochastic gradient $\hbs(\bbw_{t+1},\tbtheta_{t+1})$ used at time $t+1$ in \eqref{eqn_olbfgs_descent} is not sufficient to guarantee convergence; see e.g.,(\cite{Schraudolph, AryanAleTSP}).

To define the oLBFGS algorithm we just need to provide stochastic versions of the definitions in \eqref{parameteres} and \eqref{LBFGS_update}. The scalar constants and identity plus rank-one matrices in \eqref{parameteres} are redefined to the corresponding stochastic quantities
\begin{equation}\label{Z_rho_definitions}
   \hat{\rho}_{t-\tau+u}\ =\ \frac{1}{\bbv_{t-\tau+u}^{T}\hbr_{t-\tau+u}}
   \quad \text{and}\quad
   \hbZ_{t-\tau+u}\ =\ \bbI - \hat{\rho}_{t-\tau+u} \hbr_{t-\tau+u} \bbv_{t-\tau+u}^{T},
\end{equation}
whereas the LBFGS matrix $\bbB_t^{-1}= \bbB_{t,\tau}^{-1}$ in \eqref{LBFGS_update} is replaced by the oLBFGS Hessian inverse approximation $\hbB_t^{-1}= \hbB_{t,\tau}^{-1}$ which we define as the outcome of $\tau$ recursive applications of the update,
\begin{align}\label{oLBFGS_update}
   \hbB_{t,u+1}^{-1} \
      =\ \hbZ_{t-\tau+u}^{T}\ \hbB_{t,u}^{-1}\ \hbZ_{t-\tau+u}
      + \hat{\rho}_{t-\tau+u}\ {\bbv_{t-\tau+u}\ \bbv_{t-\tau+u}^{T}},
\end{align}
where the initial matrix $\hbB_{t,0}^{-1}$ is given and the time index is $u=0,\ldots,\tau-1$. The oLBFGS algorithm is defined by the stochastic descent iteration in \eqref{eqn_olbfgs_descent} with matrices $\hbB_t^{-1}= \hbB_{t,\tau}^{-1}$ computed by $\tau$ recursive applications of \eqref{oLBFGS_update}. Except for the fact that they use stochastic variables, \eqref{eqn_olbfgs_descent} and \eqref{oLBFGS_update} are identical to \eqref{eqn_bfgs_descent} and \eqref{LBFGS_update}. Thus, as is the case in \eqref{eqn_bfgs_descent}, the Hessian inverse approximation $\hbB_{t}^{-1}$ in \eqref{oLBFGS_update} is a function of the initial Hessian inverse approximation $\bbB_{t,0}^{-1}$ and the $\tau$ most recent curvature information pairs $\{\bbv_{u},\hbr_{u}\}_{u=t-\tau}^{t-1}$. Likewise, when $t<\tau$ there are not enough pairs $\{\bbv_{u},\hbr_{u}\}$ to perform $\tau$ updates. In such case we just redefine $\tau=t$ and proceed to use the $t=\tau$ available pairs $\{\bbv_{u},\hbr_{u}\}_{u=0}^{t-1}$ . We also point out that the update in \eqref{oLBFGS_update} necessitates $\hbr_u^T\bbv_u >0$ for all time indexes $u$. This is true as long as the instantaneous functions $f(\bbw,{\bbtheta})$ are strongly convex with respect to $\bbw$ as we show in Lemma \ref{lecce}.

The equations in \eqref{eqn_olbfgs_descent} and \eqref{oLBFGS_update} are used conceptually but not in practical implementations. For the latter we exploit the structure of \eqref{oLBFGS_update} to rearrange the terms in the computation of the product $\hbB_t^{-1}\hbs(\bbw_t,\tbtheta_t)$. To see how this is done consider the recursive update for the Hessian inverse approximation $\hbB_{t}^{-1}$ in \eqref{oLBFGS_update} and make $u=\tau-1$ to write
\begin{align}\label{eqn_oLBFGS_step_computation_10}
   \hbB_{t}^{-1} 
      \ =\ \hbB_{t,\tau}^{-1}
      \ =\ \left(\hbZ_{t-1}^{T}\right) \hbB_{t,\tau-1}^{-1} \left(\hbZ_{t-1}\right)
           + \hat{\rho}_{t-1}\ {\bbv_{t-1}\ \bbv_{t-1}^{T}} .
\end{align}
Equation \eqref{eqn_oLBFGS_step_computation_10} shows the relation between the Hessian inverse approximation $\hbB_{t}^{-1} $ and the $(\tau-1 )$st updated version of the initial Hessian inverse approximation $\hbB_{t,\tau-1}^{-1}$ at step $t$. Set now $u=\tau-2$ in \eqref{oLBFGS_update} to express $\hbB_{t,\tau-1}^{-1}$ in terms of $\hbB_{t,\tau-2}^{-1}$ and substitute the result in \eqref{eqn_oLBFGS_step_computation_10} to rewrite $\hbB_{t}^{-1}$ as
\begin{align}\label{eqn_oLBFGS_step_computation_13}
   \hbB_{t}^{-1} 
      \ =\  &\left(\hbZ_{t-1}^{T}\hbZ_{t-2}^{T}\right) \hbB_{t,\tau-2}^{-1} 
             \left(\hbZ_{t-2} \hbZ_{t-1}\right) 
            +\ \hat{\rho}_{t-2}\left(\hbZ_{t-1}^{T}\right) {\bbv_{t-2}\ \bbv_{t-2}^{T}}  
                               \left(\hbZ_{t-1}\right)
            +\ \hat{\rho}_{t-1}\ {\bbv_{t-1}\ \bbv_{t-1}^{T}} .
\end{align}
We can proceed recursively by substituting $\hbB_{t,\tau-2}^{-1}$ for its expression in terms of $\hbB_{t,\tau-3}^{-1}$ and in the result substitute $\hbB_{t,\tau-3}^{-1}$ for its expression in terms of $\hbB_{t,\tau-3}^{-1}$ and so on. Observe that a new summand is added in each of these substitutions from which it follows that repeating this process $\tau$ times yields
\begin{align}\label{SLBFGS_update}
   \hbB_{t}^{-1}
       =\ &    \left(\hbZ_{t-1}^{T}\ldots\hbZ_{t-\tau}^{T}\right) 
                \hbB_{t,0}^{-1} 
                \left(\hbZ_{t-\tau}\dots\hbZ_{t-1}\right)        
            +\ \hat{ \rho}_{t-\tau} 
                \left( \hbZ_{t-1}^{T} \dots \hbZ_{t-\tau+1}^{T} \right) 
                \bbv_{t-\tau} \bbv_{t-\tau}^{T}
                \left( \hbZ_{t-\tau+1}\dots \hbZ_{t-1} \right)   \nonumber \\
           & +\ \dots 
             +\ \hat{\rho}_{t-2} \left( \hbZ_{t-1}^{T}\right) 
                \bbv_{t-2}\bbv_{t-2}^{T}
                \left( \hbZ_{t-1}\right) 
             +\ \hat{\rho}_{t-1} 
                \bbv_{t-1}\bbv_{t-1}^{T}.
\end{align}
The important observation in \eqref{SLBFGS_update} is that the matrix $\hbZ_{t-1}$ and its transpose $\hbZ_{t-1}^T$ are the first and last product terms of all summands except the last, that the matrices $\hbZ_{t-2}$ and its transpose $\hbZ_{t-2}^T$ are second and penultimate in all terms but the last two, and so on. Thus, when computing the oLBFGS step $\hbd_{t}:=\hbB_t^{-1}\hbs(\bbw_t,\tbtheta_t)$ the operations needed to compute the product with the next to last summand of \eqref{SLBFGS_update} can be reused to compute the product with the second to last summand which in turn can be reused in determining the product with the third to last summand and so on. This observation compounded with the fact that multiplications with the identity plus rank one matrices $\hbZ_{t-1}$ requires $O(n)$ operations yields an algorithm that can compute the oLBFGS step $\hbd_{t}:=\hbB_t^{-1}\hbs(\bbw_t,\tbtheta_t)$ in $O(\tau n)$  operations. We summarize the specifics of such computation in the following proposition where we consider the computation of the product $\hbB_t^{-1}\bbp$ with a given arbitrary vector $\bbp$.

%
\begin{proposition}\label{oLBFGS_efficient_descent_dircetion_computation} 
 Consider the oLBFGS Hessian inverse approximation $\hbB_t^{-1}= \hbB_{t,\tau}^{-1}$ obtained after $\tau$ recursive applications of the update in \eqref{oLBFGS_update} with the scalar sequence $\hat{\rho}_{t-\tau+u}$ and identity plus rank-one matrix sequence $\hbZ_{t-\tau+u}$ as defined in \eqref{Z_rho_definitions} for given variable and stochastic gradient variation pairs $\{\bbv_{u}, \bbr_{u}\}_{u=t-\tau}^{t-1}$. For a given vector $\bbp=\bbp_0$ define the sequence of vectors $\bbp_{k}$ through the recursion 
\begin{equation}\label{eqn_define_pt}
   \bbp_{u+1} 
      \ =\  \bbp_{u}-\alpha_{u}\hbr_{t-u-1} \quad 
      \quad \text{for\ } u=0,\ldots , \tau-1,
\end{equation}
where we also define the constants $\alpha_{u}:=\hat{\rho}_{t-u-1}\bbv_{t-u-1}^{T}\bbp_{u}$. Further define the sequence of vectors $\bbq_{k}$ with initial value $\bbq_0= \hbB_{t,0}^{-1} \bbp_{\tau}$ and subsequent elements
\begin{equation}\label{eqn_define_qt}
   \bbq_{u+1} 
      \ =\ \bbq_{u} + 
              ( \alpha_{\tau-u-1} -  \beta_{u}) \bbv_{t-\tau+u}
      \quad \text{for\ } u=0,\ldots, \tau-1,
\end{equation}
where we define constants $\beta_{u}:= \hat{\rho}_{t-\tau+u}\hbr_{t-\tau+u}^{T}\bbq_{u}$. The product $\hbB_{t}^{-1}\bbp$ equals $\bbq_{\tau}$, i.e., $\hbB_{t}^{-1}\bbp\ =\ \bbq_{\tau}$.
\end{proposition}

\begin{proof} See Appendix \ref{apx_lbfgs_reduced_cost}. \end{proof}
%
Proposition \ref{oLBFGS_efficient_descent_dircetion_computation} asserts that it is possible to reduce the computation of the product $\hbB_{t}^{-1}\bbp$ between the oLBFGS Hessian approximation matrix and arbitrary vector $\bbp$ to the computation of two vector sequences $\{\bbp_{u}\}_{u=0}^{\tau-1}$ and $\{\bbq_{u}\}_{u=0}^{\tau-1}$. The product $\hbB_{t}^{-1}\bbp\ =\ \bbq_{\tau}$ is given by the last element of the latter sequence. Since determination of each of the elements of each sequence requires $O(n)$ operations and the total number of elements in each sequence is $\tau$ the total operation cost to compute both sequences is of order $O(\tau n)$. In computing $\hbB_{t}^{-1}\bbp$ we also need to add the cost of the product $\bbq_0= \hbB_{t,0}^{-1} \bbp_{\tau}$ that links both sequences. To maintain overall computation cost of order $O(\tau n)$ this matrix has to have a sparse or low rank structure. A common choice in LBFGS, that we adopt for oLBFGS, is to make $\hbB_{t,0}^{-1} = \hat{\gamma}_{t}\bbI$. The scalar constant $\hat{\gamma}_t$ is a function of the variable and stochastic gradient variations $\bbv_{t-1}$ and $\hbr_{t-1}$, explicitly given by
\begin{equation}\label{initial_matrix_update_2}
   \hat{\gamma_{t}} 
      \ =\ \frac{\bbv_{t-1}^{T}\hbr_{t-1}}{\hbr_{t-1}^{T}\hbr_{t-1}}\, 
      \ =\ \frac{\bbv_{t-1}^{T}\hbr_{t-1}}{\|\hbr_{t-1}\|^2}\, .
\end{equation}
with the value at the first iteration being $\hat{\gamma}_{0}=1$. The scaling factor $\hat{\gamma}_{t}$ attempts to estimate one of the eigenvalues of the Hessian matrix at step $t$ and has been observed to work well in practice; see e.g., \cite{DingNocedal, Nocedal}. Further observe that the cost of computing $\hat{\gamma}_t$ is of order $O(n)$ and that since $\hbB_{t,0}^{-1}$ is diagonal cost of computing the product $\bbq_0= \hbB_{t,0}^{-1} \bbp_{\tau}$ is also of order $O(n)$. We adopt the initialization in \eqref{initial_matrix_update_2} in our subsequent analysis and numerical experiments.

%
\begin{algorithm}[t]{
\caption{Computation of oLBFGS step $\bbq = \hbB_{t}^{-1} \bbp$ when called with $\bbp=\hbs(\bbw_t,\tbtheta_t)$.}
\label{algo_lbfgs} 
\begin{algorithmic}[1]
   \STATE \textbf{function}  
          $\bbq=\bbq_\tau$ 
          = oLBFGS Step$\left(\hbB_{t,0}^{-1}, \
                              \bbp=\bbp_0, \
                              \{\bbv_{u},\hbr_{u}\}_{u=t-\tau}^{t-1}\right)$ 
   \FOR [Loop to compute constants $\alpha_u$ and sequence $\bbp_u$] 
        {$u= 0, 1, \ldots, \tau-1$ } 
      \STATE Compute and store scalar 
             $\alpha_{u}= \hat{\rho}_{t-u-1}\bbv_{t-u-1}^{T}\bbp_{u}$
      \STATE Update sequence vector 
             $\bbp_{u+1}=\bbp_u-\alpha_{u}\hbr_{t-u-1}$.
             [cf. \eqref{eqn_define_pt}]
   \ENDFOR
   \STATE Multiply $\bbp_\tau$ by initial matrix:  
          $\bbq_{0}= \hbB_{t,0}^{-1} \bbp_{\tau}  $
   \FOR [Loop to compute constants $\beta_u$ and sequence $\bbq_u$] 
        {$u= 0, 1, \ldots, \tau-1$}
      \STATE Compute scalar 
             $\beta_u= \hat{\rho}_{t-\tau+u}\hbr_{t-\tau+u}^{T}\bbq_{u}$ 
      \STATE Update sequence vector
             $\bbq_{u+1}=\bbq_{u}+(\alpha_{\tau-u-1}-\beta_{u})\bbv_{t-\tau+u}$
             [cf. \eqref{eqn_define_qt}]             
   \ENDFOR \ \{return $\bbq = \bbq_\tau$\} 
\end{algorithmic}}\end{algorithm}


The computation of the product $\hbB_{t}^{-1}\bbp$ using the result in Proposition \ref{oLBFGS_efficient_descent_dircetion_computation} is summarized in algorithmic form in the function in Algorithm \ref{algo_lbfgs}. The function receives as arguments the initial matrix $\hbB_{t,0}^{-1},$ the sequence of variable and stochastic gradient variations $\{\bbv_{u},\hbr_{u}\}_{u=t-\tau}^{t-1}$ and the vector $\bbp$ to produce the outcome $\bbq=\bbq_\tau=\hbB_{t}^{-1}\bbp$. When called with the stochastic gradient $\bbp=\hbs(\bbw_t,\tbtheta_t)$, the function outputs the oLBFGS step $\hbd_{t}:=\hbB_t^{-1}\hbs(\bbw_t,\tbtheta_t)$ needed to implement the oLBFGS descent step in \eqref{eqn_olbfgs_descent}. The core of Algorithm \ref{algo_lbfgs} is given by the loop in steps 2-5 that computes the constants $\alpha_u$ and sequence elements $\bbp_u$ as well as the loop in steps 7-10 that computes the constants $\beta_u$ and sequence elements $\bbq_u$. The two loops are linked by the initialization of the second sequence with the outcome of the first which is performed in Step 6. To implement the first loop we require $\tau$ inner products in Step 4 and $\tau$ vector summations in Step 5 which yield a total of $2\tau n$ multiplications. Likewise, the second loop requires $\tau$ inner products and $\tau $ vector summations in steps 9 and 10, respectively, which yields a total cost of also $2\tau n$ multiplications. Since the initial Hessian inverse approximation matrix $\hbB_{t,0}^{-1}$ is diagonal the cost of computation $\hbB_{t,0}^{-1} \bbp_{\tau} $ in Step 6 is $n$ multiplications. Thus, Algorithm \ref{algo_lbfgs} requires a total of $(4\tau+1)n$ multiplications which affirms the complexity cost of order $O(\tau n)$ for oLBFGS.

For reference, oLBFGS is also summarized in algorithmic form in Algorithm \ref{algo_stochastic_bfgs}. As with any stochastic descent algorithm the descent iteration is implemented in three steps: the acquisition of $L$ samples in Step 2, the computation of the stochastic gradient in Step 3, and the implementation of the descent update on the variable $\bbw_t$ in Step 6. Steps 4 and 5 are devoted to the computation of the oLBFGS descent direction $\hbd_{t}$. In Step 4 we initialize the estimate $\hbB_{t,0}= \hat{\gamma_{t}}\bbI$ as a scaled identity matrix using the expression for $\hat{\gamma_{t}}$ in \eqref{initial_matrix_update_2} for $t>0$. The value of $\gamma_t=\gamma_0$ for $t=0$ is left as an input for the algorithm. We use $\hat{\gamma_{0}}=1$ in our numerical tests. In Step 5 we use Algorithm 1 for efficient computation of the descent direction $\hbd_{t}=\hbB_{t}^{-1}\hbs(\bbw_{t},\tbtheta_{t})$. Step 7 determines the value of the stochastic gradient $\hbs(\bbw_{t+1},\tbtheta_{t})$ so that the variable variations $\bbv_{t}$ and stochastic gradient variations $\hbr_{t}$ become available for the computation of the curvature approximation matrix $\hbB_{t}^{-1}$. In Step 8 the variable variation  $\bbv_{t}$ and stochastic gradient variation $\hbr_{t}$ are computed to be used in the next iteration. We analyze convergence properties of this algorithm in Section \ref{sec_convergence}, study its application to SVMs in Section \ref{sec:SVMproblem}, and develop an application to search engine advertisement in Section \ref{sec:ad_application}.

%
{\begin{algorithm}[tb]
\caption{oLBFGS}\label{algo_stochastic_bfgs} 
\begin{algorithmic}[1] 
\small{\REQUIRE  Initial value $\bbw_0$. Initial Hessian approximation parameter $\hat{\gamma}_{0}=1$.
\FOR {$t=0,1,2,\ldots$}
   \STATE Acquire $L$ independent samples $\tbtheta_t=[\bbtheta_{t1},\ldots,\bbtheta_{tL}]$
   \STATE Compute stochastic gradient:        
          $\displaystyle{
               \hbs(\bbw_{t},\tbtheta_{t}) 
                   = \frac{1}{L}\sum_{l=1}^{L}  \nabla_{\bbw} f(\bbw_{t},{\bbtheta_{tl}})
          }$ [cf. \eqref{stochastic_gradient}]   
   \STATE Initialize Hessian inverse estimate as 
          $\hbB_{t,0}^{-1} = \hat{\gamma_{t}}\bbI$ with 
          $\displaystyle{   
               \hat{\gamma_{t}}
               = \frac{\bbv_{t-1}^{T}\hbr_{t-1}}{\hbr_{t-1}^{T}\hbr_{t-1}}
          }$ for $t>0$ [cf \eqref{initial_matrix_update_2}]   
     \STATE Compute descent direction with Algorithm \ref{algo_lbfgs}:  
           $\displaystyle{\hbd_{t} = 
            \text{oLBFGS Step}\left(\hbB_{t,0}^{-1}, \   
                               \hbs(\bbw_{t},\tbtheta_{t}) , \
                               \{\bbv_{u},\hbr_{u}\}_{u=t-\tau}^{t-1}\right)}$    
   \STATE Descend along direction $\hbd_{t}$:
          $\displaystyle{   
               \bbw_{t+1} = \bbw_{t}-\epsilon_{t} \hbd_{t}
          }$ [cf. \eqref{eqn_olbfgs_descent}]   
   \STATE Compute stochastic gradient:       
          $\displaystyle{
               \hbs(\bbw_{t+1},\tbtheta_{t}) 
                   = \frac{1}{L}\sum_{l=1}^{L}  \nabla_{\bbw} f(\bbw_{t+1},{\bbtheta_{tl}})
          }$ [cf. \eqref{stochastic_gradient}]
   \STATE Variations    
          $\displaystyle{\bbv_{t}= \bbw_{t+1}-\bbw_{t}}$ 
          [variable, cf. \eqref{ball}]           
          $\displaystyle{\hbr_{t}=\hbs(\bbw_{t+1},\tbtheta_{t})-\hbs(\bbw_{t},\tbtheta_{t})}$ 
          [stoch. gradient, cf.\eqref{chris}]
\ENDFOR}
\end{algorithmic}\end{algorithm}
%


\section{Convergence analysis}\label{sec_convergence}

For the subsequent analysis it is convenient to define the instantaneous objective function associated with samples $\tbtheta=[\bbtheta_1,\ldots,\bbtheta_L]$ as 
\begin{equation}\label{eqn_isntantaneous_dual_function}
   \hhatf(\bbw,{\tbtheta}) := \frac{1}{L}\sum_{l=1}^L f(\bbw,{\bbtheta_{l}}).
\end{equation}
The definition of the instantaneous objective function $ \hhatf(\bbw,{\tbtheta})$ in association with  the fact that $F(\bbw) := \mbE_{\bbtheta}[f(\bbw,{\bbtheta})]$ implies that
\begin{equation}\label{eqn_dual_function_as_expectation_of_isntantaneous_dual_function}
   F(\bbw) = \mbE_{\bbtheta}[\hhatf(\bbw,{\tbtheta})].
\end{equation}
Our goal here is to show that as time progresses the sequence of variable iterates $\bbw_t$ approaches the optimal argument $\bbw^*$. In proving this result we make the following assumptions.

%
\begin{assumption}\label{ass_intantaneous_hessian_bounds}\normalfont  The instantaneous functions $\hhatf(\bbw,{\tbtheta})$ are twice differentiable and the eigenvalues of the instantaneous Hessian $\hat{ \bbH}(\bbw,{\tbtheta} )=\nabla_{\bbw}^2\hhatf(\bbw,{\tbtheta})$ are  bounded between constants $0<\tdm$ and $\tdM<\infty$ for all random variables $\tbtheta$,
\begin{equation}\label{hassan}
   \tdm\bbI \ \preceq\ \hat{ \bbH}(\bbw,{\tbtheta}) \ \preceq \ \tdM \bbI.
\end{equation} \end{assumption}

%
\begin{assumption}\normalfont\label{ass_bounded_stochastic_gradient_norm} The second moment of the norm of the stochastic gradient is bounded for all $\bbw$. i.e., there exists a constant $S^2$ such that for all variables $\bbw$ it holds
\begin{equation}\label{ekhtelaf}
   \mbE_{\bbtheta} \big{[} \| \hbs(\bbw_{t},\tbtheta_{t})\|^{2} \given \bbw_{t} \big{]} \leq S^{2}.
\end{equation} \end{assumption}

\begin{assumption}\normalfont\label{stepsize_ass}
The step size sequence is selected as nonsummable but square summable, i.e., 
\begin{equation}\label{stepsize_condition}
   {\sum_{t=0}^{\infty} \eps_t = \infty, \quad \text{and} \quad
          \sum_{t=0}^{\infty} \eps_t^2 < \infty }.
\end{equation}
\end{assumption}

%
Assumptions \ref{ass_bounded_stochastic_gradient_norm} and \ref{stepsize_ass} are customary in stochastic optimization. The restriction imposed by Assumption \ref{ass_bounded_stochastic_gradient_norm} is intended to limit the random variation of stochastic gradients. If the variance of their norm is unbounded it is possible to have rare events that derail progress towards convergence. The condition in Assumption \ref{stepsize_ass} balances descent towards optimal arguments -- which requires a slowly decreasing stepsize -- with the eventual elimination of random variations -- which requires rapidly decreasing stepsizes. An effective step size choice for which Assumption \ref{stepsize_ass} holds is to make $\eps_{t}= \eps_{0} T_{0}/(T_{0}+t)$, for given parameters $\eps_{0}$ and $T_{0}$ that control the initial step size and its speed of decrease, respectively. Assumption \ref{ass_intantaneous_hessian_bounds} is stronger than usual and specific to oLBFGS. Observe that considering the linearity of the expectation operator and the expression in \eqref{eqn_dual_function_as_expectation_of_isntantaneous_dual_function} it follows  that the Hessian of the average function can be written as $\nabla_{\bbw}^{2}F(\bbw)= \bbH(\bbw)=\mbE_{\bbtheta}[\hat{\bbH}(\bbw,{\tbtheta})]$. Combining this observation with the bounds in \eqref{hassan} we conclude that there are constants $m\geq\tdm$ and $M\leq\tdM$ such that
\begin{equation}\label{bbb}
   \tdm\bbI \ \preceq\ m\bbI\ \preceq\ \bbH(\bbw) \ \preceq  M\bbI \ \preceq \ \tdM\bbI.
\end{equation} 
The bounds in \eqref{bbb} are customary in convergence proofs of descent methods. For the results here the stronger condition spelled in Assumption \ref{ass_intantaneous_hessian_bounds} is needed. This assumption in necessary to guarantee that the inner product $\hbr_{t}^{T}\bbv_{t}>0$ is positive as we show in the following lemma.

%
\begin{lemma}\label{lecce}
Consider the stochastic gradient variation $\hbr_{t}$ defined in \eqref{chris} and the variable variation $\bbv_{t}$ defined in \eqref{ball}. Let Assumption \ref{ass_intantaneous_hessian_bounds} hold so that we have lower and upper bounds $\tdm$ and $\tdM$ on the eigenvalues of the instantaneous Hessians. Then, for all steps $t$ the inner product of variable and stochastic gradient variations $ \hbr_{t}^{T}\bbv_{t}$ is bounded below as
\begin{equation}\label{claim233}
\tdm\|\bbv_{t}\|^2    \leq\    \hbr_{t}^{T}\bbv_{t} \  .
\end{equation}
Furthermore, the ratio of stochastic gradient variation squared norm $\|\hbr_{t}\|^2=\hbr_{t}^{T}\hbr_{t}$ to inner product of variable and stochastic gradient variations is bounded as
\begin{equation}\label{claim444}
   \tdm \ \leq\ \frac{\hbr_{t}^{T}\hbr_{t}}{\hbr_{t}^{T}\bbv_{t}} 
        \ =   \ \frac{\|\hbr_{t}\|^2}{\hbr_{t}^{T}\bbv_{t}} 
        \ \leq\ \tdM . 
\end{equation}
\end{lemma}

\begin{proof} See Appendix \ref{appx_proof_lemma_1}. \end{proof}

%
According to Lemma \ref{lecce}, strong convexity of instantaneous functions $\hhatf(\bbw,{\tbtheta})$ guaranties positiveness of the inner product $\bbv_{t}^{T}\hbr_{t}$ as long as the variable variation is not identically null. In turn, this implies that the constant $\hat{\gamma}_{t}$ in \eqref{initial_matrix_update_2} is nonnegative and that, as a consequence, the initial Hessian inverse approximation ${\hbB_{t,0}}^{-1}$ is positive definite for all steps $t$. The positive definiteness of ${\hbB_{t,0}}^{-1}$ in association with the positiveness of the inner product of variable and stochastic gradient variations $\bbv_{t}^{T}\hbr_{t}>0$ further guarantees that all the matrices $\hbB_{t,u+1}^{-1}$, including the matrix $\hbB_{t}^{-1}=\hbB_{t,\tau}^{-1}$ in particular, that follow the update rule in \eqref{oLBFGS_update} stay positive definite -- see \cite{AryanAleTSP} for details. This proves that \eqref{eqn_olbfgs_descent} is a proper stochastic descent iteration because the stochastic gradient $\hbs(\bbw_t,\tbtheta_t)$ is moderated by a positive definite matrix. However, this fact alone is not enough to guarantee convergence because the minimum and maximum eigenvalues of $\hbB_{t}^{-1}$ could become arbitrarily small and arbitrarily large, respectively. To prove convergence we show this is not possible by deriving explicit lower and upper bounds on these eigenvalues.

The analysis is easier if we consider the matrix $\hbB_{t}$ -- as opposed to $\hbB_{t}^{-1}$. Consider then the update in \eqref{oLBFGS_update}, and use the Sherman-Morrison formula to rewrite as an update that relates $\hbB_{t,u+1}$ to $\hbB_{t,u}$,
\begin{equation}\label{Hessian_appro_update}          
  					  \hbB_{t,u+1} \ =  \  \hbB_{t,u}
	- {{ \hbB_{t,u}\bbv_{t-\tau+u}\bbv_{t-\tau+u}^{T}{\hbB_{t,u}}}\over{\bbv_{t-\tau+u}^{T}\hbB_{t,u}\bbv_{t-\tau+u}}} 
	+ 			  		{{\hbr_{t-\tau+u}\hbr_{t-\tau+u}^{T}}\over{\bbv_{t-\!\tau\!+u}^{T}\hbr_{t-\tau+u}}},
          \end{equation}
for $u=0, \dots, \tau-1$ and $\hbB_{t,0}= 1/\hat{\gamma}_{t} \bbI$ as per \eqref{initial_matrix_update_2}. As in \eqref{oLBFGS_update}, the Hessian approximation at step $t$ is $\hbB_{t}=\hbB_{t,\tau}$. In the following lemma we use the update formula in \eqref{Hessian_appro_update} to find bounds on the trace and determinant of the Hessian approximation $\hbB_{t}$.

%
\begin{lemma}\label{lemma_determinant_and_trace_bounds}
Consider the Hessian approximation $\hbB_{t} = \hbB_{t,\tau}$ defined by the recursion in \eqref{Hessian_appro_update} with $\hbB_{t,0} = \hat{\gamma}_{t}^{-1}\bbI$ and $\hat{\gamma}_{t}$ as given by \eqref{initial_matrix_update_2}. If Assumption \ref{ass_intantaneous_hessian_bounds} holds true, the trace $\tr(\hbB_{t})$ of the Hessian approximation $\hbB_{t}$ is uniformly upper bounded for all times $t\geq1$,
\begin{equation}\label{trace_bound_1}
    \tr\left(\hbB_{t}\right)\ \leq\  (n + \tau) \tdM.
\end{equation}
Likewise, if Assumption \ref{ass_intantaneous_hessian_bounds} holds true, the determinant $\det(\hbB_{t})$ of the Hessian approximation $\hbB_{t}$ is uniformly lower bounded for all times $t$
\begin{equation}\label{det_bound_1}
\det\left(\hbB_{t}\right)\  \geq\  \frac{\tdm^{n+\tau}}{[( n+ \tau) \tdM]^{\tau}} \ .
\end{equation}\end{lemma}

\begin{proof} See Appendix \ref{appx_lemma3}. \end{proof}

%
Lemma \ref{lemma_determinant_and_trace_bounds} states that the trace and determinants of the Hessian approximation matrix $\hbB_{t} = \hbB_{t,\tau}$ are bounded for all times $t\geq1$. For time $t=0$ we can write a similar bound that takes into account the fact that the constant $\gamma_t$ that initializes the recursion in \eqref{Hessian_appro_update} is $\gamma_0=1$. Given that we are interested in an asymptotic convergence analysis, this bound in inconsequential. The bounds on the trace and determinant of $\hbB_{t}$ are respectivey equivalent to bounds in the sum and product of its eigenvalues. Further considering that the matrix $\hbB_{t}$ is positive definite, as it follows from Lemma \ref{lecce}, these bounds can be further transformed into bounds on the smalls and largest eigenvalue of $\hbB_t$. The resulting bounds are formally stated in the following lemma.

%
\begin{lemma}\label{norooz}
Consider the Hessian approximation $\hbB_{t} = \hbB_{t,\tau}$ defined by the recursion in \eqref{Hessian_appro_update} with $\hbB_{t,0} = \hat{\gamma}_{t}^{-1}\bbI$ and $\hat{\gamma}_{t}$ as given by \eqref{initial_matrix_update_2}. Define the strictly positive constant $0<c:= \tdm^{n+\tau}/[(n + \tau )\tdM]^{n+\tau-1}$ and the finite constant $C:=(n + \tau) \tdM<\infty$. If Assumption \ref{ass_intantaneous_hessian_bounds} holds true, the range of eigenvalues of $\hbB_{t}$ is bounded by $c$ and $C$ for all time steps $t\geq1$, i.e., 
\begin{equation}\label{claim888}
     \frac{\tdm^{n+\tau}} {{[(n + \tau )\tdM]}^{n+\tau-1}}\, \bbI
        \ =:     \  c \bbI 
        \ \preceq\  \hbB_{t}\ 
        \ \preceq\  C\bbI
        \ : =    \  (n + \tau) \tdM\, \bbI.
\end{equation} \end{lemma}

\begin{proof}  See Appendix \ref{appx_lemma4}. \end{proof}

%
The bounds in Lemma \ref{norooz} imply that their respective inverses are bounds on the range of the eigenvalues of the Hessian inverse approximation matrix $\hbB_{t}^{-1}$. Specifically, the minimum eigenvalue of the Hessian inverse approximation $\hbB_{t}^{-1}$ is larger than $1/C$ and the maximum eigenvalue of  $\hbB_{t}^{-1}$ does not exceed $1/c$, or, equivalently,
\begin{equation}\label{eqn_eigenvalue_critical_bounds}
    \frac{1}{C}\, \bbI\ \preceq\  \hbB_{t}^{-1}\ \preceq\ \frac{1}{c}\,\bbI\,.
\end{equation}
We further emphasize that the bounds in \eqref{eqn_eigenvalue_critical_bounds}, or \eqref{claim888} for that matter, limit the conditioning of $\hbB_t^{-1}$ for all realizations of the random samples $\{\tbtheta_t\}_{t=0}^\infty$, irrespective of the particular random draw. Having matrices $\hbB_t^{-1}$ that are strictly positive definite with eigenvalues uniformly upper bounded by $1/c$ leads to the conclusion that if $\hbs(\bbw_{t},\tbtheta_{t})$ is a descent direction, the same holds true of $\hbB_t^{-1}\ \! \hbs(\bbw_{t},\tbtheta_{t})$. The stochastic gradient $\hbs(\bbw_{t},\tbtheta_{t})$ is not a descent direction in general, but we know that this is true for its conditional expectation $\mbE[\hbs(\bbw_{t},\tbtheta_{t}) \given \bbw_{t}] = \nabla F(\bbw_{t})$. Hence, we conclude that $\hbB_t^{-1} \hbs(\bbw_{t},\tbtheta_{t})$ is an average descent direction since $\mbE[\hbB_t^{-1}\hbs(\bbw_{t},\tbtheta_{t})\! \given \! \bbw_{t}] = \hbB_t^{-1}\nabla F(\bbw_{t})$. Stochastic optimization methods whose displacements $\bbw_{t+1}-\bbw_t$ are descent directions on average are expected to approach optimal arguments. We show that this is true of oLBFGS in the following lemma.

%
\begin{lemma}\label{helpful}
Consider the online Limited memory BFGS algorithm as defined by the descent iteration in \eqref {eqn_olbfgs_descent} with matrices $\hbB_t^{-1}= \hbB_{t,\tau}^{-1}$ obtained after $\tau$ recursive applications of the update in \eqref{oLBFGS_update} initialized with $\hbB_{t,0}^{-1} = \hat{\gamma}_{t}\bbI$ and $\hat{\gamma}_{t}$ as given by \eqref{initial_matrix_update_2}. If Assumptions \ref{ass_intantaneous_hessian_bounds} and \ref{ass_bounded_stochastic_gradient_norm} hold true, the sequence of average function values $F(\bbw_{t})$ satisfies
\begin{equation}\label{pedarsag}
  \!  \E{F(\bbw_{t+1})\given \bbw_{t}}  
         \leq F(\bbw_{t})
          -  \frac{\epsilon_{t}}{C}\| \nabla F(\bbw_{t})\|^{2} 
          +\frac{MS^2\epsilon_{t}^{2}}{2c^2}.
\end{equation}
%
\end{lemma}

\begin{proof}  See Appendix \ref{appx_lemma5}. \end{proof}

%
Setting aside the term $MS^2\epsilon_{t}^{2}/2c^2$ for the sake of argument, \eqref{pedarsag} defines a supermartingale relationship for the sequence of average functions $F(\bbw_{t})$. This implies that the sequence $\eps_t\| \nabla F(\bbw_{t})\|^{2}/C$ is almost surely summable which, given that the step sizes $\eps_t$ are nonsummable as per \eqref{stepsize_condition}, further implies that the limit infimum $\liminf_{t\to\infty}\|\nabla F(\bbw_{t})\|$ of the gradient norm $\|\nabla F(\bbw_{t})\|$ is almost surely null. This latter observation is equivalent to having $\liminf_{t\to\infty}\| \bbw_{t}-\bbw^{*} \|^{2} =0$ with probability 1 over realizations of the random samples $\{\tbtheta_t\}_{t=0}^\infty$. The term $MS^2\epsilon_{t}^{2}/2c^2$ is a relatively minor nuisance that can be taken care of with a technical argument that we present in the proof of the following theorem.

%
\begin{theorem}\label{convg}
Consider the online Limited memory BFGS algorithm as defined by the descent iteration in \eqref {eqn_olbfgs_descent} with matrices $\hbB_t^{-1}= \hbB_{t,\tau}^{-1}$ obtained after $\tau$ recursive applications of the update in \eqref{oLBFGS_update} initialized with $\hbB_{t,0}^{-1} = \hat{\gamma}_{t}\bbI$ and $\hat{\gamma}_{t}$ as given by \eqref{initial_matrix_update_2}. If Assumptions \ref{ass_intantaneous_hessian_bounds}-\ref{stepsize_ass} hold true the limit infimum of the squared Euclidean distance to optimality $\| \bbw_{t}-\bbw^{*} \|^{2}$ converges to zero almost surely, i.e.,
\begin{equation}\label{eqn_convg}
   \Pr{\liminf_{t \to \infty }\| \bbw_{t}-\bbw^{*} \|^{2} = 0} =1,
\end{equation} 
where the probability is over realizations of the random samples $\{\tbtheta_t\}_{t=0}^\infty$.\end{theorem}
}
\begin{proof} See Appendix \ref{appx_theorem_6}. \end{proof}
%

Theorem \ref{convg} establishes convergence of a subsequence of the oLBFGS algorithm summarized in Algorithm \ref{algo_stochastic_bfgs}. The lower and upper bounds on the eigenvalues of $\hbB_t$ derived in Lemma \ref{norooz} play a fundamental role in the proofs of the prerequisite Lemma \ref{helpful} and Theorem \ref{convg} proper. Roughly speaking, the lower bound on the eigenvalues of $\hbB_t$ results in an upper bound on the eigenvalues of $\hbB_t^{-1}$ which limits the effect of random variations on the stochastic gradient $\hbs(\bbw_{t},\tbtheta_{t})$. If this bound does not exist -- as is the case, e.g., of regular stochastic BFGS -- we may observe catastrophic amplification of random variations of the stochastic gradient. The upper bound on the eigenvalues of  $\hbB_t$, which results in a lower bound on the eigenvalues of $\hbB_t^{-1}$, guarantees that the random variations in the curvature estimate $\hbB_t$ do not yield matrices with arbitrarily small norm. If this bound does not hold, it is possible to end up halting progress before convergence as the stochastic gradient is nullified by multiplication with an arbitrarily small eigenvalue.

The result in Theorem \ref{convg} is strong because it holds almost surely over realizations of the random samples $\{\tbtheta_t\}_{t=0}^\infty$ but not stronger than the same convergence guarantees that hold for SGD. We complement the convergence result in Theorem \ref{convg} with a characterization of the expected convergence rate that we introduce in the following theorem.

\begin{theorem}\label{theo_convergence_rate}
Consider the online Limited memory BFGS algorithm as defined by the descent iteration in \eqref {eqn_olbfgs_descent} with matrices $\hbB_t^{-1}= \hbB_{t,\tau}^{-1}$ obtained after $\tau$ recursive applications of the update in \eqref{oLBFGS_update} initialized with $\hbB_{t,0}^{-1} = \hat{\gamma}_{t}\bbI$ and $\hat{\gamma}_{t}$ as given by \eqref{initial_matrix_update_2}. Let Assumptions \ref{ass_intantaneous_hessian_bounds} and \ref{ass_bounded_stochastic_gradient_norm} hold, and further assume that the stepsize sequence is of the form $\eps_t = \eps_0/(t+T_0)$ with the parameters $\eps_0$ and $T_{0}$ satisfying the inequality $2m\eps_{0} T_{0} /C >1$. Then, the difference between the expected optimal objective $\E {F(\bbw_{t})}$ and the optimal objective $F(\bbw^*)$ is bounded as
\begin{equation}\label{eqn_thm_cvg_rate_20}
\E {F(\bbw_{t})}- F(\bbw^*)\ \leq\ \frac{C_{0}}{T_{0}+t}\ ,
\end{equation}
where the constant $C_{0}$ is defined as
\begin{equation}\label{eqn_thm_cvg_rate_30}
 C_{0} := \max \left\{ \frac{\epsilon_{0}^{2}\   T_{0}^{2} C {MS^2}}{{2c^2}(2m \epsilon_{0} T_{0}  - C)}\ ,\ T_{0}\ \!  (F(\bbw_{0}) -\ F(\bbw^*))  \right\} .
\end{equation}
\end{theorem}

\begin{proof} See Appendix \ref{appx_theorem_7}. \end{proof}

Theorem \ref{theo_convergence_rate} shows that under specified assumptions the expected error  in terms of the objective value after $t$ oLBFGS iterations is of order $O(1/t)$. As is the case of Theorem \ref{convg}, this result is not better than the convergence rate of conventional SGD. As can be seen in the proof of Theorem \ref{theo_convergence_rate}, the convergence rate is dominated by the noise term introduced by the difference between stochastic and regular gradients. This noise term would be present even if exact Hessians were available and in that sense the best that can be proven of oLBFGS is that the convergence rate is not worse than that of SGD. Given that theorems \ref{convg} and \ref{theo_convergence_rate} parallel the theoretical guarantees of SGD it is perhaps fairer to describe oLBFGS as an adaptive reconditioning strategy instead of a stochastic quasi-Newton method. The latter description refers to the genesis of the algorithm, but the former is more accurate description of its behavior. Do notice that while the convergence rate doesn't change, improvements in convergence time are significant as we illustrate with the numerical experiments that we present in the next two sections.


\section{Support vector machines} \label{sec:SVMproblem}

Given a training set with points whose classes are known the goal of an SVM is to find a hyperplane that best separates the training set. Let $\ccalS = \{ (\bbx_{i},y_{i}) \}_{i=1}^{N}$ be a training set containing $N$ pairs of the form $(\bbx_{i},y_i)$, where $\bbx_{i}\in\reals^n$ is a feature vector and $y_{i}\in \{-1,1 \}$ is the corresponding class. The goal is to find a hyperplane supported by a vector $\bbw\in\reals^n$ which separates the training set so that $\bbw^T\bbx_i>0$ for all points with $y_i=1$ and $\bbw^T\bbx_i<0$ for all points with $y_i=-1$. A loss function $l((\bbx,y);\bbw)$ defines a measure of distance between the point $\bbx_i$ and the hyperplane supported by $\bbw$. We then select the hyperplane supporting vector as
\begin{equation}\label{SVM}\vspace{-3pt}
   \bbw^* := \argmin_{\bbw}\   \frac{\lambda}{2}\|\bbw\|^2 
                      + \frac{1}{N} \sum_{i=1}^{N} l((\bbx_{i},y_{i});\bbw),
\end{equation}
where we have also added the regularization term ${\lambda}\|\bbw\|^2 /{2} $ for some constant $\lambda>0$. Common selections for the loss function are the hinge loss $l((\bbx,y);\bbw)=\max(0,1-y(\bbw^{T}\bbx))$ and the squared hinge loss $l((\bbx,y);\bbw)=\max(0,1-y(\bbw^{T}\bbx))^{2}$. See, e.g., \cite{Bottou}. To model \eqref{SVM} as a problem in the form of \eqref{optimization_problem}, define $\bbtheta_{i}=(\bbx_{i},y_{i})$ as a given training point and the probability distribution of $\theta$ as uniform on the training set $\ccalS = \{ (\bbx_{i},y_{i}) \}_{i=1}^{N}= \{ \bbtheta_{i} \}_{i=1}^{N}$. It then suffices to define 
\begin{equation}\label{eqn_svn_reformulation_random_functions}
f(\bbw,\bbtheta)  =\ f(\bbw,(\bbx,y))\ 
                 :=\ \frac{\lambda}{2}\|\bbw\|^2 + l((\bbx,y);\bbw),
\end{equation}
as sample functions to see that the objective in \eqref{SVM} can be written as the average $F(\bbw) = \mbE_{\bbtheta}[f(\bbw,{\bbtheta})]$ as in \eqref{optimization_problem}. We can then use SGD, oBFGS, RES, and oLBFGS to find the optimal classifier $\bbw^*$. There are also several algorithms that accelerate SGD through the use of memory. These algorithms reduce execution times because they reduce randomness, not because they improve curvature, but are nonetheless alternatives to oLBFGS. We further add Stochastic Average Gradient (SAG) to the comparison set. SAG is a variant of SGD that uses an average of stochastic gradients as a descent direction (\cite{Bach}). The performances of other SGD algorithms with memory are similar to SAG. For these five algorithms we want to compare achieved objective values with respect to the number of feature vectors processed (Section \ref{sec:SVM2}) as well as with respect to processing times (Section \ref{sec:SVM3}).

%

\begin{figure}[t]{
   \subfigure[oLBFGS]{\label{fig:obfgs_100}%
   \includegraphics[width=0.49\linewidth]{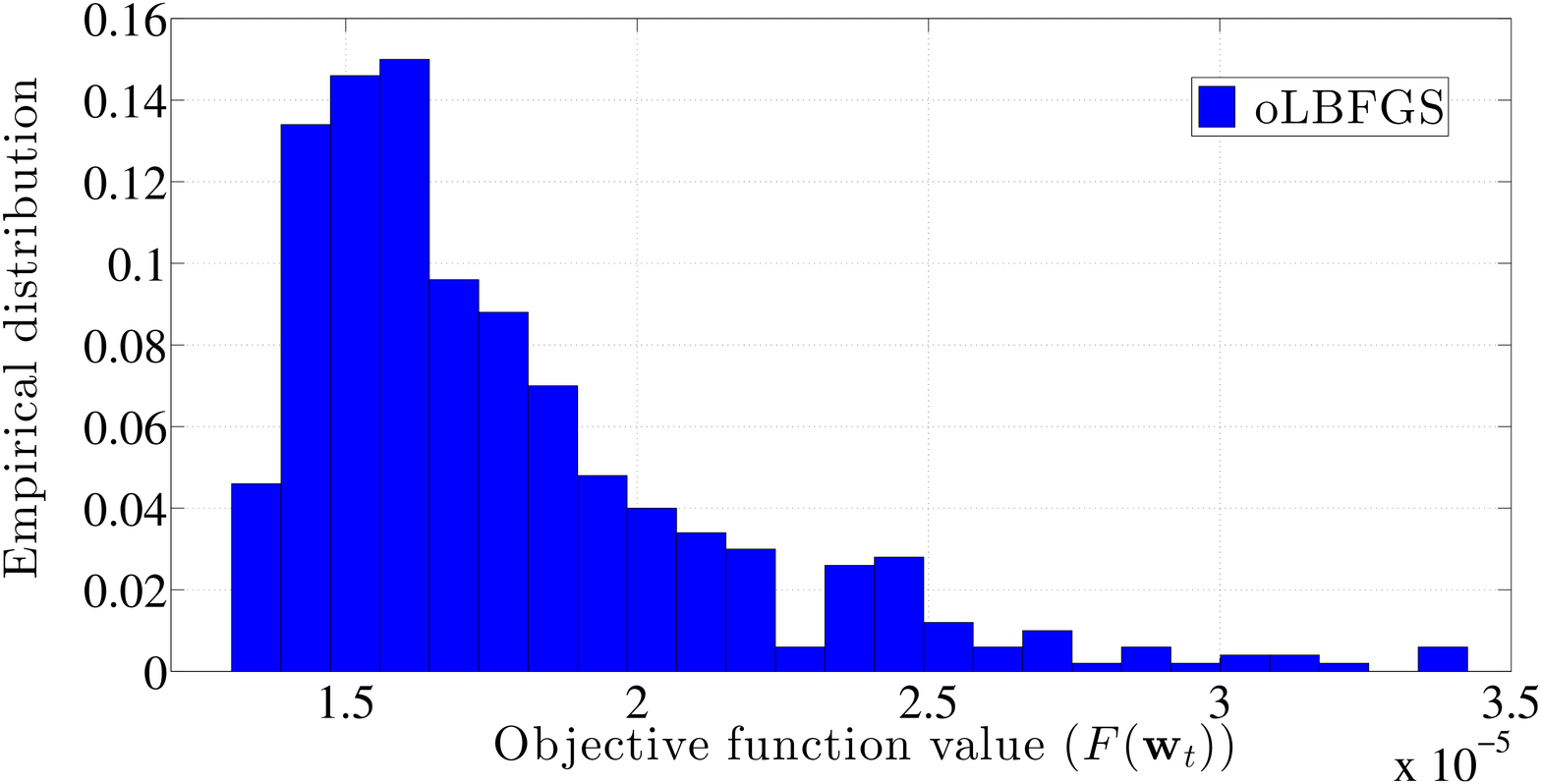}}
   \subfigure[oBFGS]{\label{fig:olbfgs_100}%
   \includegraphics[width=0.49\linewidth]{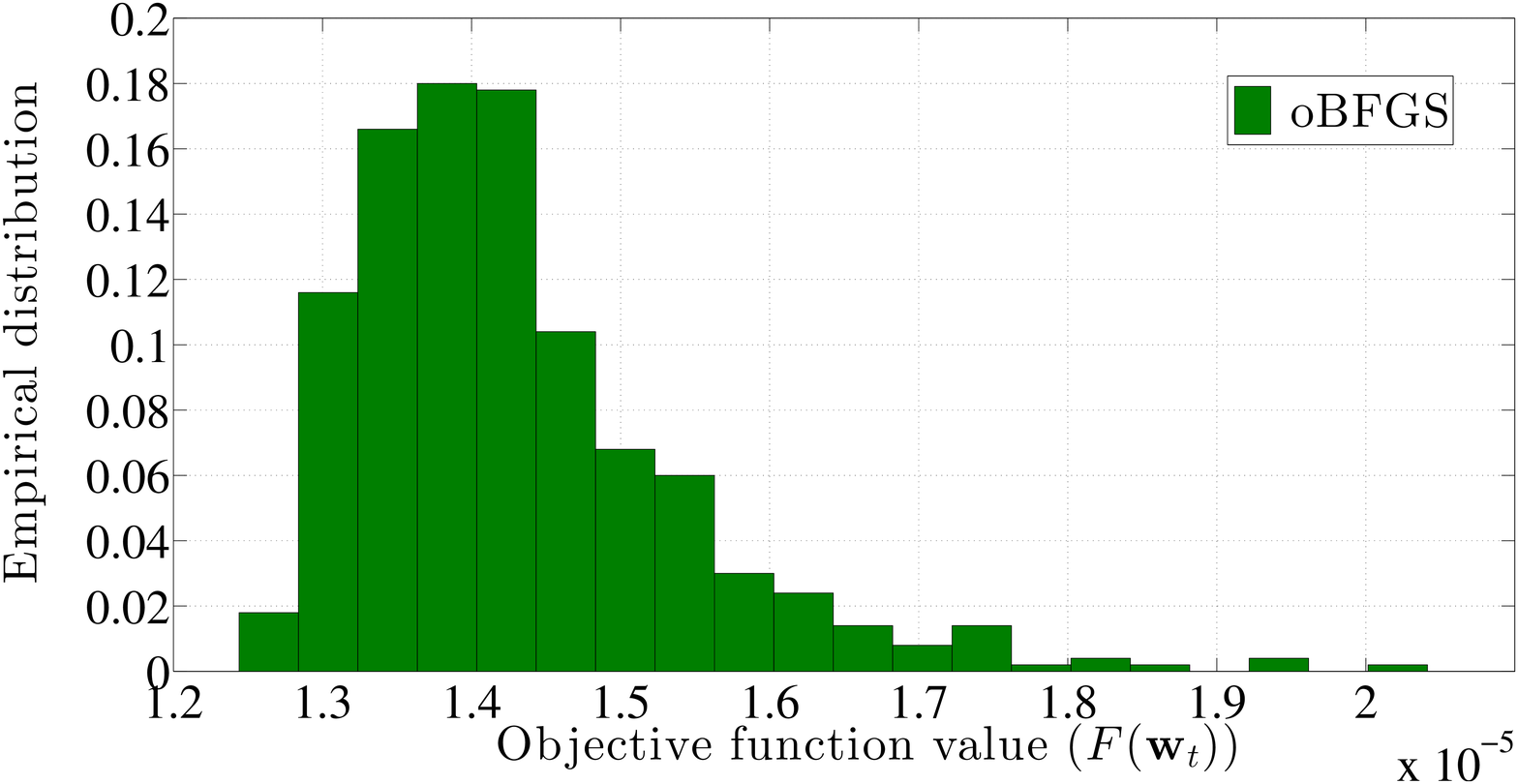}}
   \subfigure[RES]{\label{fig:res_100}%
   \includegraphics[width=0.49\linewidth]{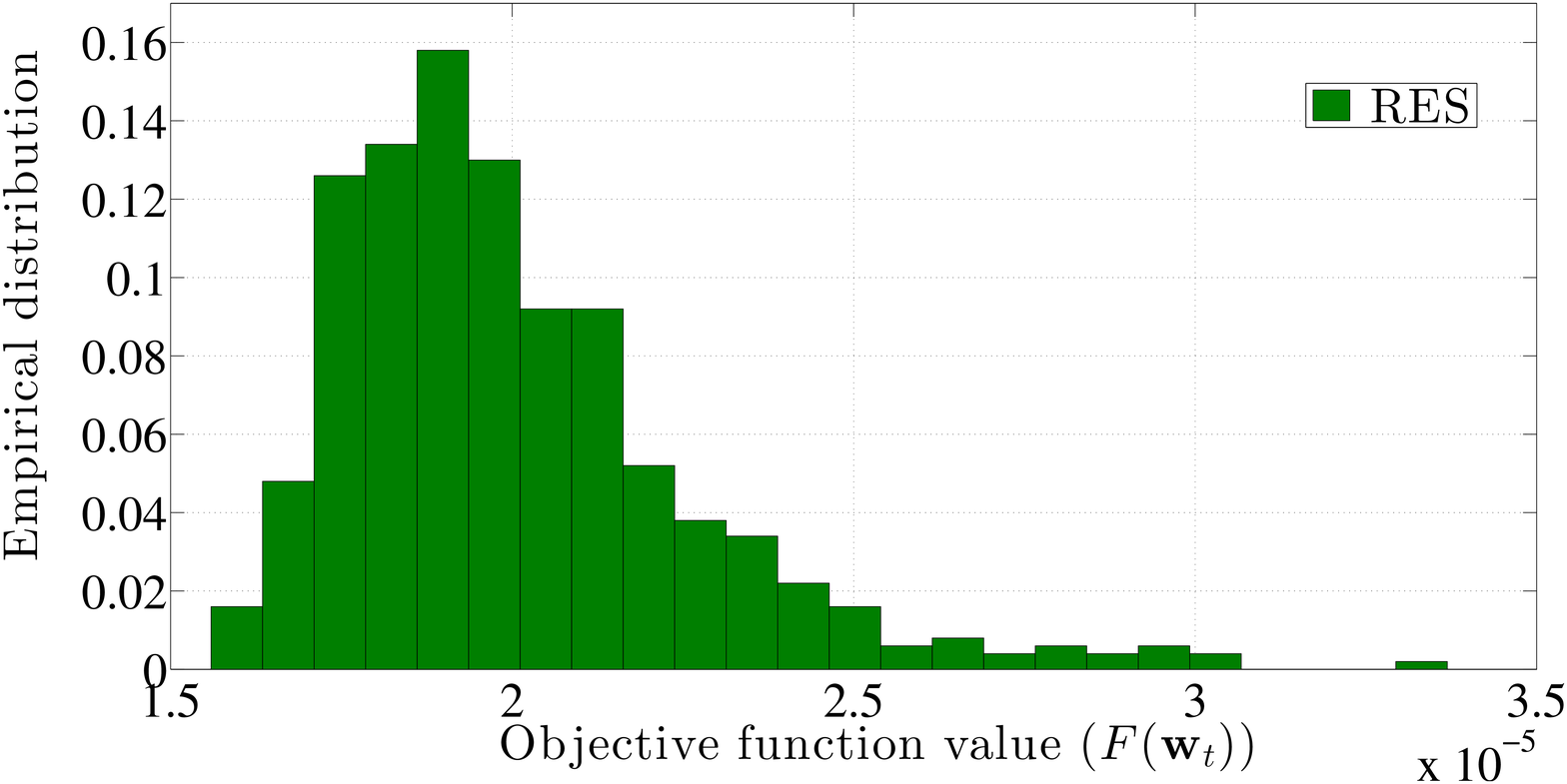}}%
   \subfigure[SGD]{\label{fig:sgd_100}%
   \includegraphics[width=0.49\linewidth]{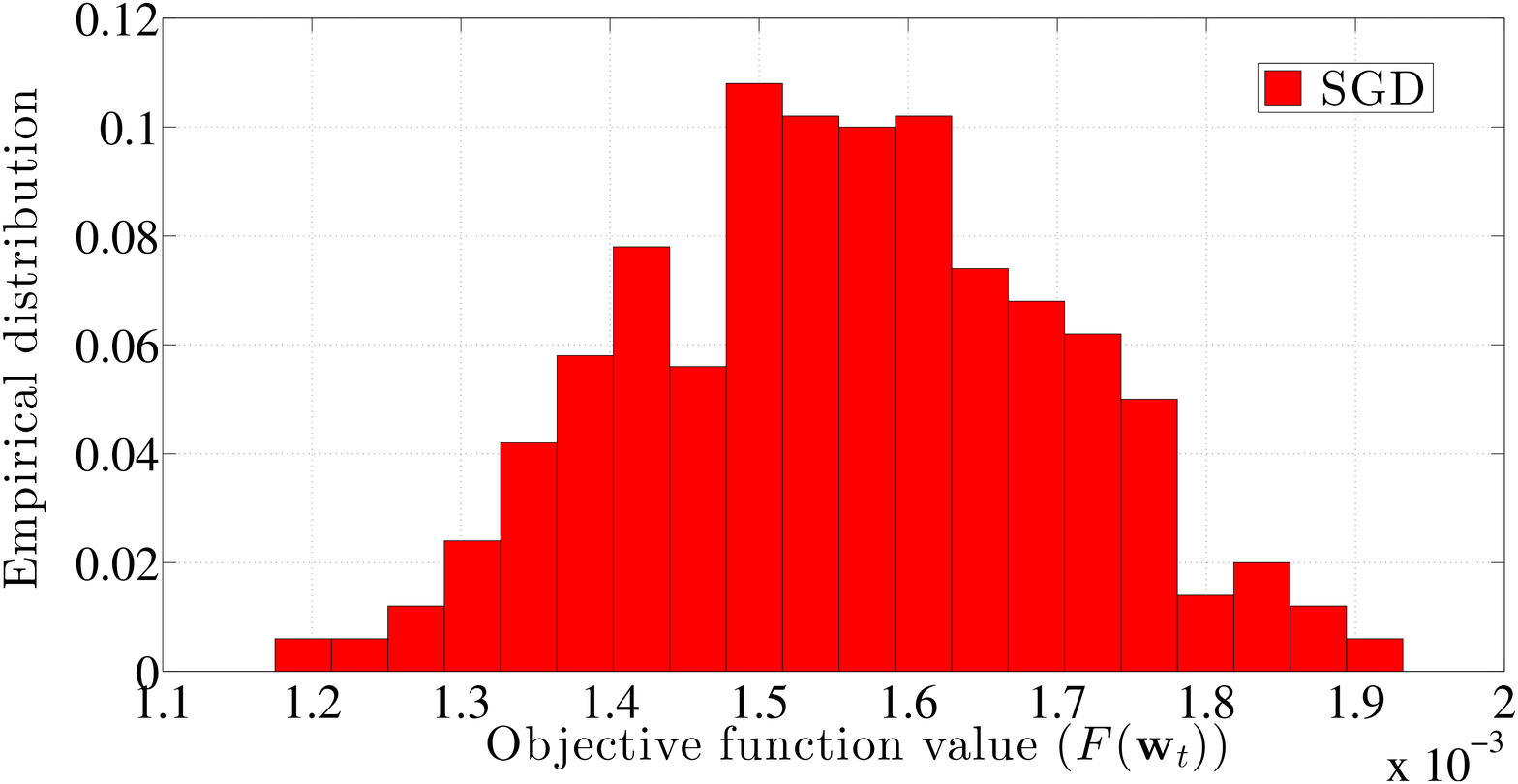}}
   \subfigure[SAG]{\label{fig:sag_100}%
   \includegraphics[width=0.49\linewidth]{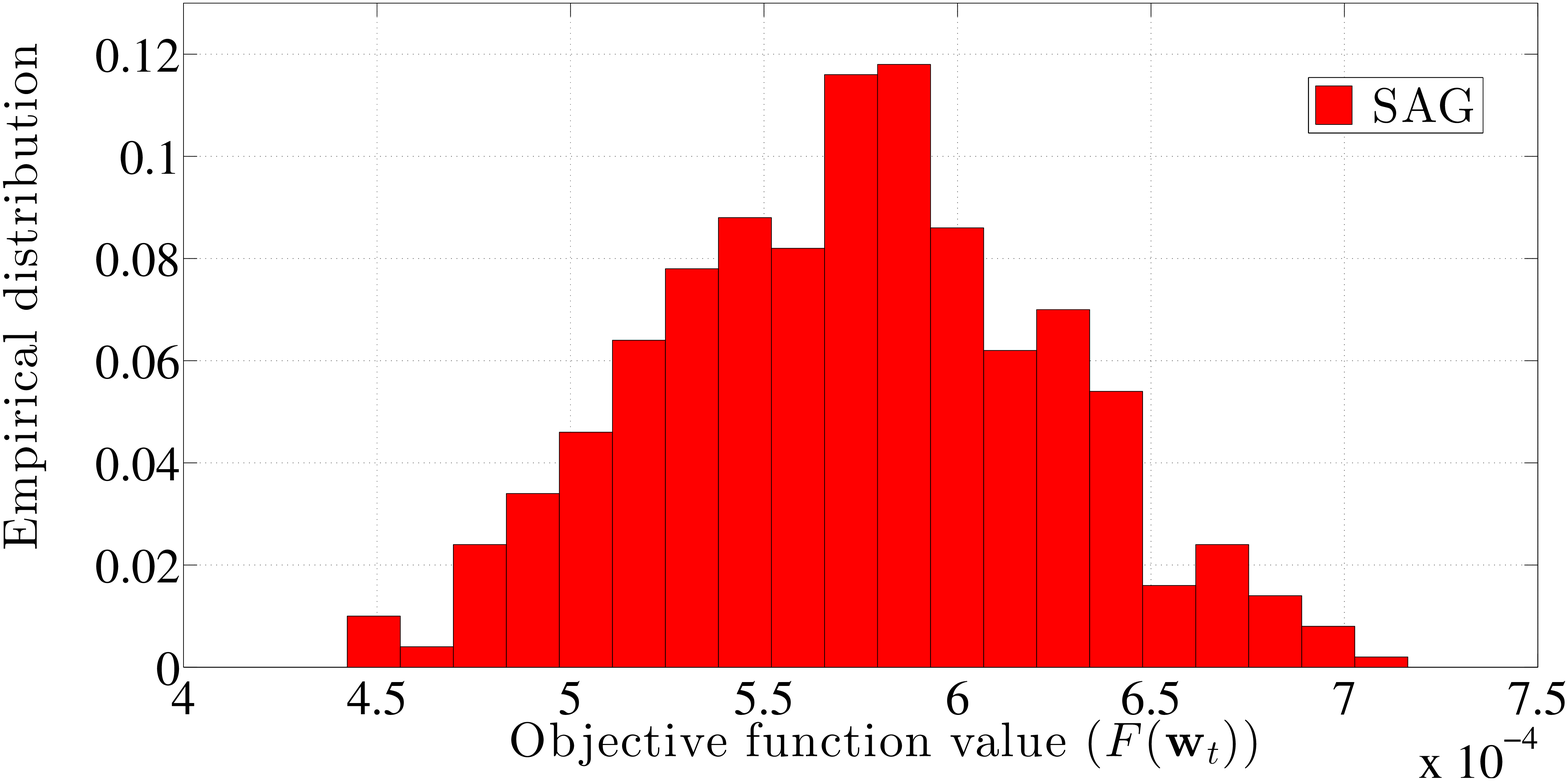}}
   \subfigure[Statistics]{\label{fig:sag_100}{\footnotesize 
   \renewcommand\arraystretch{1.3}  \vspace{-40mm}
   \begin{tabular}{ l l l  l  l }
      \vspace{-40mm}      \\ \hline
                && \multicolumn{3}{c}{Objective function value}\\
      Algorithm && Minimum & Average & Maximum  \\ \hline
      oLBFGS     && 1.3$\times10^{-5}$  & 1.7$\times10^{-5}$ & 3.4$\times10^{-5}$  \\ 
      oBFGS     &&  1.2$\times10^{-5}$  & 1.4$\times10^{-5}$  & 2.0$\times10^{-5}$  \\ 
      RES    && 1.5$\times10^{-5}$    & 1.9$\times10^{-5}$    &  3.3$\times10^{-5}$  \\ 
      SGD   && 1.2$\times10^{-3}$ & 1.6$\times10^{-3}$  & 1.9$\times10^{-3}$   \\ 
      SAG   &&  4.4$\times10^{-4}$  & 5.7$\times10^{-4}$ & 7.1$\times10^{-4}$       \\ \hline
   \end{tabular}}}}
\caption{Histograms of objective function value $F(\bbw_{t})$ after processing $Lt=4\times10^4$ feature vectors for $n=10^2$. The values of objective function for oLBFGS, oBFGS and RES are close to each other and smaller than the objective function values for SAG and SGD. 
}
\label{fig:obj_func_100}\end{figure}

%
\subsection{Convergence versus number of feature vectors processed}\label{sec:SVM2}

For numerical tests we use the squared hinge loss $l((\bbx,y);\bbw)=\max(0,1-y(\bbx^{T}\bbw))^{2}$ in \eqref{SVM}. The training set $\ccalS = \{ (\bbx_{i},y_{i}) \}_{i=1}^{N}$ contains $N=10^4$ feature vectors, half of which belong to the class $y_i=-1$ with the other half belonging to the class $y_i=1$. For the class $y_i=-1$ each of the $n$ components of each of the feature vectors $\bbx_i\in\reals^n$ is chosen uniformly at random from the interval $[-0.8,0.2]$. Likewise, each of the $n$ components of each of the feature vectors $\bbx_i\in\reals^n$ is chosen uniformly at random from the interval $[-0.2,0.8]$ for the class $y_i=1$. In all of our numerical experiments the parameter $\lambda$ in \eqref{SVM} is set to $\lambda=10^{-4}$. In order to study the advantages of oLBFGS we consider two different cases where the dimensions of the feature vectors are $n=10^2$ and $n=10^3$. The size of memory for oLBFGS is set to $\tau=10$ in both cases. For SGD and SAG the sample size in \eqref{stochastic_gradient} is $L=1$ and for RES, oBFGS and oLBFGS is $L=5$.
In all tests, the number of feature vectors processed is represented by the product $Lt$ between the iteration index and the sample size used to compute stochastic gradients. This is done because the sample sizes are different. For all five algorithms we use a decreasing stepsize sequence of the form $\epsilon_t=\epsilon_{0}T_{0} / (T_{0}+t)$. We report results for $\epsilon_{0} = 2\times10^{-2}$ and $T_{0}=10^2$ for RES, oLBFGS and oBFGS, which are the values that yield best average performance after processing $4\times10^4$ feature vectors. Further improvements can be obtained by tuning stepsize parameters individually for each individual algorithm and feature dimension $n$. Since these improvements are minor we report  common parameters for easier reproducibility. For SGD and SAG, whose performance is more variable, we tune the various parameters individually for each dimension $n$ and report results for the combination that yields best average performance after processing $4\times10^4$ feature vectors. 

%
\begin{figure}[t]{
   \subfigure[oLBFGS]{\label{fig:obfgs_1000}%
   \includegraphics[width=0.49\linewidth]{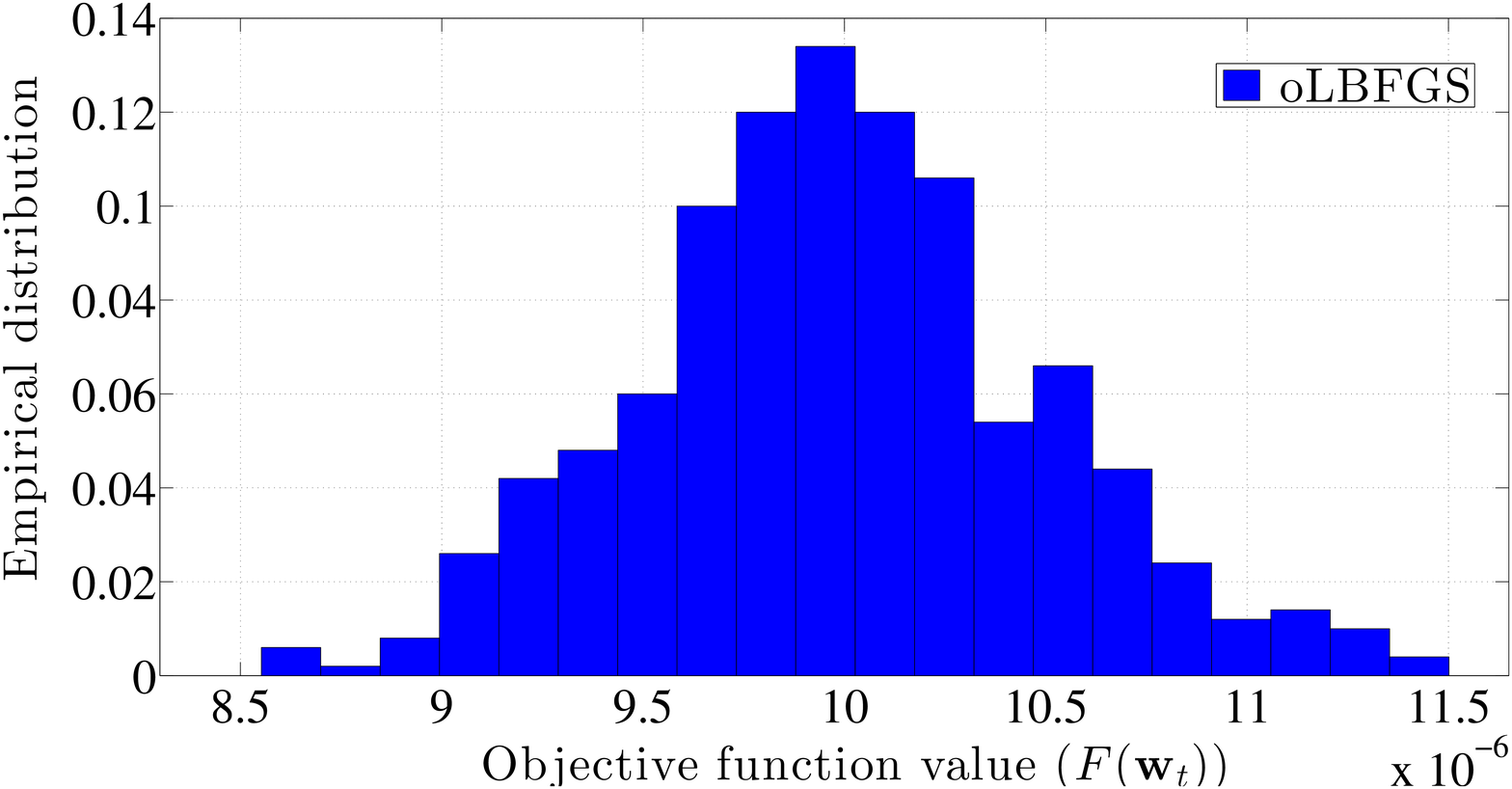}}%
   \subfigure[oBFGS]{\label{fig:olbfgs_1000}%
   \includegraphics[width=0.49\linewidth]{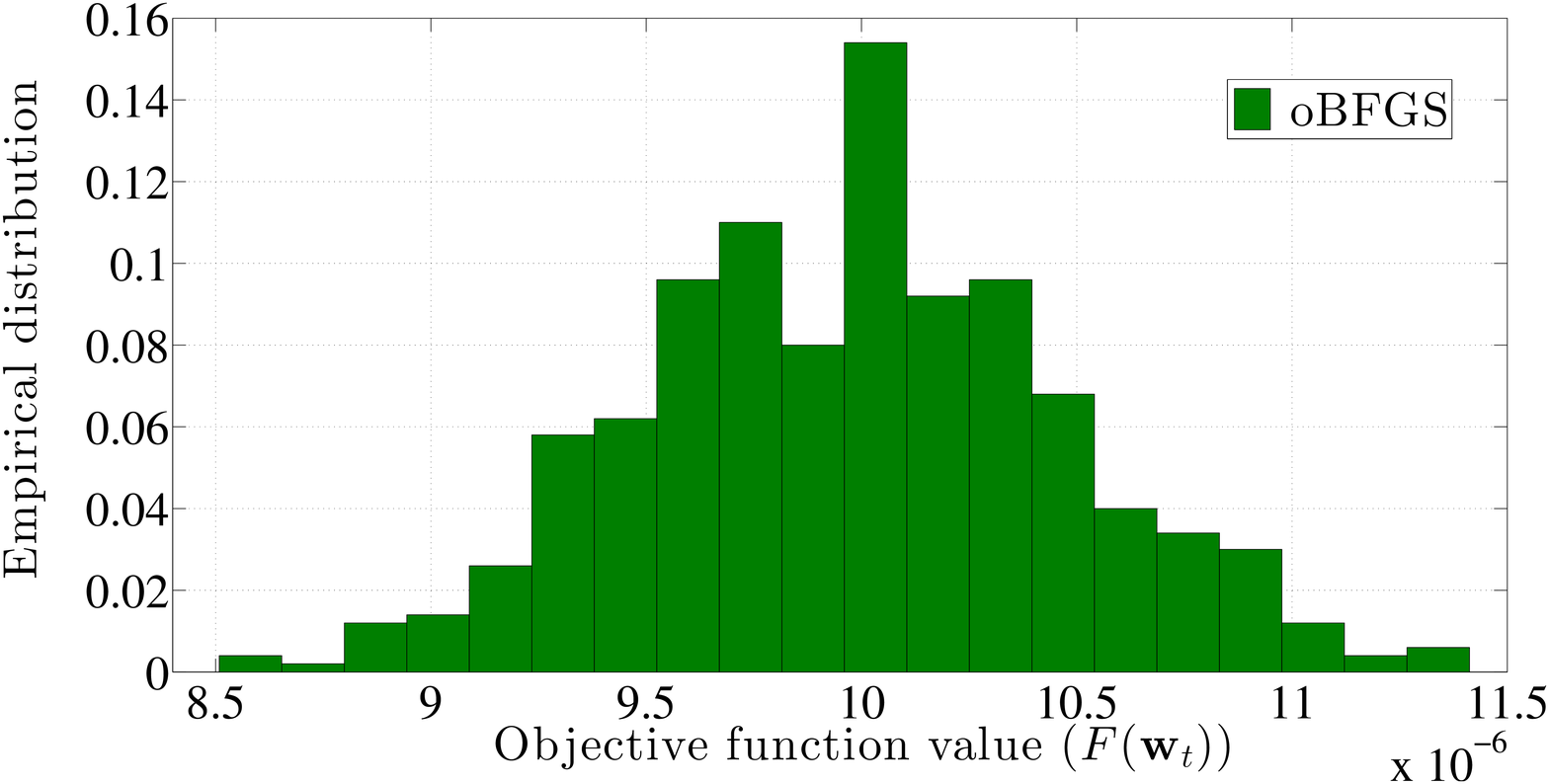}}
   \subfigure[RES]{\label{fig:res_1000}%
   \includegraphics[width=0.49\linewidth]{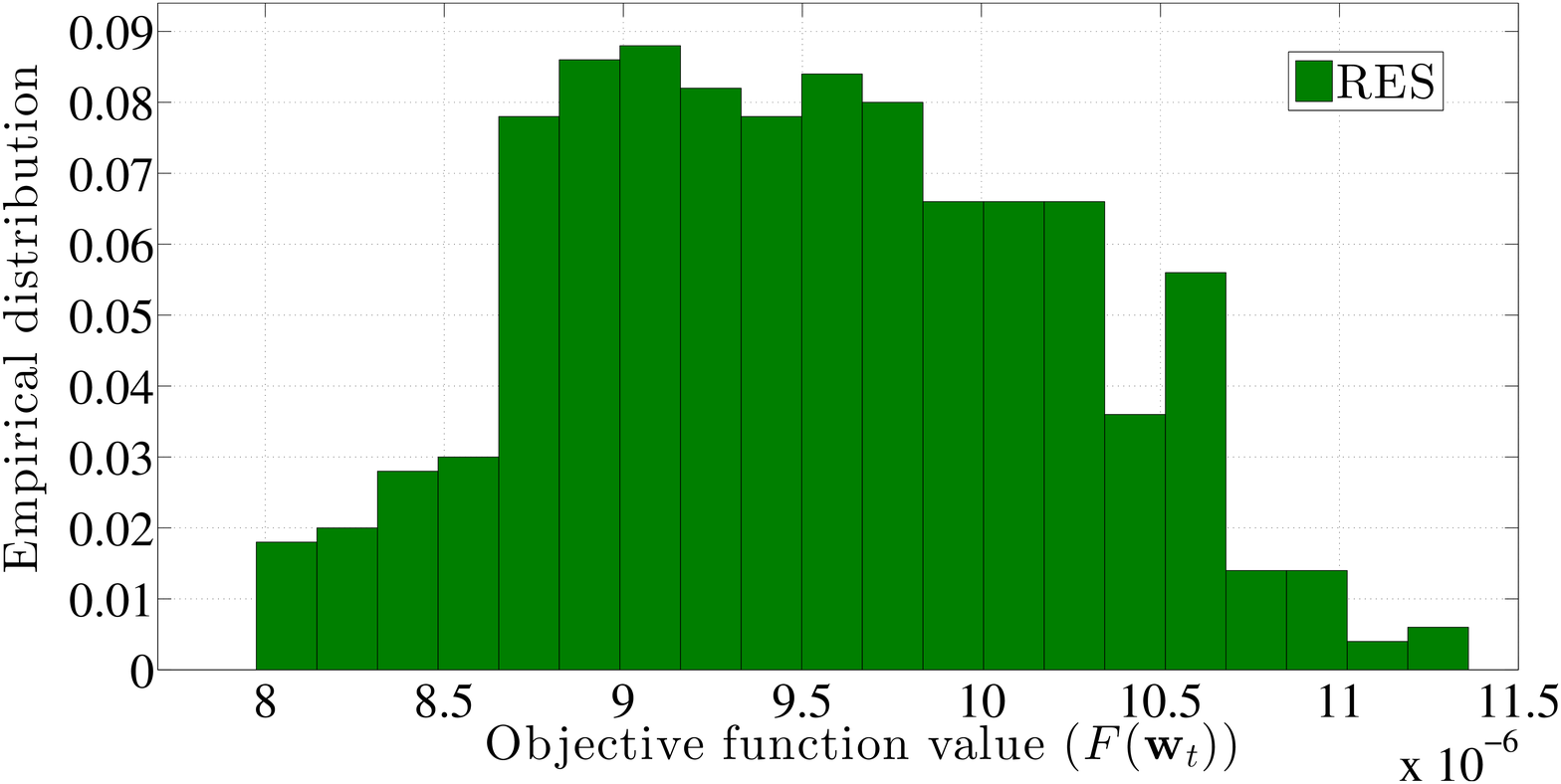}}%
   \subfigure[SGD]{\label{fig:sgd_1000}%
   \includegraphics[width=0.49\linewidth]{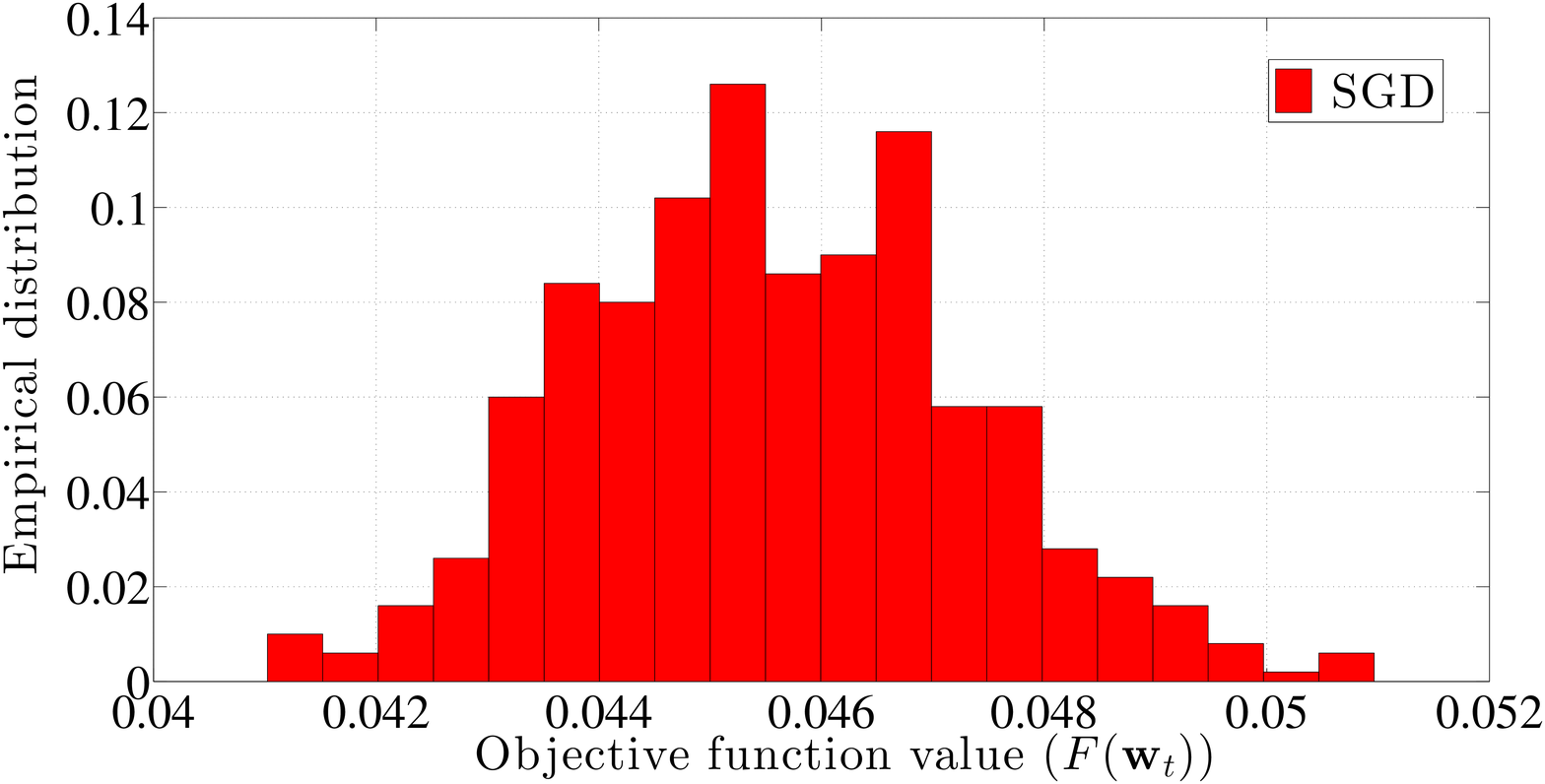}}
   \subfigure[SAG]{\label{fig:sag_1000}%
   \includegraphics[width=0.49\linewidth]{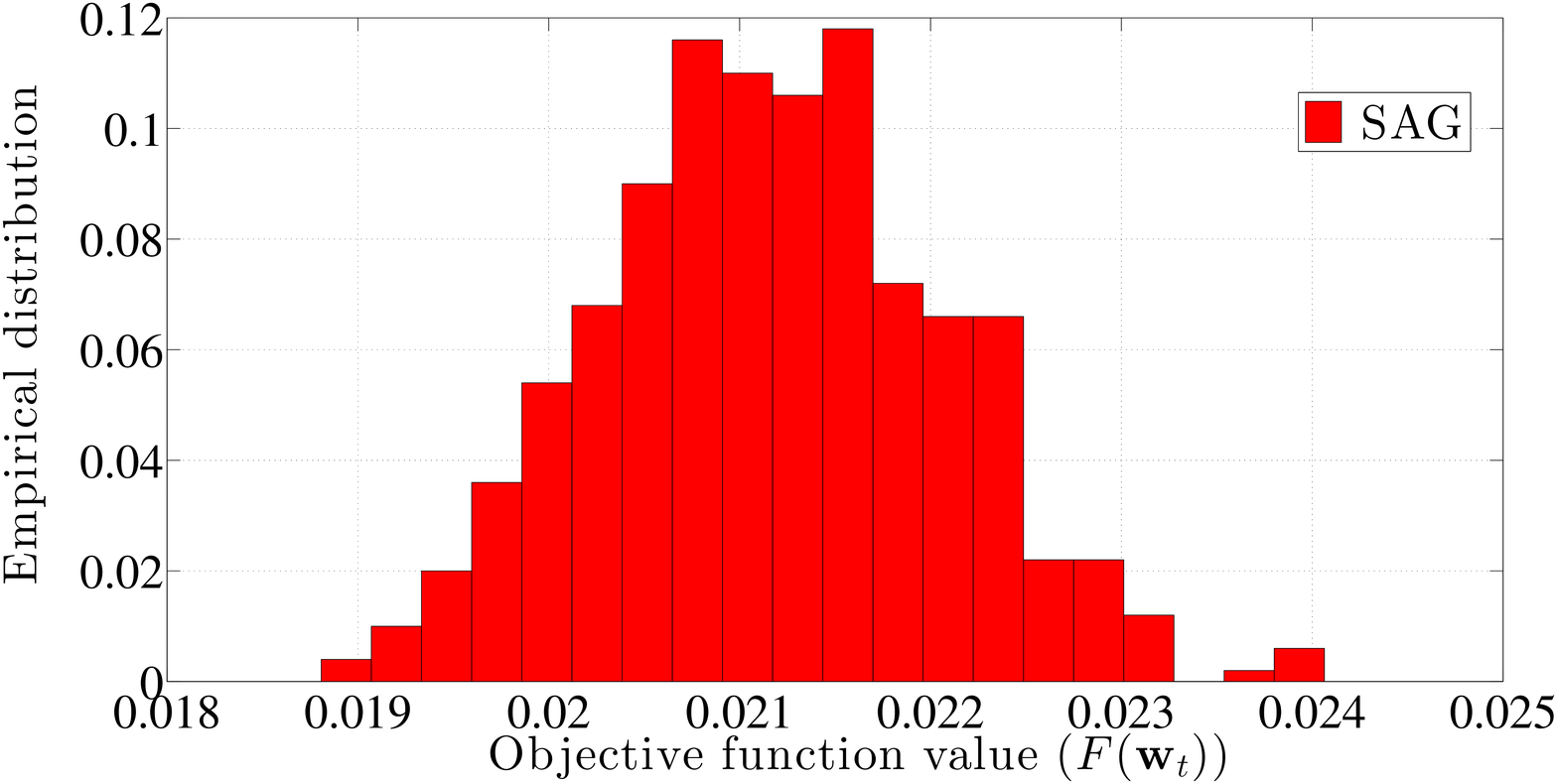}}
   \subfigure[Statistics]{\label{fig:sag_100}{\footnotesize 
   \renewcommand\arraystretch{1.3}  \vspace{-40mm}
   \begin{tabular}{ l l l  l  l }
      \vspace{-40mm}      \\ \hline
                && \multicolumn{3}{c}{Objective function value}\\
      Algorithm && Minimum & Average & Maximum  \\ \hline
      oLBFGS   && 8.6$\times10^{-6}$ & 9.9$\times10^{-6}$ & 11.5$\times10^{-6}$\\ 
      oBFGS  && 8.5$\times10^{-6}$ & 9.8$\times10^{-6}$ & 11.4$\times10^{-6}$\\
      RES   && 7.9$\times10^{-6}$ & 9.5$\times10^{-6}$ & 11.3$\times10^{-6}$\\
      SGD    && 4.1$\times10^{-2}$ & 4.5$\times10^{-2}$ & 5.1$\times10^{-2}$\\
      SAG  && 1.9$\times10^{-2}$ & 2.1$\times10^{-2}$ & 2.4$\times10^{-2}$ \\ \hline
   \end{tabular}}}}
\caption{Histograms of objective function value $F(\bbw_{t})$ after processing $Lt=4\times10^4$ feature vectors for $n=10^3$. The values of objective function for oBFGS, oLBFGS and RES are close to each other and smaller than the objective function values for SAG and SGD. }
\label{fig:obj_func_1000}\end{figure}

%
\begin{figure}[t]{
   \subfigure[oLBFGS]{\label{fig:obfgs_100_2}%
   \includegraphics[width=0.49\linewidth]{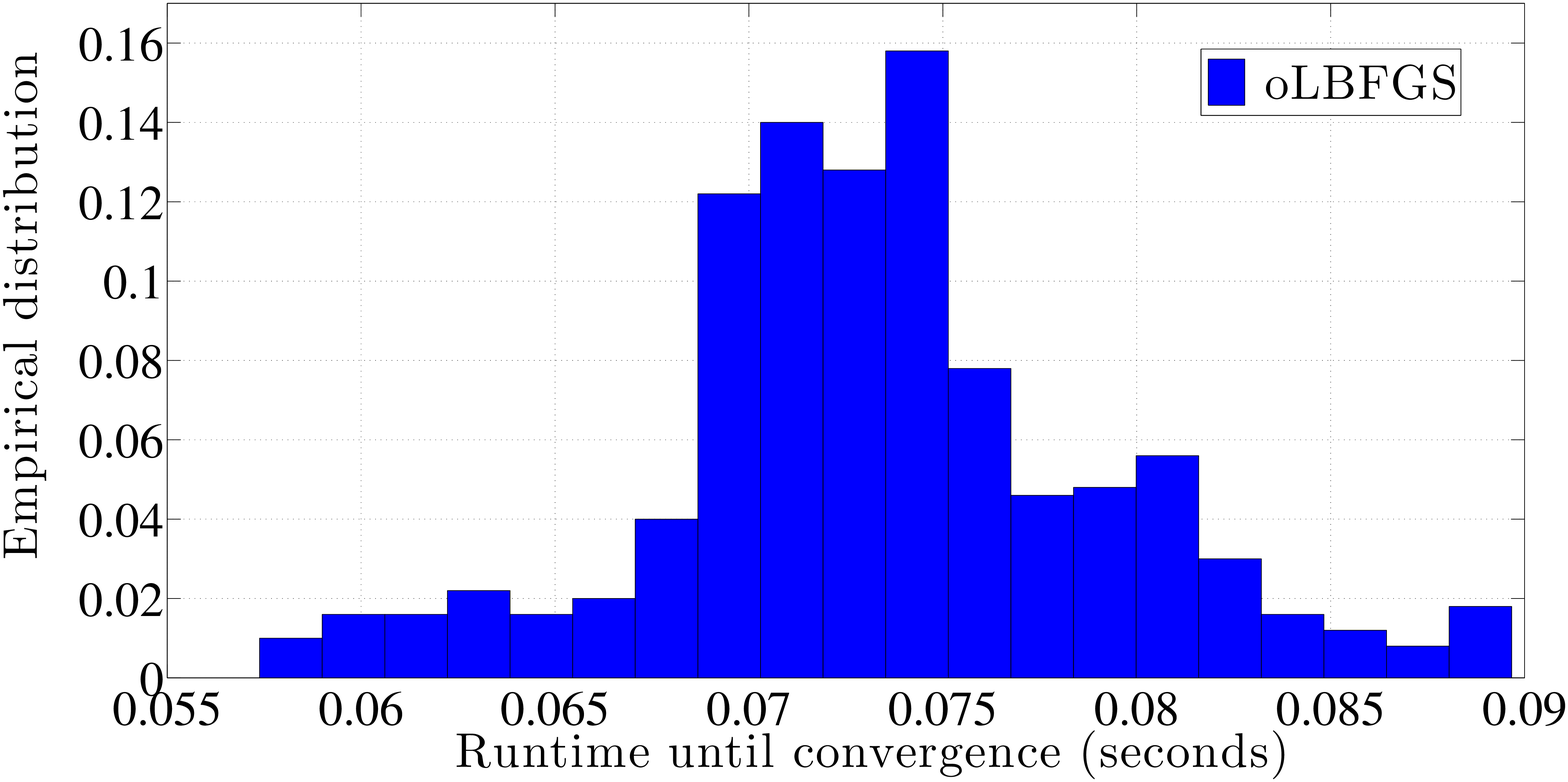}}%
   \subfigure[oBFGS]{\label{fig:olbfgs_100_2}%
   \includegraphics[width=0.49\linewidth]{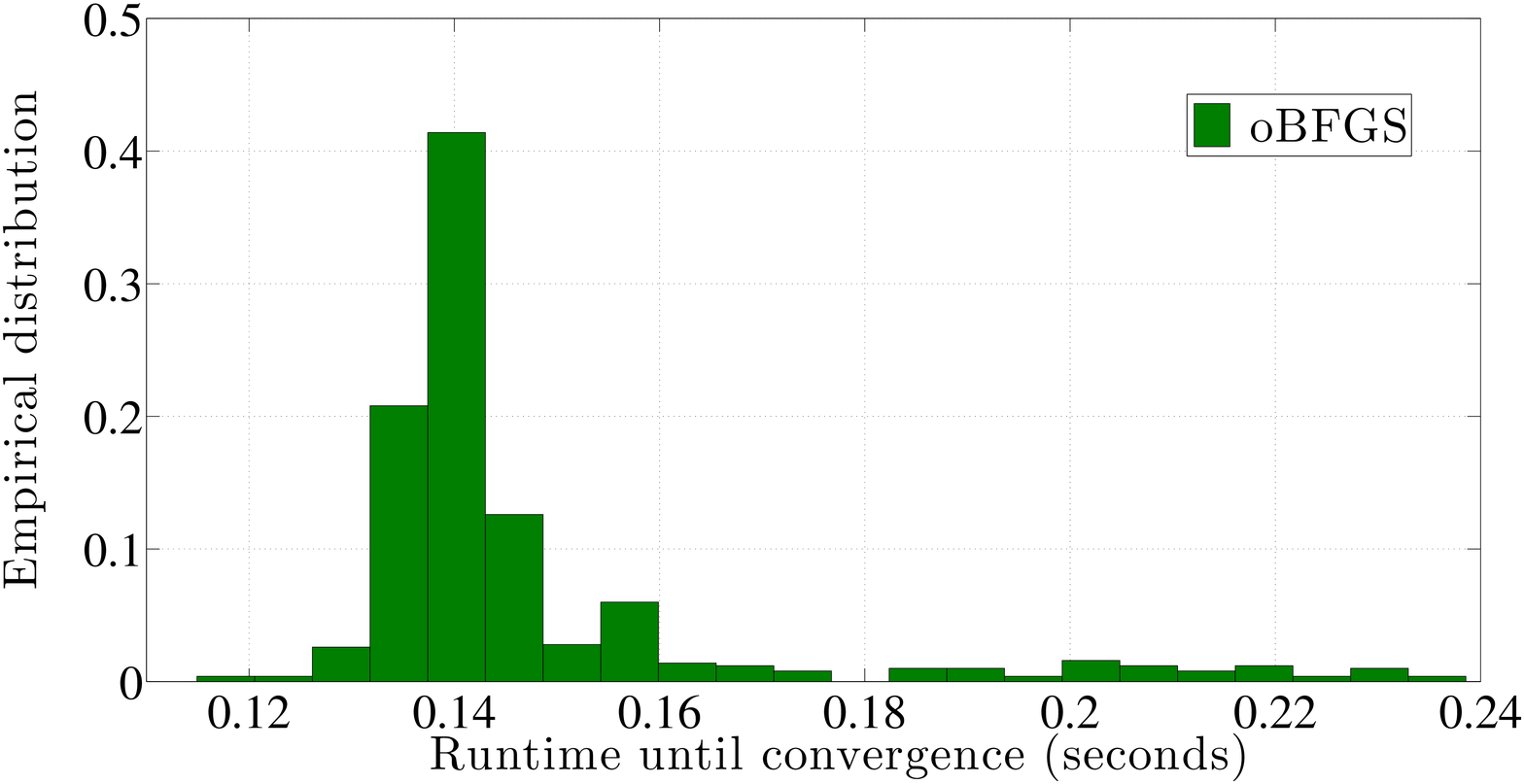}}
   \subfigure[RES]{\label{fig:res_100_2}%
   \includegraphics[width=0.49\linewidth]{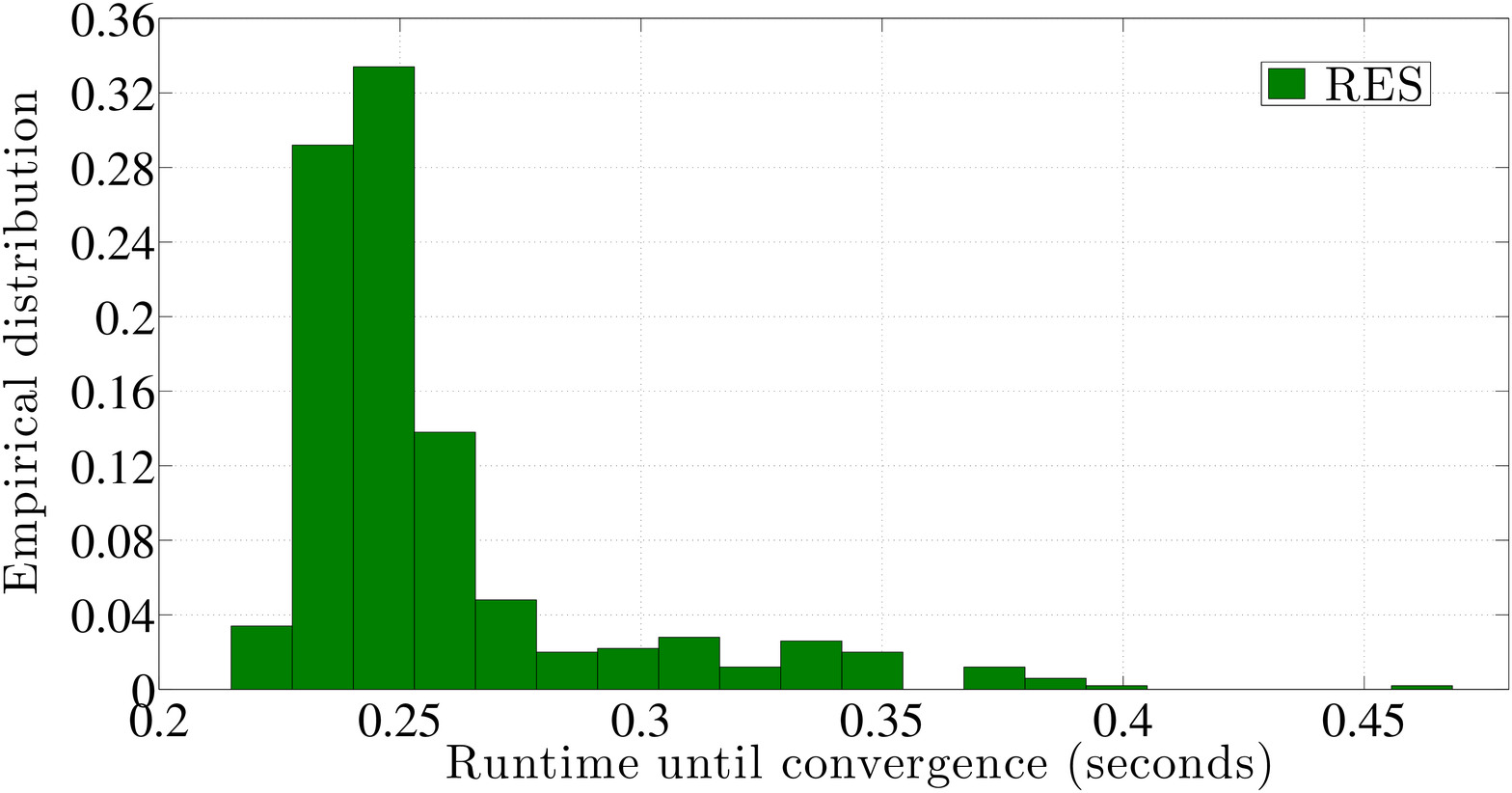}}%
   \subfigure[SGD]{\label{fig:sgd_100_2}%
   \includegraphics[width=0.49\linewidth]{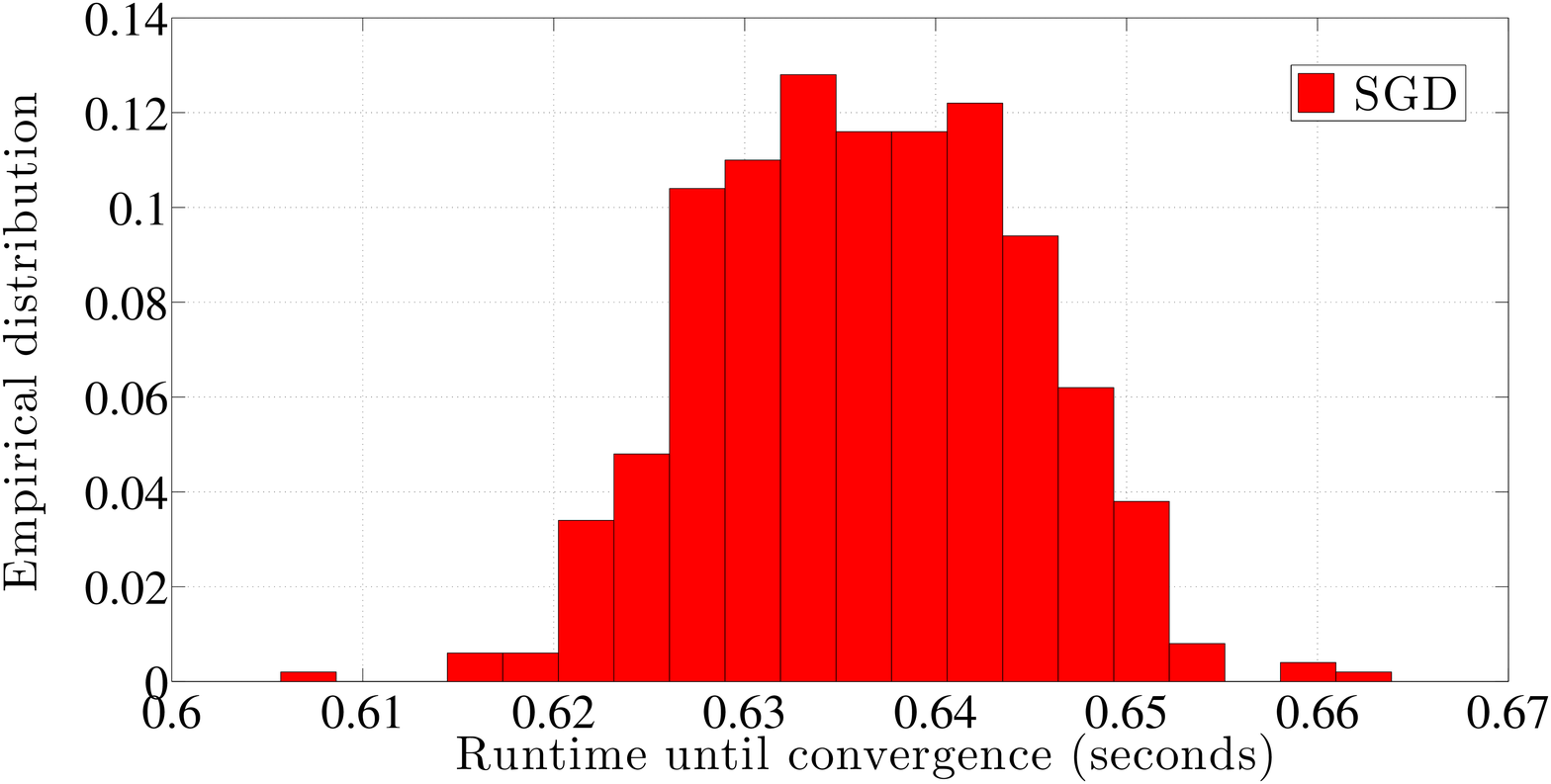}}
   \subfigure[SAG]{\label{fig:sag_100_2}%
   \includegraphics[width=0.49\linewidth]{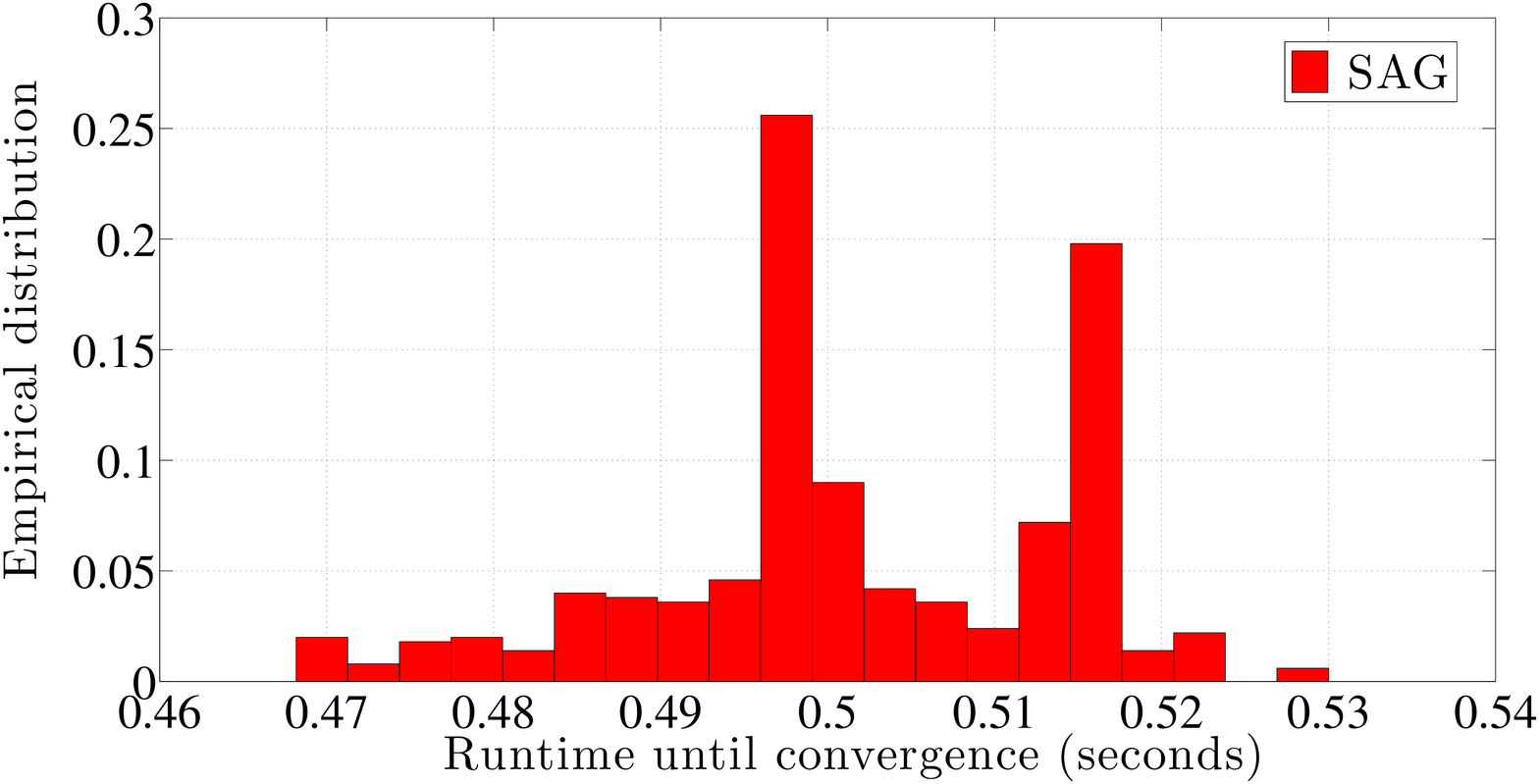}}
   \subfigure[Statistics]{\label{fig:sag_100}{\footnotesize 
   \renewcommand\arraystretch{1.3}  \vspace{-40mm}
   \begin{tabular}{ l l l  l  l }
      \vspace{-40mm}      \\ \hline
                && \multicolumn{3}{c}{CPU runtime (seconds)}\\
      Algorithm && Minimum & Average & Maximum  \\ \hline
      oLBFGS     && 0.0574     & 0.0734      & 0.0897      \\ 
      oBFGS       && 0.1149     & 0.1485      & 0.2375       \\ 
      RES    && 0.2150      & 0.2609     & 0.4623      \\ 
      SGD       && 0.6057      & 0.6364      & 0.6639       \\ 
      SAG       && 0.4682      & 0.5017      & 0.5300      \\ \hline
   \end{tabular}}}}
\caption{ Histograms of required CPU runtime for achieving objective function value $F(\bbw_{t})=10^{-4}$ when $n=10^2$. The convergence time of oLBFGS is smaller than the required runtimes of oBFGS and RES, while SAG and SGD are slower than all the three quasi-Newton methods.}
\label{fig:runtime_100} \end{figure}

%
Figures \ref{fig:obj_func_100} and \ref{fig:obj_func_1000} show the empirical distributions of the objective function value $F(\bbw_{t})$ attained after processing $Lt=4\times 10^4$ feature vectors using $J=10^3$ realizations for the cases that $n=10^2$ and $n=10^3$, respectively. According to Figure \ref{fig:obj_func_100}  the averages of objective value function for oLBFGS, oBFGS and RES are {$1.7\times10^{-5}$, $1.4\times10^{-5}$ and $1.9\times10^{-5}$}, respectively. These numbers show that the performance of oLBFGS is very close to the performances of oBFGS and RES. This similarity holds despite the fact that oLBFGS uses only the last $\tau=10$ stochastic gradients to estimate curvature whereas oBFGS and RES utilize all past stochastic gradients to do so. The advantage of oLBFGS is in the smaller computational cost of processing feature vectors as we discuss in Section \ref{sec:SVM3}. The corresponding average objective values achieved by SGD and SAG after processing $Lt=4\times 10^4$ feature vectors are $1.6\times10^{-3}$ and $5.7\times10^{-4}$, respectively. Both of these are at least an order of magnitude larger than the average objective value achieved by oLBFGS -- or RES and oBFGS for that matter.

Figure \ref{fig:obj_func_1000} repeats the study in Figure \ref{fig:obj_func_100} for the case in which the feature vector dimension is increased to $n=10^3$. The performance of oLBGS is still about the same as the performances of oBFGS and RES. The average objective function values achieved after processing $Lt=4\times 10^4$ feature vectors are $9.9\times10^{-6}$, $9.8\times10^{-6}$ and $9.5\times10^{-6}$} for oLBFGS, oBFGS and RES, respectively. The relative performance with respect to SGD and SAG, however, is now larger. The averages of objective function values for SAG and SGD in this case are $2.1\times10^{-2}$ and $4.5\times10^{-2}$, respectively. These values are more than 3 orders of magnitude larger than the corresponding values achieved by oLBFGS. This relative improvement can be further increased if we consider problems of even larger dimension. Further observe that oBFGS and RES start to become impractical if we further increase the feature vector dimension since the respective iterations have computational costs of order $O(n^2)$ and $O(n^3)$. We analyze this in detail in the following section.

%
\subsection{Convergence versus processing time}\label{sec:SVM3}

%
\begin{figure}[t]{
   \subfigure[oLBFGS]{\label{fig:obfgs_1000_2}%
   \includegraphics[width=0.49\linewidth]{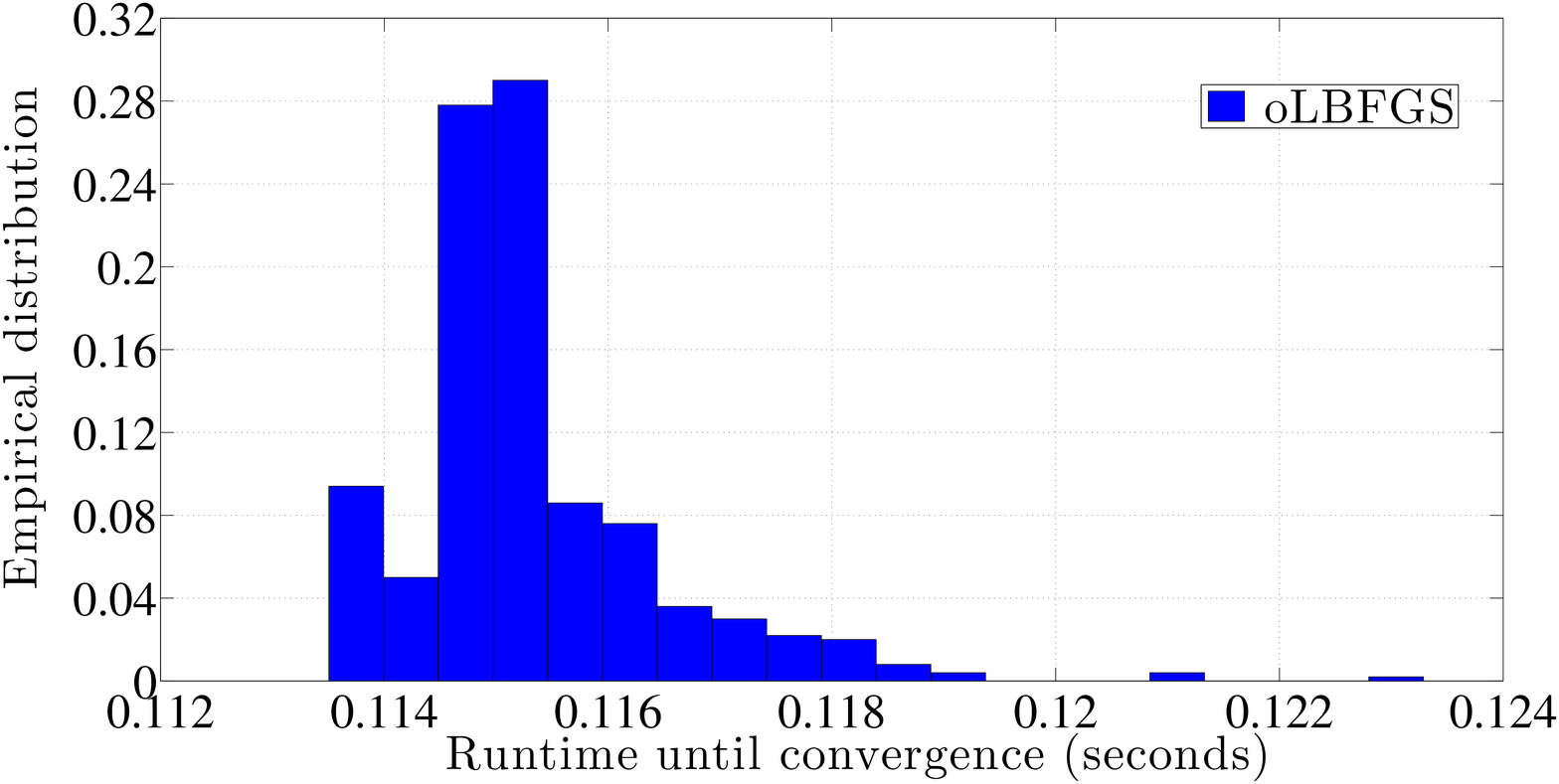}}%
   \subfigure[oBFGS]{\label{fig:olbfgs_1000_2}%
   \includegraphics[width=0.49\linewidth]{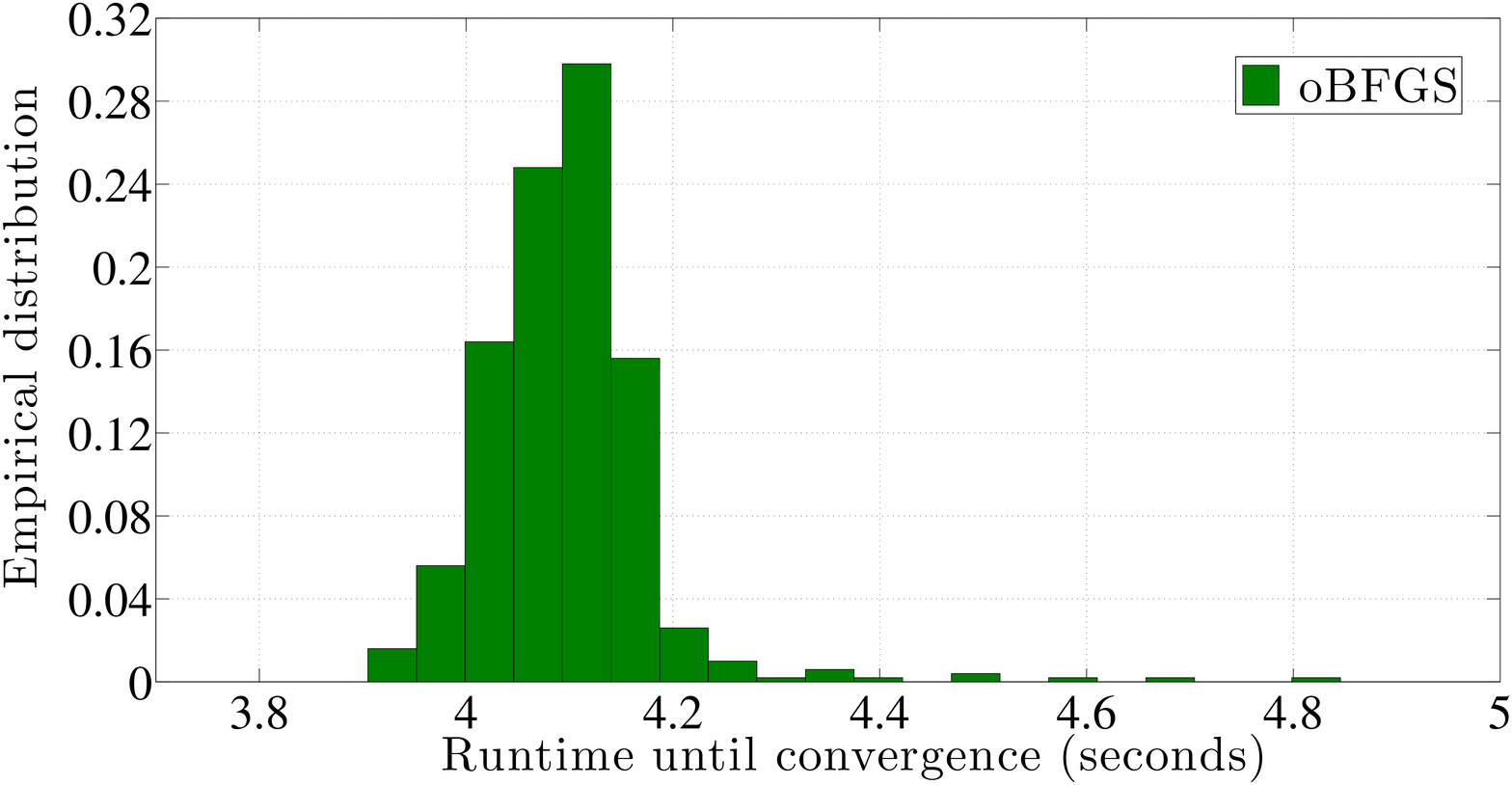}}
   \subfigure[RES]{\label{fig:res_1000_2}%
   \includegraphics[width=0.49\linewidth]{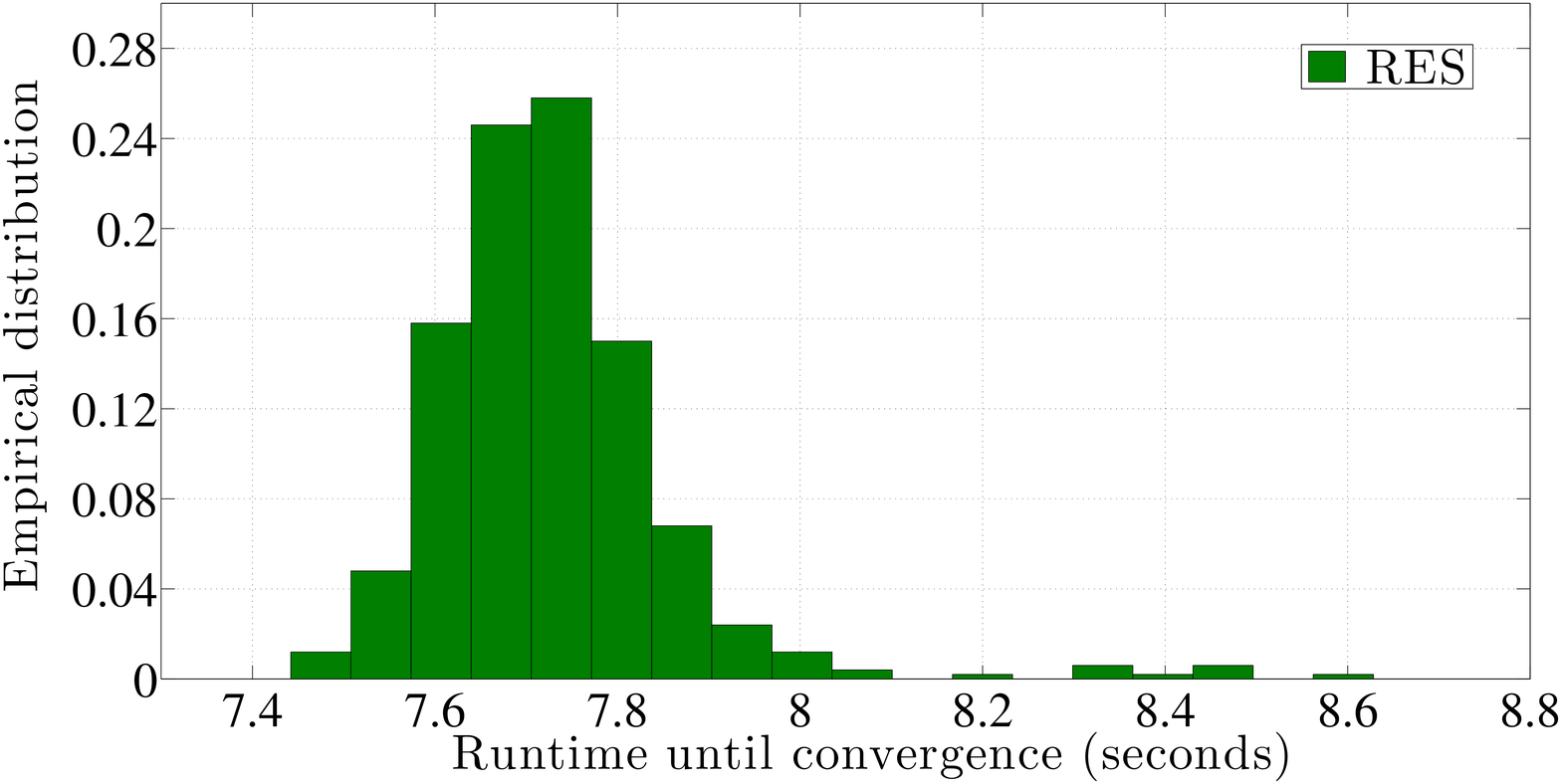}}%
   \subfigure[SGD]{\label{fig:sgd_1000_2}%
   \includegraphics[width=0.49\linewidth]{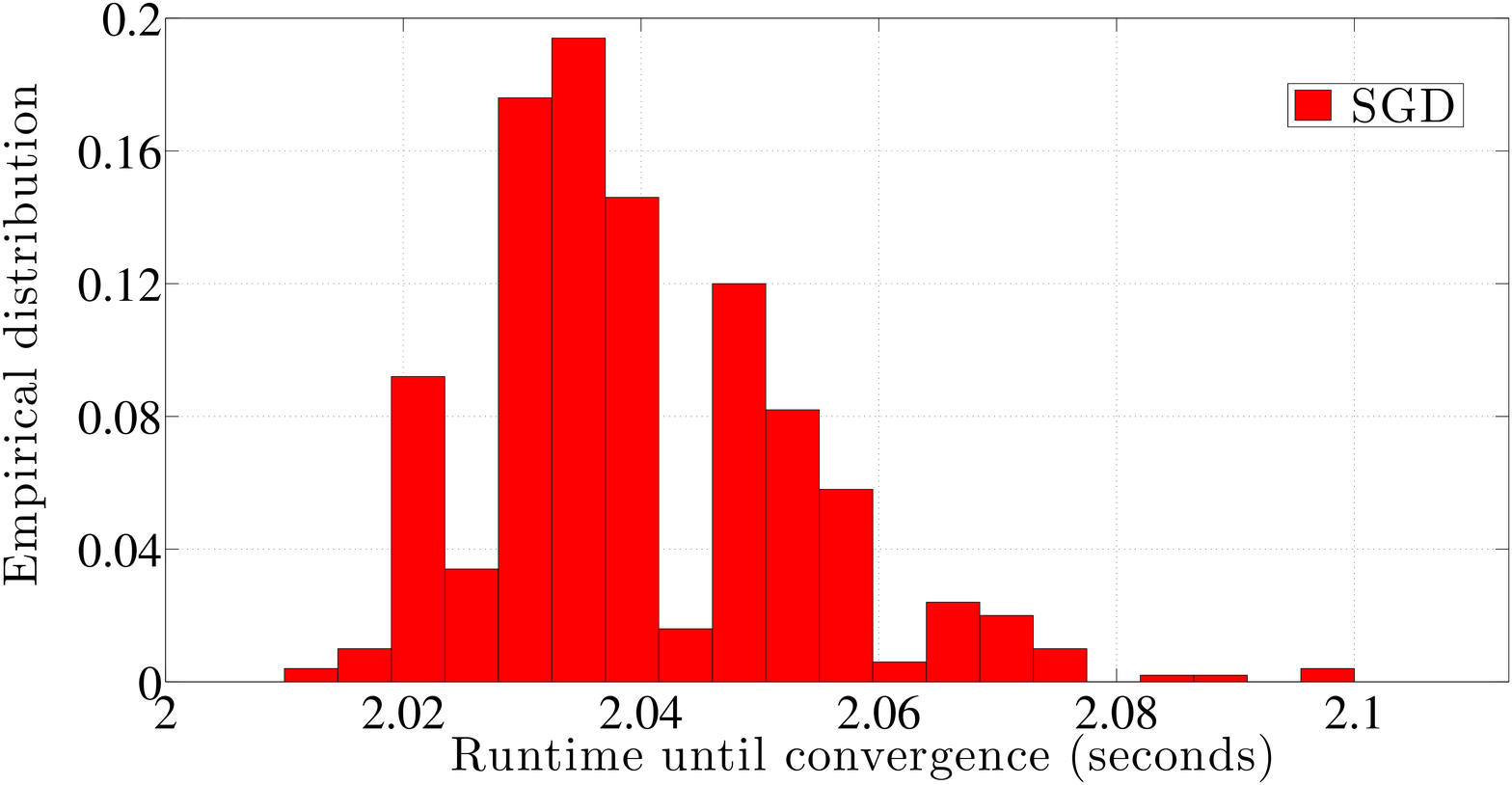}}
   \subfigure[SAG]{\label{fig:sag_1000_2}%
   \includegraphics[width=0.49\linewidth]{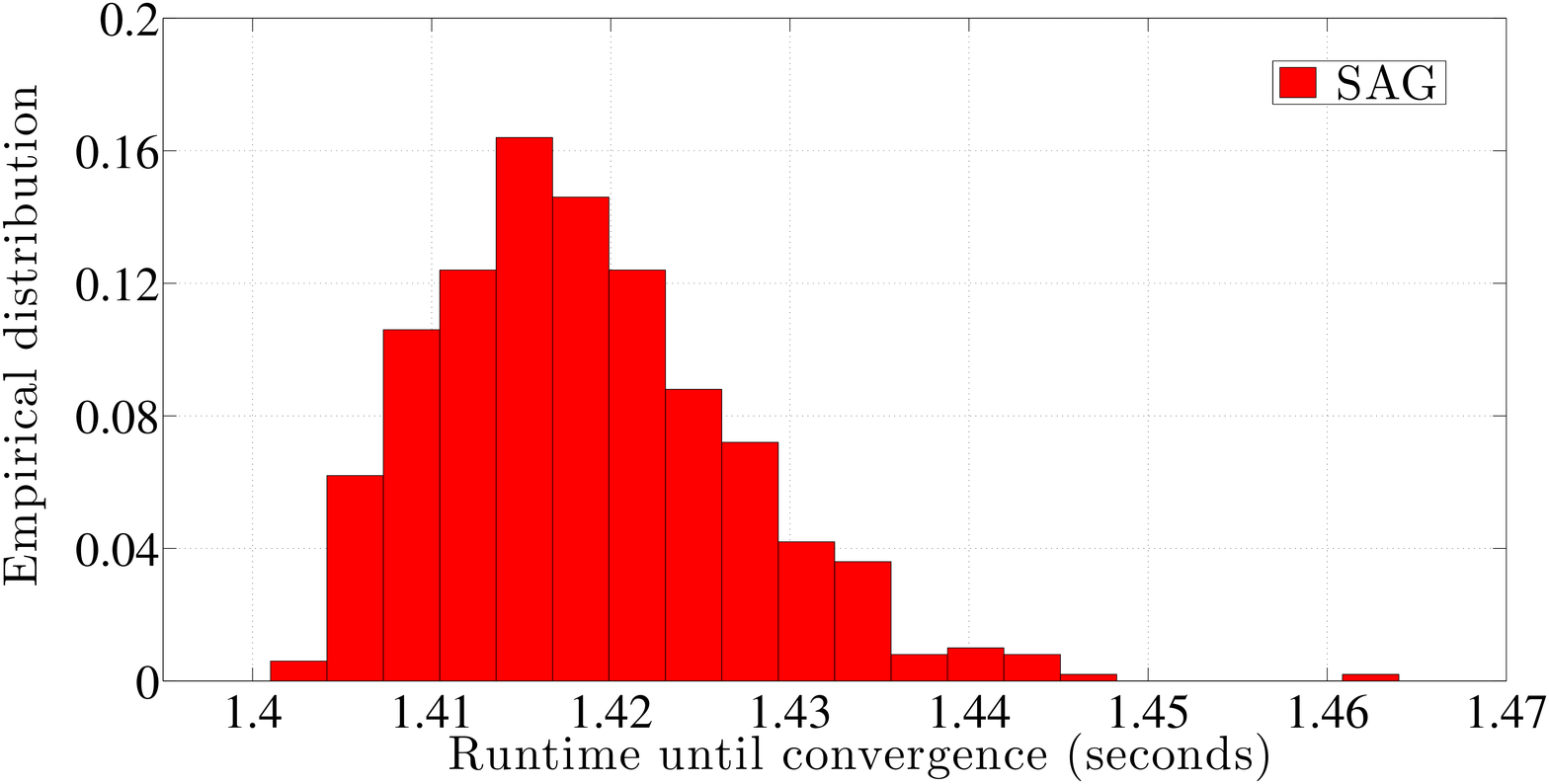}}
   \subfigure[Statistics]{\label{fig:sag_100}{\footnotesize 
   \renewcommand\arraystretch{1.3}  \vspace{-40mm}
   \begin{tabular}{ l l l  l  l }
      \vspace{-40mm}      \\ \hline
                && \multicolumn{3}{c}{CPU runtime (seconds)}\\
      Algorithm && Minimum & Average & Maximum  \\ \hline
      oLBFGS  &&     0.1139      &   0.1153    & 0.1230       \\ 
      oBFGS       &&   3.90    &    4.11   & 4.83       \\ 
      RES   &&    7.44   & 7.73      &  8.61     \\ 
      SGD       && 2.01      & 2.03      & 2.10       \\ 
      SAG       && 1.40      & 1.42      & 1.46       \\ \hline
   \end{tabular}}}}
\caption{ Histograms of required CPU runtime for achieving objective function value $F(\bbw_{t})=10^{-5}$ when $n=10^3$.  SAG and SGD have a faster convergence time in comparison to oBFGS and RES, while oLBFGS is the fastest algorithm among all.}
\label{fig:runtime_1000} \end{figure}

%

The analysis in Section \ref{sec:SVM2} is relevant for online implementations in which the goal is to make the best possible use of the information provided by each new acquired feature vector. In implementations where computational cost is of dominant interest we have to account for the fact that the respective iteration costs are of order $O(n)$ for SGD and SAG, of order $O(\tau n)$ for oLBFGS, and of orders $O(n^2)$ and $O(n^3)$ for oBFGS and RES. As we increase the problem dimension we expect the convergence time advantages of oBFGS and RES in terms of number of feature vectors processed to be overwhelmed by the increased computational cost of each iteration. For oLBFGS, on the contrary, we expect the convergence time advantages in terms of number of feature vectors processed to persist in terms of processing time. To demonstrate that this is the case we repeat the experiments in Section \ref{sec:SVM2} but record the processing time required to achieve a target objective value. The parameters used here are the same parameters of Section \ref{sec:SVM2}.

In Figure \ref{fig:runtime_100} we consider $n=10^2$ and record the processing time required to achieve the objective function value $F(\bbw_{t})=10^{-4}$. Histograms representing empirical distributions of execution times measured in seconds (s) are shown for oLBFGS, oBFGS, RES, SGD, and SAG. We also summarize the average minimum and maximum times observed for each algorithm. The average run times for oBFGS and RES are $0.14\,\text{s}$ and $0.26\,\text{s}$ which are better than the average run times of SGD and SAG that stand at $0.63\,\text{s}$ and $0.50\,\text{s}$. The advantage, however, is less marked than when measured with respect to the number of feature vector processed. For oLBGS the advantage with respect to SGD and SAG is still close to one order of magnitude since the average convergence time stands at $0.073\,\text{s}$. When measured in computation time oLBGS is also better than RES and oBFGS, as expected.

Figure \ref{fig:runtime_1000} presents the analogous histograms and summary statistics when the feature vector dimension is $n=10^3$ and the algorithm is run until achieving the objective value $F(\bbw_{t})=10^{-5}$. For this problem and metric the performances of RES and oBFGS are worse than the corresponding performances of SGD and SAG. The respective average convergence times are $7.7\,\text{s}$ and $4.1\,\text{s}$ for RES and oBFGS and $1.4\,\text{s}$ and $2.0\,\text{s}$ for SAG and SGD. The oLBFGS algorithm, however, has an average convergence time of $0.11\,\text{s}$. This is still an order of magnitude faster than the first order methods SAG and SGD -- and has an even larger advantage with respect to oBFGS and RES, by extension. The relative reduction of execution times of oLBGS relative to all other 4 methods becomes more marked for problems of larger dimension. We investigate these advantages on the search engine advertising problem that we introduce in the following section.


\section{Search engine advertising} \label{sec:ad_application}

We apply oLBFGS to the problem of predicting the click-through rate (CTR) of an advertisement displayed in response to a specific search engine query by a specific visitor. In these problems we are given meta information about an advertisement, the words that appear in the query, as well as some information about the visitor and are asked to predict the likelihood that this particular ad is clicked by this particular user when performing this particular query. The information specific to the ad includes descriptors of different characteristics such as the words that appear in the title, the name of the advertiser, keywords that identify the product, and the position on the page where the ad is to be displayed. The information specific to the user is also heterogeneous and includes gender, age, and propensity to click on ads.  To train a classifier we are given information about past queries along with the corresponding click success of the ads displayed in response to the query. The ad metadata along with user data and search words define a feature vector that we use to train a logistic regressor that predicts the CTR of future ads. Given the heterogeneity of the components of the feature vector we expect a logistic cost function with skewed level sets and consequent large benefits from the use of oLBFGS. 

%
\renewcommand\arraystretch{1.3} 
\begin{table*}\centering\small{
\caption{Components of the feature vector for prediction of advertisements click-through rates. For each feature class we report the total number of components in the feature vector as well as the maximum and average number of nonzero components.                    \qquad\qquad}
\begin{tabular}{ l  r  r  r }\hline    
    && \multicolumn{2}{c}{{\bf Nonzero components}} \\
       {\bf Feature type}    
     & {\bf Total components} 
     & {\bf Maximum (observed/structure)} 
     & {\bf Mean (observed)} \\ \hline
    Age         & 6$\qquad$                  & 1 (structure)$\qquad$               & $1.0\qquad$   \\ 
    Gender      & 3$\qquad$                  & 1 (structure)$\qquad$               & $1.0\qquad$\\ 
      Impression      & 3$\qquad$                  & 1 (structure)$\qquad$              & $1.0\qquad$  \\ 
Depth      & 3$\qquad$                  & 1 (structure)$\qquad$                & $1.0\qquad$   \\ 
 Position      & 3$\qquad$                  & 1  (structure)$\qquad$              & $1.0\qquad$  \\ 
Query      & 20,000$\qquad$             &  $ 125$ (observed)$\qquad$           &$3.0\qquad$  \\ 
   Title     & 20,000$\qquad$             &  $29$ (observed)$\qquad$          &$8.8\qquad$  \\ 
   Keyword       & 20,000$\qquad$             &      $16$ (observed)$\qquad$           &$2.1\qquad$  \\ 
  Advertiser ID       &    5,184$\qquad$         &     1 (structure)$\qquad$          &$1.0\qquad$   \\ 
  Advertisement ID  & 108,824$\qquad$               &  1 (structure)$\qquad$           &$1.0\qquad$   \\ \hline
     {\bf Total} 
   & {\bf 174,026}$\qquad$            
   & {\bf148 (observed)}$\qquad$         
   & {\bf 20.9$\qquad\!\!$}   \\ 
\end{tabular}}
\label{table_summary_methods_properties}
\end{table*}

%
\subsection{Feature vectors}\label{sec:dataset}

For the CTR problem considered here we use the Tencent search engine data set \cite{Soso}. This data set contains the outcomes of 236 million ($236 \times 10^6$) searches along with information about the ad, the query, and the user. The information contained in each sample point is the following:

\begin{itemize}
\item User profile: If known, age and gender of visitor performing query.
\item Depth: Total number of advertisements displayed in the search results page.
\item Position: Position of the advertisement in the search page.
\item Impression: Number of times the ad was displayed to the user who issued the query.
\item Query: The words that appear in the user's query.
\item Title: The words that appear in the title of ad.
\item Keywords: Selected keywords that specify the type of product.
\item Ad ID: Unique identifier assigned to each specific advertisement.
\item Advertiser ID: Unique identifier assigned to each specific advertiser.
\item Clicks: Number of times the user clicked on the ad.
\end{itemize}

From this information we create a set of feature vectors $\{\bbx_i\}_{i=1}^{N}$, with corresponding labels $y_i\in \{-1,1\}$. The label associated with feature vector $\bbx_{i}$ is $y_i=1$ if the number of clicks in the ad is more than $0$. Otherwise the label is $y_{i}=-1$. We use a binary encoding for all the features in the vector $\bbx_i$. For the age of the user we use the six age intervals $(0,12] $, $(12,18]$, $(18,24]$, $(24,30]$, $(30,40]$, and $(40,\infty)$ to construct six indicator entries in $\bbx_i$ that take the value 1 if the age of the user is known to be in the corresponding interval. E.g., a 21 year old user has an age that falls in the third interval which implies that we make $[\bbx_i]_3=1$ and $[\bbx_i]_k=0$ for all other $k$ between 1 and 6. If the age of the user is unknown we make $[\bbx_i]_k=0$ for all $k$ between 1 and 6. For the gender of the visitors we use the next three components of $\bbx_i$ to indicate male, female, or unknown gender. For a male user we make $[\bbx_i]_7=1$, for a female user $[\bbx_i]_8=1$, and for visitors of unknown gender we make $[\bbx_i]_9=1$. The next three components of $\bbx_{i}$ are used for the depth feature. If the the number of advertisements displayed in the search page is $1$ we make $[\bbx_{i}]_{10}=1$, if $2$ different ads are shown we make $[\bbx_{i}]_{11}=1$, and for depths of $3$ or more we make $[\bbx_{i}]_{12}=1$. To indicate the position of the ad in the search page we also use three components of $\bbx_{i}$. We use $[\bbx_{i}]_{13}=1$, $[\bbx_{i}]_{14}=1$, and $[\bbx_{i}]_{15}=1$ to indicate that the ad is displayed in the first, second, and third position, respectively. Likewise we use $[\bbx_{i}]_{16}$, $[\bbx_{i}]_{17}$ and $[\bbx_{i}]_{18}$ to indicate that the impression of the ad is $1$, $2$ or more than $3$.

For the words that appear in the query we have in the order of $10^5$ distinct words. To reduce the number of elements necessary for this encoding we create 20,000 bags of words through random hashing with each bag containing 5 or 6 distinct words. Each of these bags is assigned an index $k$. For each of the words in the query we find the bag in which this word appears. If the word appears in the $k$th bag we indicate this occurrence by setting the $k+18$th component of the feature vector to $[\bbx_i]_{k+18}=1$. Observe that since we use 20,000 bags, components 19 through 20,018 of $\bbx_i$ indicate the presence of specific words in the query. Further note that we may have more than one $\bbx_i$ component different from zero because there may be many words in the query, but that the total number of nonzero elements is much smaller than 20,000. On average, $3.0$ of these elements of the feature vector are nonzero. The same bags of words are used to encode the words that appear in the title of the ad and the product keywords. We encode the words that appear in the title of the ad by using the next $20,000$ components of vector $\bbx_{i}$, i.e. components $20,019$ through $40,018$. Components $40,019$ through $60,018$ are used to encode product keywords. As in the case of the words in the search just a few of these components are nonzero.  On average, the number of non-zero components of feature vectors that describe the title features is $8.8$. For product keywords the average is $2.1$. Since the number of distinct advertisers in the training set is $5,184$ we use feature components $60,019$ through $65202$ to encode this information. For the $k$th advertiser ID we set the $k+60,018{th}$ component of the feature vector to $[\bbx_{i}]_{k+60,018}=1$. Since the number of distinct advertisements is $108,824$ we allocate the last $108,824$ components of the feature vector to encode the ad ID. Observe that only one out of $5,184$ advertiser ID components and one of the $108,824$ advertisement ID components are nonzero.

In total, the length of the feature vector is 174,026 where each of the components are either $0$ or $1$. The vector is very sparse. We observe a maximum of 148 nonzero elements and an average of $20.9$ nonzero elements in the training set -- see Table \ref{table_summary_methods_properties}. This is important because the cost of implementing inner products in the oLBFGS training of the logistic regressor that we introduce in the following section is proportional to the number of nonzero elements in $\bbx_i$.

%
\subsection{Logistic regression of click-through rate}\label{sec:ctr_prediction}

We use the training set to estimate the CTR with a logistic regression as in, e.g., \cite{ZhangDuchi}. For that purpose let $\bbx\in\reals^n$ be a vector containing the features described in Section \ref{sec:dataset}, $\bbw\in\reals^n$ a classifier that we want to train, and $y\in{-1,1}$ an indicator variable that takes the value $y=1$ when the ad presented to the user is clicked and $y=-1$ when the ad is not clicked by the user. We hypothesize that the CTR, defined as the probability of observing $y=1$, can be written as the logistic function 
\begin{equation}\label{regression_problem}
   \text{CTR}(\bbx; \bbw) \ :=\P{y=1\given\bbx; \bbw} \ 
                =\  \frac{1}{1+\exp\big(-\bbx^{T}\bbw\big)}\ .
\end{equation}
We read \eqref{regression_problem} as stating that for a feature vector $\bbx$ the CTR is determined by the inner product $\bbx^{T}\bbw$ through the given logistic transformation. 

Consider now the training set $\ccalS = \{ (\bbx_{i},y_{i}) \}_{i=1}^{N}$ which contains $N$ realizations of features $\bbx_{i}$ and respective click outcomes $y_i$ and further define the sets $\ccalS_1:=\{(\bbx_i,y_i)\in\ccalS : y_i=1\}$ and $\ccalS_{-1}:=\{(\bbx_i,y_i)\in\ccalS : y_i=-1\}$ containing clicked and unclicked advertisements, respectively. With the data given in $\ccalS$ we define the optimal classifier $\bbw^{*}$ as a maximum likelihood estimate (MLE) of $\bbw$ given the model in \eqref{regression_problem} and the training set $\ccalS$. This MLE can be found as the minimizer of the log-likelihood loss 
\begin{align}\label{SVM_problem}
   \bbw^{*} 
   \ :=\ & \argmin \frac{\lambda}{2} \|\bbw\|^2\ +\ \frac{1}{N}
                   \sum_{i=1}^{N}\log\Big(1+\exp\big(-y_{i}\bbx_{i}^{T}\bbw\big)\Big)
                   \nonumber\\
   \  =\ & \argmin \frac{\lambda}{2} \|\bbw\|^2 \ +\ \frac{1}{N}
                   \bigg[\
                        \sum_{\bbx_i\in\ccalS_{1}}  \log\Big(1+\exp(-\bbx_{i}^{T}\bbw)\Big)	 
                      + \sum_{\bbx_i\in\ccalS_{-1}} \log\Big(1+\exp( \bbx_{i}^{T}\bbw)\Big)
                   \ \bigg],
\end{align}
where we have added the regularization term $\lambda\|\bbw\|^2/2$ to disincentivize large values in the weight vector $\bbw^*$; see e.g., \cite{Ng}.

The practical use of \eqref{regression_problem} and \eqref{SVM_problem} is as follows. We use the data collected in the training set $\ccalS$ to determine the vector $\bbw^*$ in \eqref{SVM_problem}. When a user issues a query we concatenate the user and query specific elements of the feature vector with the ad specific elements of several candidate ads. We then proceed to display the advertisement with, say, the largest CTR. We can interpret the set $\ccalS$ as having been acquired offline or online. In the former case we want to use a stochastic optimization algorithm because computing gradients is infeasible -- recall that we are considering training samples with a number of elements $N$ in the order of $10^6$. The performance metric of interest in this case is the logistic cost as a function of computational time. If elements of $\ccalS$ are acquired online we update $\bbw$ whenever a new vector becomes available so as to adapt to changes in preferences. In this case we want to exploit the information in new samples as much as possible. The correct metric in this case is the logistic cost as a function of the number of feature vectors processed. We use the latter metric for the numerical experiments in the following section.

%
\subsection{Numerical Results}\label{sec_ad_numerical_results}

Out of the $236 \times 10^6$ in the Tencent dataset we select $10^6$ sample points to use as the training set $\ccalS$ and $10^5$ sample points to use as a test set $\ccalT$. To select elements of the training and test set we divide the first $1.1\times 10^6$ sample points of the complete dataset in $10^5$ consecutive blocks with 11 elements. The first 10 elements of the block are assigned to the training set and the 11th element to the test set. To solve for the optimal classifier we implement SGD and oLBFGS by selecting feature vectors $\bbx_i$ at random from the training set $\ccalS$.  In all of our numerical experiments the regularization parameter in \eqref{SVM_problem} is $\lambda=10^{-6}$. The stepsizes for both algorithms are of the form $\epsilon_t=\epsilon_{0}T_{0} / (T_{0}+t)$. We set $\epsilon_{0} = 10^{-2}$ and $T_{0}=10^4$ for oLBFGS and $\epsilon_{0} = 10^{-1}$ and $T_{0}=10^6$ for SGD. For SGD the sample size in \eqref{stochastic_gradient} is set to $L=20$ whereas for oLBFGS it is set to $L=100$. The values of parameters $\eps_0$, $T_0$, and $L$ are chosen to yield best convergence times in a rough parameter optimization search. {Observe the relatively large values of $L$ that are used to compute stochastic gradients. This is necessary due to the extreme sparsity of the feature vectors $\bbx_i$ that contain an average of only $20.9$ nonzero out 174,026 elements. Even when considering $L=100$ vectors they are close to orthogonal.} The size of memory for oLBFGS is set to $\tau=10$. With $L=100$ features with an average sparsity of $20.9$ nonzero elements and memory $\tau=10$ the cost of each LBGS iteration is in the order of $2.1\times 10^4$ operations. 

%
\begin{figure}\centering
\includegraphics[width=0.7\linewidth,height=0.34\linewidth]{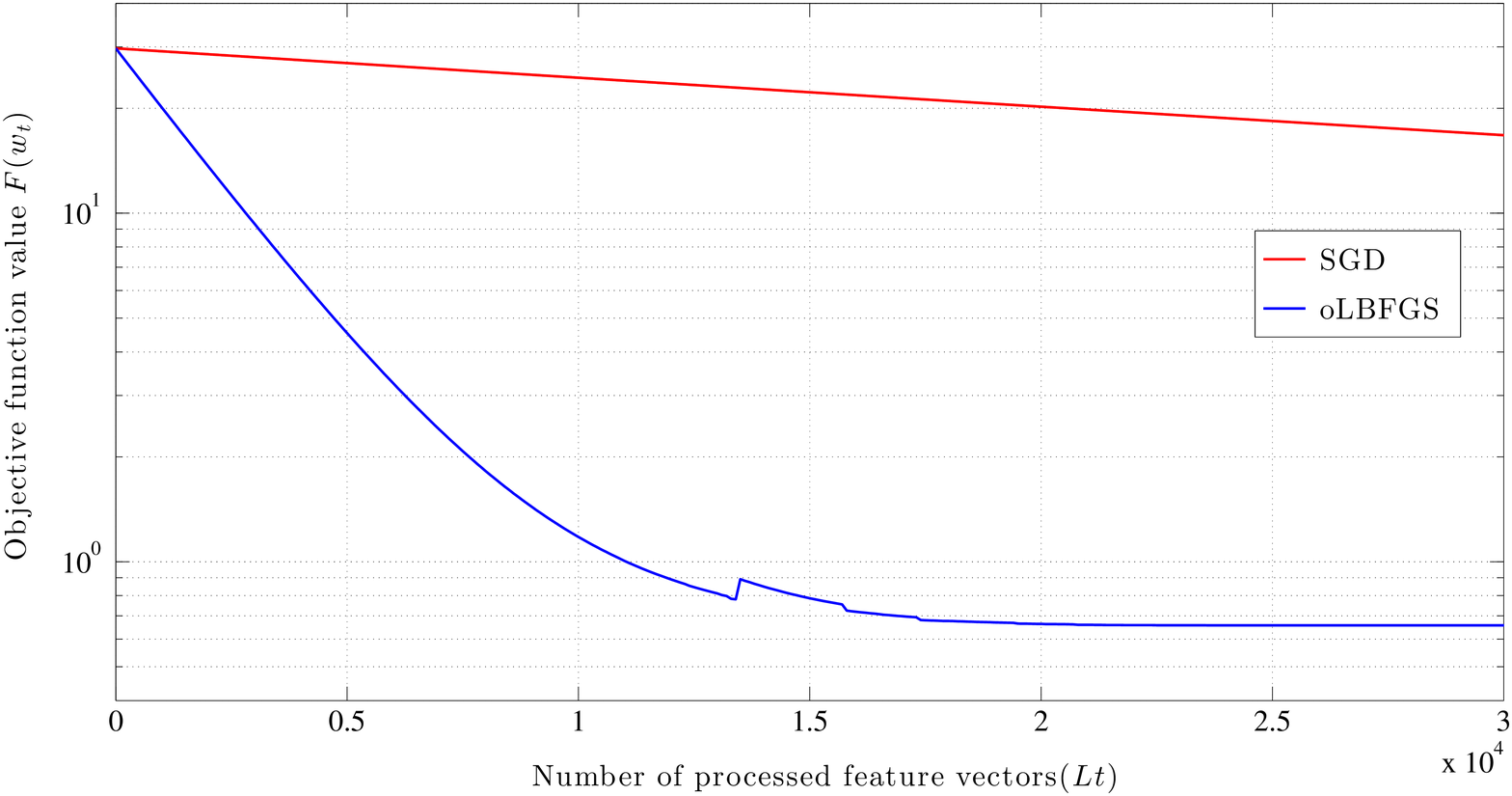}
\caption{Illustration of Negative log-likelihood value for oLBFGS and SGD after processing certain amount of feature vectors. The accuracy of oLBFGS  is better than SGD after processing a specific number of feature vectors.}
\label{fig:illustration} \end{figure}
%

Figure \ref{fig:illustration} illustrates the convergence path of SGD and oLBFGS on the advertising training set. We depict the value of the log likelihood objective in  \eqref{SVM_problem} evaluated at $\bbw=\bbw_t$ where $\bbw_t$ is the classifier iterate determined by SGD or oLBFGS. The horizontal axis is scaled by the number of feature vectors $L$ that are used in the evaluation of stochastic gradients. This results in a plot of log likelihood cost versus the number $Lt$ of feature vectors processed. To read iteration indexes from  Figure \ref{fig:illustration} divide the horizontal axis values by $L=100$ for oLBGS and $L=20$ for SGD. Consistent with the synthetic data results in Section \ref{sec:SVMproblem}, the curvature correction of oLBFGS results in significant reductions in convergence time. For way of illustration observe that after processing $Lt=3\times10^4$ feature vectors the objective value achieved by oLBFGS is $F(\bbw_t) =0.65$, while for SGD it still stands at $F(\bbw_t) =16$ which is a meager reduction from the random initialization point at which $F(\bbw_0) =30$. In fact, oLBFGS converges to the minimum possible log likelihood cost $F(\bbw_t) =0.65$ after processing $1.7\times10^4$ feature vectors. This illustration hints that oLBGS makes better use of the information available in feature vectors.

{To corroborate that the advantage of oLBGS is not just an artifact of the structure of the  log likelihood cost in \eqref{SVM_problem} we process $2\times10^4$ feature vectors with SGD and oLBFGS and evaluate the predictive accuracy of the respective classifiers on the test set.} As measures of predictive accuracy we adopt the frequency histogram of the predicted click through rate $\text{CTR}(\bbx;\bbw)$ for all clicked ads and the frequency histogram of the complementary predicted click through rate $1 - \text{CTR}(\bbx;\bbw)$ for all the ads that were {\it not} clicked. {To do so we separate the test set by defining the set $\ccalT_1:=\{(\bbx_i,y_i)\in\ccalT : y_i=1\}$ of clicked ads and the set $\ccalT_{-1}:=\{(\bbx_i,y_i)\in\ccalT : y_i=-1\}$ of ads in the test set that were not clicked. For a given classifier $\bbw$ we compute the predicted probability $\text{CTR}(\bbx_i;\bbw)$ for each of the ads in the clicked set $\ccalT_1$. We then consider a given interval $[a,b]$ and define the frequency histogram of the predicted click through rate as the fraction of clicked ads for which the prediction $\text{CTR}(\bbx_i;\bbw)$ falls in $[a,b]$, 
\begin{equation}\label{eqn_clicked_ads_histogram}
   \ccalH_1(\bbw; a,b) :=  \frac{1}{\#(\ccalT_1)} 
                           \sum_{(\bbx_i, y_i)\in\ccalT_{1}} 
                           \mbI \Big\{ \text{CTR}(\bbx_i;\bbw) \in[a,b] \Big\},
\end{equation}
where $\#(\ccalT_1)$ denotes the cardinality of the set $\ccalT_1$. Likewise, we consider the ads in the set $\ccalT_{-1}$ that were not clicked and compute the prediction $1 - \text{CTR}(\bbx_i;\bbw)$ on the probability of the ad not being clicked. We then consider a given interval $[a,b]$ and define the frequency histogram $\ccalH_{-1}(\bbw; a,b)$ as the fraction of unclicked ads for which the prediction $1-\text{CTR}(\bbx_i;\bbw)$ falls in $[a,b]$, 
\begin{equation}\label{eqn_unclicked_ads_histogram}
   \ccalH_{-1}(\bbw; a,b) :=  \frac{1}{\#(\ccalT_{-1})} 
                           \sum_{(\bbx_i, y_i)\in\ccalT_{-1}} 
                           \mbI \Big\{1- \text{CTR}(\bbx_i;\bbw) \in[a,b] \Big\}.
\end{equation}
The histogram $\ccalH_1(\bbw; a,b)$ in \eqref{eqn_clicked_ads_histogram} allows us to study how large the predicted probability $\text{CTR}(\bbx_i;\bbw)$ is for the clicked ads. Conversely, the histogram $\ccalH_{-1}(\bbw; a,b)$ in \eqref{eqn_unclicked_ads_histogram} gives an indication of how large the predicted probability $1- \text{CTR}(\bbx_i;\bbw)$ is for the unclicked ads. An ideal classifier is one for which the frequency counts in $\ccalH_1(\bbw; a,b)$ accumulate at $\text{CTR}(\bbx_i;\bbw) =1$ and for which $\ccalH_{-1}(\bbw; a,b)$ accumulates observations at $1- \text{CTR}(\bbx_i;\bbw) = 1$. This corresponds to a classifier that predicts a click probability of 1 for all ads that were clicked and a click probability of 0 for all ads that were not clicked.\\}

%
\begin{figure}\centering
\subfigure[{Histogram $\ccalH_1(\bbw; a,b)$, [cf. \eqref{eqn_clicked_ads_histogram}].}]{
\includegraphics[width=0.48\linewidth]
{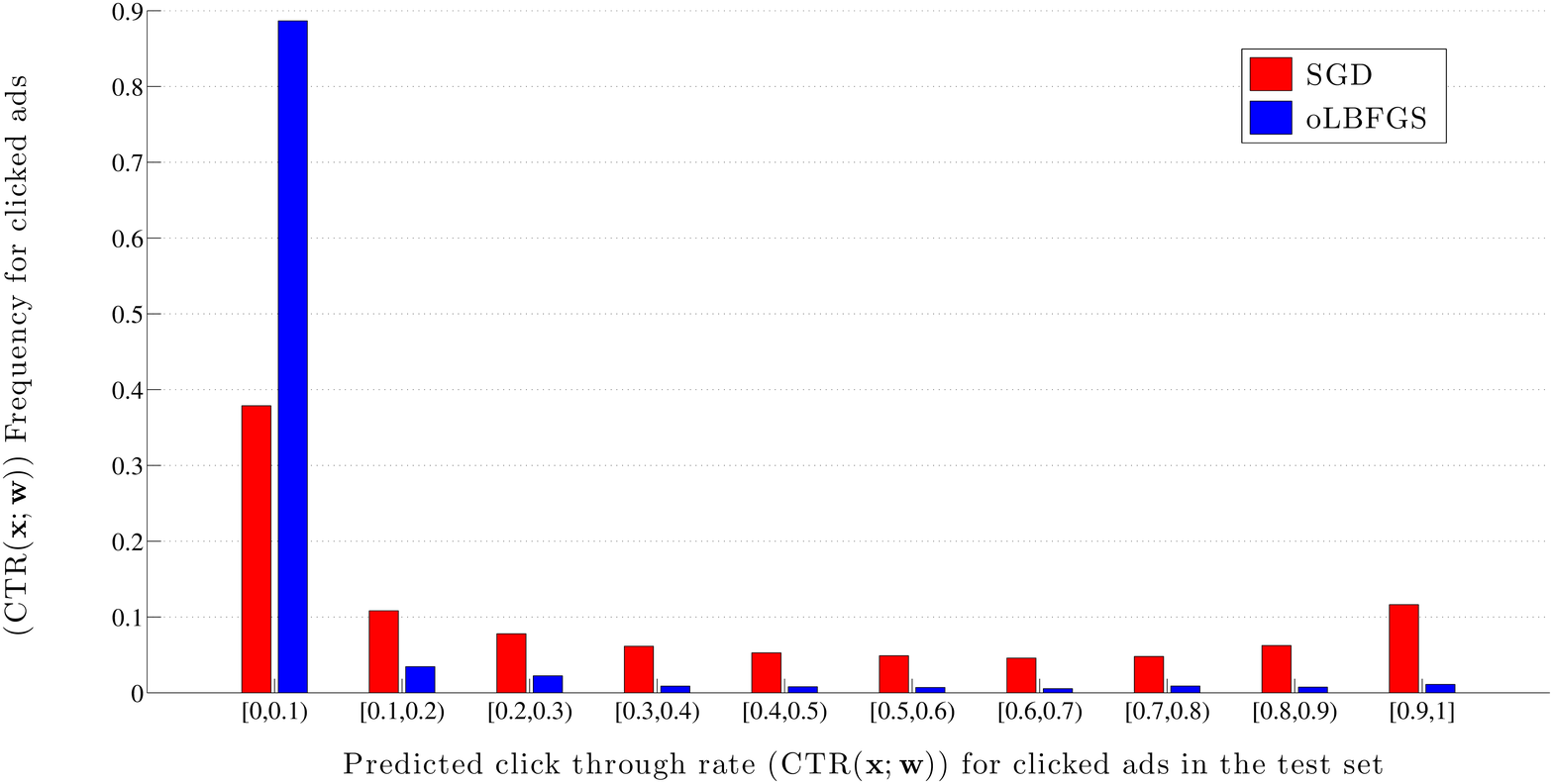}\label{fig:subfig1}}
\subfigure[{Histogram $\ccalH_{-1}(\bbw; a,b)$, [cf. \eqref{eqn_unclicked_ads_histogram}].}]{
\includegraphics[width=0.48\linewidth]
{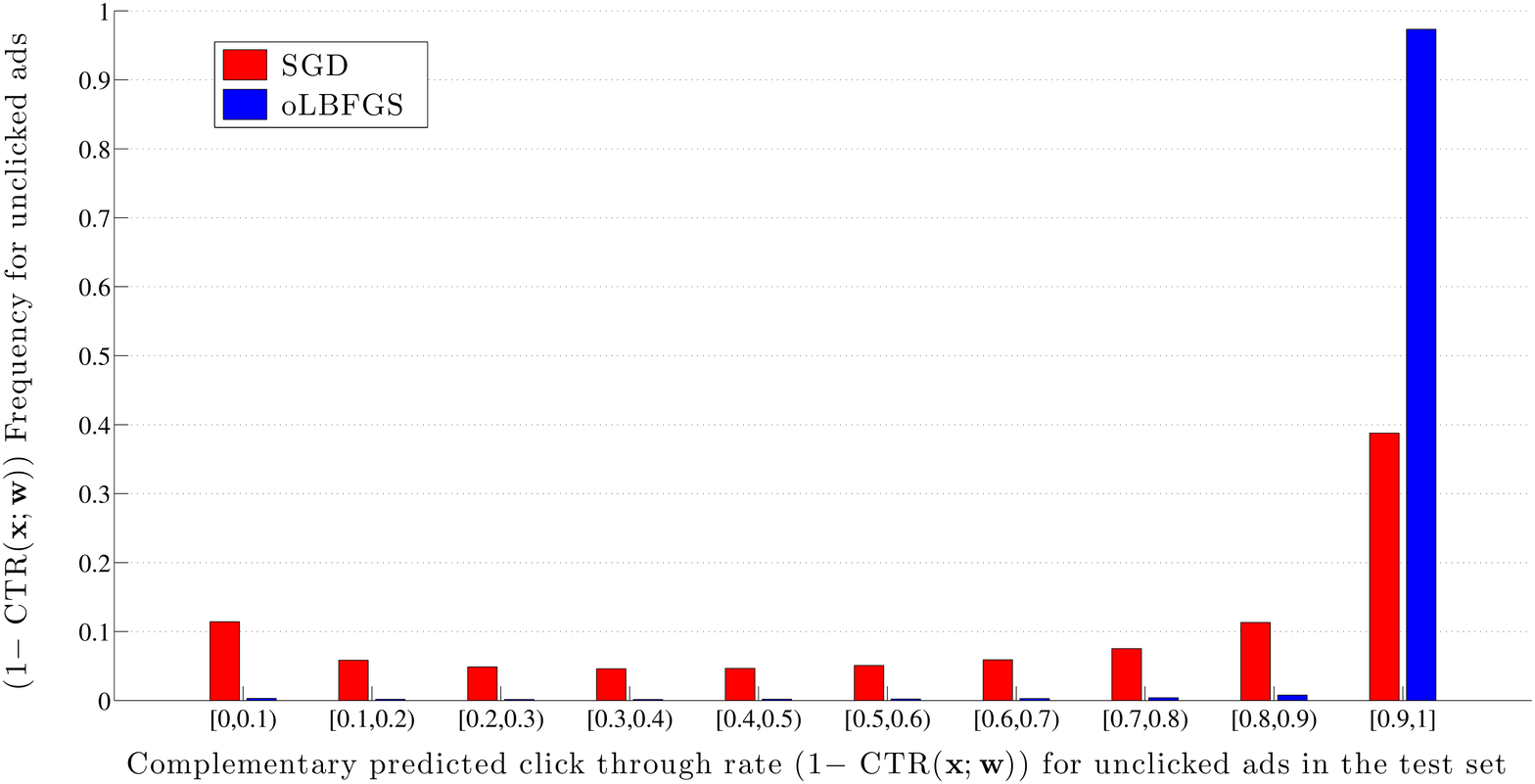}\label{fig:subfig2}}
\caption{Performance of classifier after processing $2\times10^4$ feature vectors with SGD and oLBFGS  for the cost in \eqref{SVM_problem}. Histograms for: (a) predicted click through rate $\text{CTR}(\bbx;\bbw)$ for all clicked ads; and (b) complementary predicted click through rate $1 - \text{CTR}(\bbx;\bbw)$ for all unclicked ads. For an ideal classifier that predicts a click probability $\text{CTR}(\bbx;\bbw) =1$ for all clicked ads and a click probability $\text{CTR}(\bbx;\bbw) =0$ for all unclicked ads the frequency counts in $\ccalH_1(\bbw; a,b)$ and $\ccalH_{-1}(\bbw; a,b)$ would accumulate in the $[0.9,1]$ bin. Neither SGD nor oLBFGS compute acceptable classifiers because the number of clicked ads in the test set is very small and predicting $\text{CTR}(\bbx;\bbw)=0$ for all ads is close to the minimum of \eqref{SVM_problem}. }
\label{gholi}\end{figure}

%

{
 Fig. \ref{fig:subfig1} shows the histograms of predicted click through rate $\text{CTR}(\bbx;\bbw)$ for all clicked ads by oLBFGS and SGD classifiers after processing $2\times10^4$ training sample points. oLBFGS classifier for $88\%$ of test points in $\ccalT_{1}$ predicts $ \text{CTR}(\bbx;\bbw)$ in the interval $[0,0.1]$ and the classifier computed by SGD estimates the click through rate $\text{CTR}(\bbx;\bbw)$ in the same interval for $37\%$ of clicked ads in the test set. These numbers shows the inaccurate click through rate predictions of both classifiers for the test points with label $y=1$. Although, SGD and oLBFGS classifiers have catastrophic performances in predicting click through rate $ \text{CTR}(\bbx;\bbw)$ for the clicked ads  in the test set, they perform well in estimating complementary predicted click through rate $1- \text{CTR}(\bbx;\bbw)$ for the test points with label $y=-1$. This observation implied by Fig. \ref{fig:subfig2} which shows the histograms of complementary predicted click through rate $1- \text{CTR}(\bbx;\bbw)$ for all \textit{not} clicked ads by oLBFGS and SGD classifiers after processing $2\times10^4$ training sample points. As it shows after processing $2\times10^4$ sample points of the training set the predicted probability $1-\text{CTR}(\bbx;\bbw)$ by the SGD classifier for $38.8\%$ of the test points are in the interval $[0.9,1]$, while for the  classifier computed by oLBFGS $97.3\%$ of predicted probability $1-\text{CTR}(\bbx;\bbw)$ are in the interval $[0.9,1]$ which is a significant performance. \\ } 

The reason for the inaccurate predictions of both classifiers is that most elements in the training set $\ccalS$ are unclicked ads. Thus, the minimizer $\bbw^*$ of the log likelihood cost in \eqref{SVM_problem} is close to a classifier that predicts $\text{CTR}(\bbx; \bbw^*)\approx0$ for most ads. Indeed, out of the $10^6$ elements in the training set, $94.8\%$ of them have labels $y_i=-1$ and only the remaining $5.2\times10^4$ feature vectors correspond to clicked ads. To overcome this problem we replicate observations with labels $y_i=1$ to balance  the representation of both labels in the training set. Equivalently, we introduce a constant $\gamma$ and redefine the log likelihood objective in \eqref{SVM_problem} to give a larger weight to feature vectors that correspond to clicked ads,
\begin{align}\label{SVM_problem_beta}
   \bbw^{*} 
   \  =\ & \argmin \frac{\lambda}{2} \|\bbw\|^2 \ +\ \frac{1}{M}
                   \bigg[\ \gamma
                        \sum_{\bbx_i\in\ccalS_{1}}  \log\Big(1+\exp(-\bbx_{i}^{T}\bbw)\Big)	 
                      + \sum_{\bbx_i\in\ccalS_{-1}} \log\Big(1+\exp( \bbx_{i}^{T}\bbw)\Big)
                   \bigg],
\end{align}
where we defined $M:=\gamma\#(\ccalS_1) +\#(\ccalS_{-1})$ to account for the replication of clicked featured vectors that is implicit in \eqref{SVM_problem_beta}. To implement SGD and oLBFGS in the weighted log function in \eqref{SVM_problem_beta} we need to bias the random choice of feature vector so that vectors in $\ccalS_1$ are $\gamma$ times more likely to be selected than vectors in $\ccalS_2$. Although our justification to introduce $\gamma$ is to balance the types of feature vectors, $\gamma$ is just a tradeoff constant to increase the percentage of correct predictions for clicked ads -- which is close to zero in Figure \ref{gholi} -- at the cost of reducing the accuracy of correct predictions of unclicked ads  -- which is close to one in Figure \ref{gholi}. 

%
\begin{figure}\centering
\subfigure[{Histogram $\ccalH_1(\bbw; a,b)$, [cf. \eqref{eqn_clicked_ads_histogram}].}]{
\includegraphics[width=0.48\linewidth]
{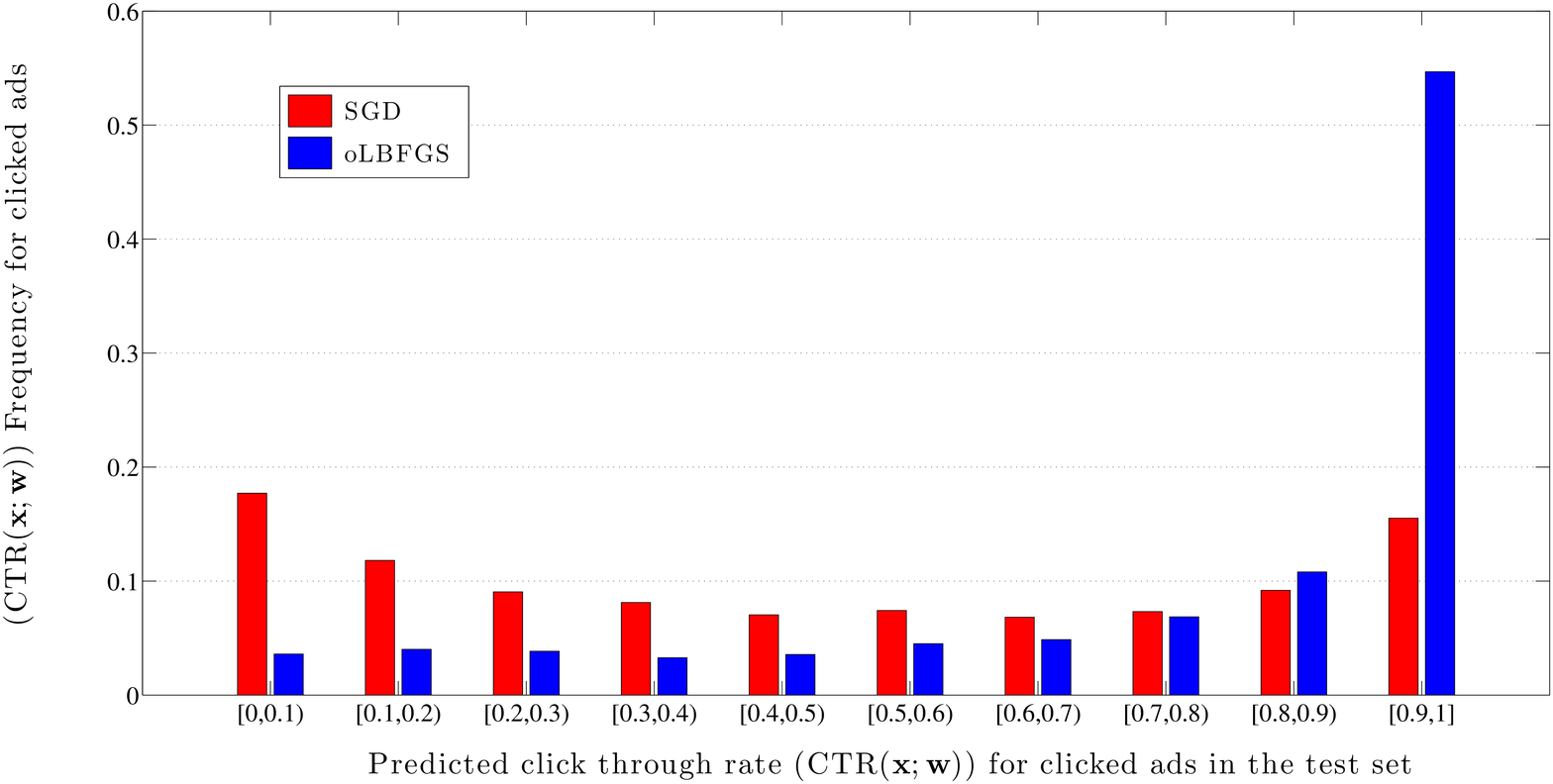}\label{fig:subfig111}}
\subfigure[{Histogram $\ccalH_{-1}(\bbw; a,b)$, [cf. \eqref{eqn_unclicked_ads_histogram}].}]{
\includegraphics[width=0.48\linewidth]
{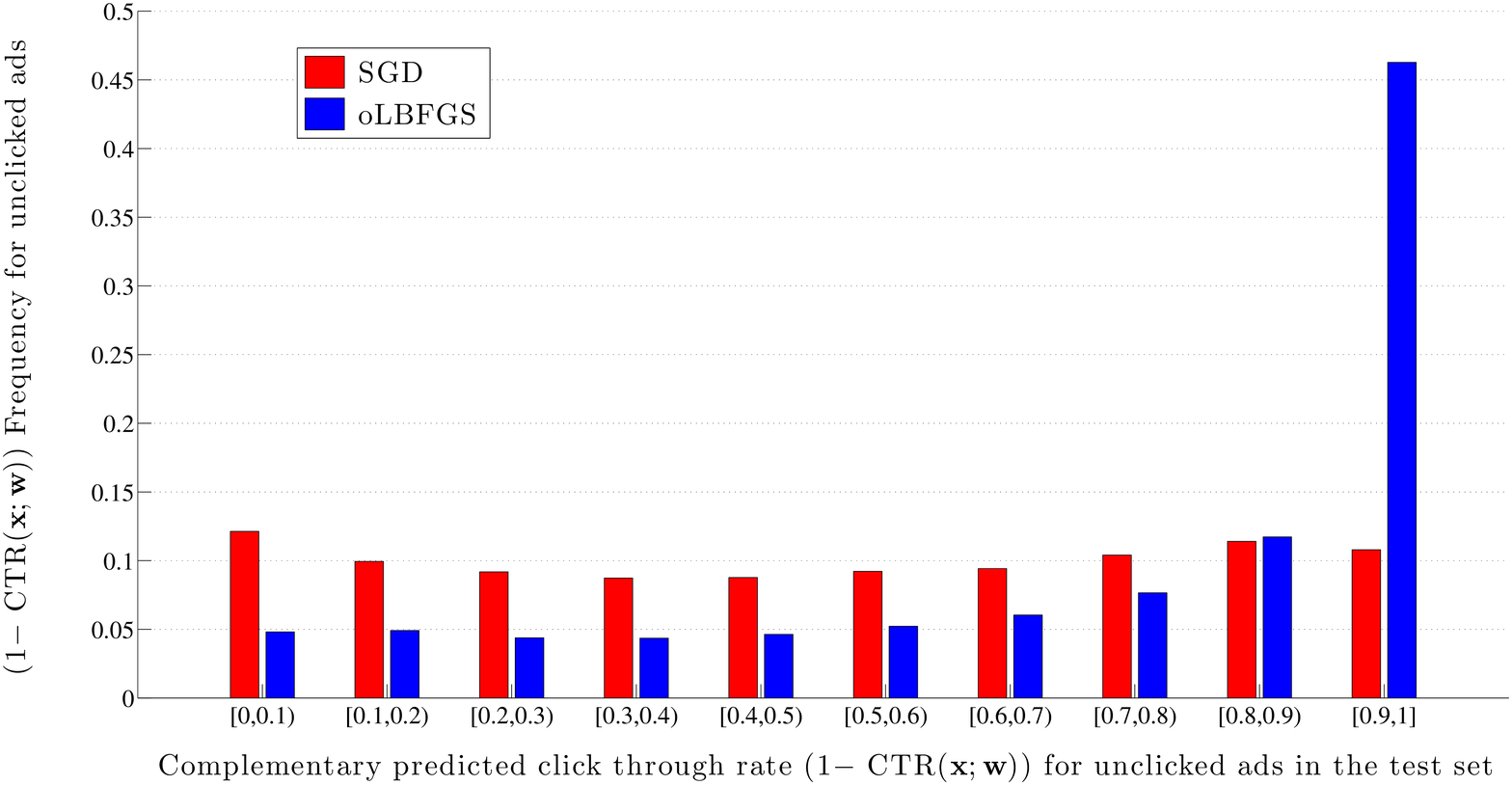}\label{fig:subfig222}}
\caption{Performance of classifier after processing $2\times10^4$ feature vectors with SGD and oLBFGS  for the cost in \eqref{SVM_problem_beta}. Histograms for: (a) predicted click through rate $\text{CTR}(\bbx;\bbw)$ for all clicked ads; and (b) complementary predicted click through rate $1 - \text{CTR}(\bbx;\bbw)$ for all unclicked ads. For an ideal classifier that predicts a click probability $\text{CTR}(\bbx;\bbw) =1$ for all clicked ads and a click probability $\text{CTR}(\bbx;\bbw) =0$ for all unclicked ads the frequency counts in $\ccalH_1(\bbw; a,b)$ and $\ccalH_{-1}(\bbw; a,b)$ would accumulate in the $[0.9,1]$ bin. The classifier computed by oLBFGS is much more accurate than the one computed by SGD. }
\label{fig:subfigureExample}\end{figure}

%
We repeat the experiment of processing $2\times10^4$ feature vectors that we sumamrized in Figure \ref{gholi} but now we use the objective cost in \eqref{SVM_problem_beta} instead of the cost in \eqref{SVM_problem}. We set $\gamma=18.2$ which makes replicated clicked ads as numerous as unclicked ads. The resulting SGD and oLBFGS histograms of the predicted click through rates for all clicked ads and complementary predicted click through rates for all unclicked ads are shown in Figure \ref{fig:subfigureExample}. In particular, Figure \ref{fig:subfig111} shows the histograms of predicted click through rate $\text{CTR}(\bbx;\bbw)$ for all clicked ads after processing $2\times10^4$ training sample points. The modification of the log likelihood cost increases the accuracy of the oLBFGS classifier which is now predicting a click probability $\text{CTR}(\bbx;\bbw)\in[0.9,1]$ for $54.7\%$ of the ads that were indeed clicked. There is also improvement for the SGD classifier but the prediction is much less impressive. Only $15.5\%$ of the clicked ads are associated with a click probability prediction in the interval $[0.9,1]$. This improvement is at the cost of reducing the complementary predicted click through rate $1- \text{CTR}(\bbx;\bbw)$ for the ads that were indeed not clicked. However, the classifier computed by oLBFGS after processing  $2\times10^4$ feature vectors still predicts a probability $1-\text{CTR}(\bbx;\bbw)\in[0.9,1]$  for $46.3\%$ of the unclicked ads. The corresponding frequency for the SGD classifier is $10.8\%$. 

Do note that the relatively high prediction accuracies in Figure \ref{fig:subfigureExample} are a reflection of sample bias to some extent. Since ads were chosen for display because they were deemed likely to be clicked they are not a completely random test set. Still, the point to be made here is that oLBFGS succeeds in finding an optimal classifier when SGD fails. It would take the processing of about $10^6$ feature vectors for SGD to achieve the same accuracy of oLBFGs.


\section{Conclusions}\label{sec_conclusions}

An online limited memory version of the (oL)BFGS algorithm was studied for solving strongly convex optimization problems with stochastic objectives. Almost sure convergence was established by bounding the traces and determinants of curvature estimation matrices under the assumption that sample functions have well behaved Hessians. The convergence rate of oLBFGS was further determined to be at least of order $O(1/t)$ in expectation. This rate is customary of stochastic optimization algorithms which are limited by their ability to smooth out the noise in stochastic gradient estimates. The application of oLBFGS to support vector machines was also developed and numerical tests on synthetic data were provided. The numerical results show that oLBFGS affords important reductions with respect to stochastic gradient descent (SGD) in terms of the number of feature vectors that need to be processed to achieve a target accuracy as well as in the associated execution time. Moreover, oLBFGS also exhibits a significant execution time reduction when compared to other stochastic quasi-Newton methods. These reductions increase with the problem dimension and can become arbitrarily large. A detailed comparison between oLBFGS and SGD for training a logistic regressor in a large scale search engine advertising problem was also presented. The numerical tests show that oLBFGS trains the regressor using less than $1\%$ of the data required by SGD to obtain similar classification accuracy.
\acks{We acknowledge the support of the National Science Foundation (NSF CAREER
CCF-0952867) and the Office of Naval Research (ONR N00014-12-1-0997).}

\appendix

%
\section{Proof of Proposition 1}\label{apx_lbfgs_reduced_cost}

We begin by observing that the $\bbp_u$ sequence in \eqref{eqn_define_pt} is defined so that we can write $\bbp_{u+1} = \hbZ_{t-u-1}\bbp_{u}$ with $\bbp_0=\bbp$. Indeed, use the explicit expression for $\hbZ_{t-u-1}$ in \eqref{Z_rho_definitions} to write the product $\hbZ_{t-u-1}\bbp_{u}$ as
\begin{equation}\label{eqn_multiply_Z}
   \hbZ_{t-u-1}\bbp_{u}
      \ =\ \Big(\bbI - \hat{\rho}_{t-u-1} \hbr_{t-u-1} \bbv_{t-u-1}^{T} \Big)\bbp_{u}
      \ =\ \bbp_u  - \alpha_{u}\hbr_{t-u-1}
      \ =\ \bbp_{u+1},
\end{equation}
where the second equality follows from the definition $\alpha_{u}:=\hat{\rho}_{t-u-1}\bbv_{t-u-1}^{T}\bbp_u$ and the third equality from the definition of the $\bbp_u$ sequence in \eqref{eqn_define_pt}.

Recall now the oLBFGS Hessian inverse approximation expression in \eqref{SLBFGS_update}. It follows that for computing the product $\hbB_{t}^{-1}\bbp$ we can multiply each of the $\tau+1$ summands in the right hand side of \eqref{SLBFGS_update} by $\bbp=\bbp_0$. Implementing this procedure yields 
\begin{align}\label{SLBFGS_update_11}
\hbB_{t}^{-1}\bbp
		= &   \left(\!\hbZ_{t-1}^{T} {\small\dots} \hbZ_{t-\tau}^{T}\!\right) 
                  \!\hbB_{t,0}^{-1}\! 
            \left(\!\hbZ_{t-\tau}{\small \dots}\hbZ_{t-1}\!\right) \bbp_0   
          + \hat{ \rho}_{t-\tau} 
            \left(\!\hbZ_{t-1}^{T} {\small \dots} \hbZ_{t-\tau+1}^{T}\!\right) 
            \bbv_{t-\tau} \bbv_{t-\tau}^{T}
            \left(\!\hbZ_{t-\tau+1}{\small \dots} \hbZ_{t-1}\!\right)\bbp_0      \nonumber \\
       &  + \dots 
          + \hat{\rho}_{t-2} \left( \hbZ_{t-1}^{T}\right) 
            \bbv_{t-2}\bbv_{t-2}^{T}
            \left( \hbZ_{t-1}\right)\bbp_0  
         +  \hat{\rho}_{t-1} 
            \bbv_{t-1}\bbv_{t-1}^{T}\bbp_0.
\end{align}
The fundamental observation in \eqref{SLBFGS_update_11} is that all summands except the last contain the product $\hbZ_{t-1}\bbp_0$. This product cannot only be computed efficiently but, as shown in \eqref{eqn_multiply_Z}, is given by $\bbp_1=\hbZ_{t-1}\bbp_0$. A not so fundamental, yet still important observation, is that the last term can be simplified to $\hat{\rho}_{t-1}   \bbv_{t-1}\bbv_{t-1}^{T}\bbp_0= \alpha_{0}\bbv_{t-1}$ given the definition of $\alpha_0:=\hat{\rho}_{t-1} \bbv_{t-1}^{T}\bbp_0$. Implementing both of these substitutions in \eqref{SLBFGS_update_11} yields
\begin{align}\label{SLBFGS_update_22}
\hbB_{t}^{-1}\bbp
		= &   \left(\!\hbZ_{t-1}^{T} {\small\dots} \hbZ_{t-\tau}^{T}\!\right) 
                  \!\hbB_{t,0}^{-1}\! 
            \left(\!\hbZ_{t-\tau}{\small \dots}\hbZ_{t-2}\!\right) \bbp_1   
          + \hat{ \rho}_{t-\tau} 
            \left(\!\hbZ_{t-1}^{T} {\small \dots} \hbZ_{t-\tau+1}^{T}\!\right) 
            \bbv_{t-\tau} \bbv_{t-\tau}^{T}
            \left(\!\hbZ_{t-\tau+1}{\small \dots} \hbZ_{t-2}\!\right)\bbp_1      \nonumber \\
       &  + \dots 
          + \hat{\rho}_{t-2} \left( \hbZ_{t-1}^{T}\right) 
            \bbv_{t-2}\bbv_{t-2}^{T}
            \bbp_1  
         +  \alpha_0 \bbv_{t-1}.
\end{align}
The structure of \eqref{SLBFGS_update_22} is analogous to the structure of \eqref{SLBFGS_update_11}. In all terms except the last two we require determination of the product $\hbZ_{t-2}\bbp_1$, which, as per \eqref{eqn_multiply_Z} can be computed with $2n$ multiplications and is given by $\bbp_2 = \hbZ_{t-2}\bbp_1$. Likewise, in the second to last term we can simplify the product $\hat{\rho}_{t-2}\bbv_{t-2}\bbv_{t-2}^{T}\bbp_1 = \alpha_1\bbv_{t-2}$ using the definition $\alpha_1=\hat{\rho}_{t-2}\bbv_{t-2}^{T}\bbp_1$. Implementing these substitutions in \eqref{SLBFGS_update_22} yields an expression that is, again, analogous. In all of the resulting summands except the last three we need to compute the product $\hbZ_{t-3}\bbp_2$, which is given by $\bbp_3 = \hbZ_{t-3}\bbp_2$ and in the third to last term we can simplify the product $\hat{\rho}_{t-3}\bbv_{t-3}\bbv_{t-3}^{T}\bbp_2 = \alpha_2\bbv_{t-3}$. Repeating this process keeps yielding terms with analogous structure and, after $\tau-1$ repetitions we simplify \eqref{SLBFGS_update_22} to
\begin{align}\label{decsent_update202}
 \hbB_{t}^{-1}\bbp= & \left(\!\hbZ_{t-1}^{T} {\small\dots}   \hbZ_{t-\tau+1}^{T}\hbZ_{t-\tau}^{T}\!\right) 
             {\hbB_{t,0}}^{-1} \bbp_{\tau } + \left(\!\hbZ_{t-1}^{T} {\small\dots}  \hbZ_{t-\tau+1}^{T}\!\right) 
               \alpha_{\tau-1} \bbv_{t-\tau}  + {\small\dots} +   \hbZ_{t-1}^{T}
                \alpha_{1}  \bbv_{t-2} + \alpha_{0}\bbv_{t-1}.
\end{align}
In the first summand in \eqref{decsent_update202} we can substitute the definition of the first element of the $\bbq_u$ sequence $\bbq_{0}:={\hbB_{t,0}}^{-1} \bbp_{\tau }$. More important, observe that the matrix $\hbZ_{t-1}^{T}$ is the first factor in all but the last summand. Likewise, the matrix $\hbZ_{t-2}^{T}$ is the second factor in all but the last two summands and, in general, the matrix $\hbZ_{t-u}^{T}$ is the $u$th factor in all but the last $u$ summands. Pulling these common factors recursively through \eqref{decsent_update202} it follows that $\hbB_{t}^{-1}\bbp_{t}$ can be equivalently written as
\begin{align}\label{eqn_inner_product_simplification_30}
   &\hbB_{t}^{-1}\bbp 
       =                             \alpha_{0} \bbv_{t-1}
         + \hbZ_{t-1}^{T}     \Bigg[ \alpha_{1} \bbv_{t-2}  
         + \hbZ_{t-2}^{T}     \bigg[ \ldots
                              \Big [ \alpha_{\tau-2}\bbv_{t-\tau+1} 
         + \hbZ_{t-\tau+1}^{T}\big [ \alpha_{\tau-1}\bbv_{t-\tau} 
         + \hbZ_{t-\tau}^{T}\bbq_{0}\big] \Big] \ldots\bigg]\Bigg] .
\end{align}
To conclude the proof we just need to note that the recursive definition of $\bbq_u$ in \eqref{eqn_define_qt} is a computation of the nested elements of \eqref{eqn_inner_product_simplification_30}. To see this consider the innermost element of \eqref{eqn_inner_product_simplification_30} and use the definition of $\beta_{0}:=\hat{\rho}_{t-\tau}\hbr_{t-\tau}^{T}\bbq_0$ to conclude that $\alpha_{\tau-1}\bbv_{t-\tau} + \hbZ_{t-\tau}^{T}\bbq_0$ is given by
\begin{equation}\label{something_useful}
\alpha_{\tau-1}\bbv_{t-\tau} + \hbZ_{t-\tau}^{T}\bbq_0  
	\ =\ \alpha_{\tau-1}\bbv_{t-\tau}  
	     + \bbq_0- \hat{\rho}_{t-\tau}\bbv_{t-\tau}\hbr_{t-\tau}^{T}\bbq_0
	\ =\ \bbq_0 + ( \alpha_{\tau-1} -  \beta_{0})\bbv_{t-\tau}
	\ =\ \bbq_1	
\end{equation}
where in the last equality we use the definition of $\bbq_1$ [cf. \eqref{eqn_define_qt}. Substituting this simplification into \eqref{eqn_inner_product_simplification_30} eliminates the innermost nested term and leads to 
\begin{align}\label{eqn_inner_product_simplification_40}
   &\hbB_{t}^{-1}\bbp 
       =                             \alpha_{0} \bbv_{t-1}
         + \hbZ_{t-1}^{T}     \Bigg[ \alpha_{1} \bbv_{t-2}  
         + \hbZ_{t-2}^{T}     \bigg[ \ldots
                              \Big [ \alpha_{\tau-2}\bbv_{t-\tau+1} 
         + \hbZ_{t-\tau+1}^{T}\bbq_1 \Big] \ldots\bigg]\Bigg] .
\end{align}
Mimicking the computations in \eqref{something_useful} we can see that the innermost term in \eqref{eqn_inner_product_simplification_40} is $\alpha_{\tau-2}\bbv_{t-\tau+1}  + \hbZ_{t-\tau+1}^{T}\bbq_1 = \bbq_2$ and obtain an analogous expression that we can substitute for $\bbq_3$ and so on. Repeating this process $\tau-2$ times leads to the last term being $\hbB_{t}^{-1}\bbp  =  \alpha_{0} \bbv_{t-1}+ \hbZ_{t-1}^{T}\bbq_{\tau-1}$ which we can write as $\alpha_{0} \bbv_{t-1}+ \hbZ_{t-1}^{T}\bbq_{\tau-1} = \bbq_\tau$ by repeating the operations in \eqref{something_useful}. This final observation yields $\hbB_{t}^{-1}\bbp =\bbq_\tau$.

%
\section{Proof of Lemma \ref{lecce}}\label{appx_proof_lemma_1}

As per \eqref{hassan} in Assumption 1 the eigenvalues of the instantaneous Hessian $\hat{ \bbH}(\bbw,{\tbtheta})$ are bounded by $\tdm$ and $\tdM$. Thus, for any given vector $\bbz$ it holds
\begin{equation}\label{liverpool}
\tdm \|\bbz\|^{2} \leq \bbz^{T} \hat{ \bbH}(\bbw,{\tbtheta}) \bbz \leq \tdM \|\bbz\|^{2}.
\end{equation}
For given $\bbw_{t}$ and $\bbw_{t+1}$ define the mean instantaneous Hessian $\hbG_{t}$ as the average Hessian value along the segment $[\bbw_{t},\bbw_{t+1}]$
\begin{equation}\label{average}
\hbG_{t}= \int_{0}^{1}  \hat{ \bbH} \left( \bbw_{t}+\tau(\bbw_{t+1}-\bbw_{t}),{\tbtheta_{t}}\right) d\tau.
\end{equation}
Consider now the instantaneous gradient $\hbs( \bbw_{t}+\tau(\bbw_{t+1}-\bbw_{t}),\tbtheta_{t})$ evaluated at $\bbw_{t}+\tau(\bbw_{t+1}-\bbw_{t})$ and observe that its derivative with respect to $\tau$ is $\partial\hbs\big( \bbw_{t}+\tau(\bbw_{t+1}-\bbw_{t}),\tbtheta_{t}\big)/\partial \tau = \hbH(\bbw_{t}+\tau( \bbw_{t+1} - \bbw_{t}),{\tbtheta_{t}})  (\bbw_{t+1} - \bbw_{t})$. Then according to the fundamental theorem of calculus 
\begin{equation}\label{lazio}
   \int_{0}^{1} \hbH \left( \bbw_{t}+\tau(\bbw_{t+1}-\bbw_{t})\ \! , \ \! {\tbtheta_{t}} \right)  
                        (\bbw_{t+1} - \bbw_{t})\ d\tau  \
      = \
       \hbs(\bbw_{t+1},\tbtheta_{t}) - \hbs(\bbw_{t},\tbtheta_{t}) .
\end{equation}
Using the definitions of the mean instantaneous Hessian $\hbG_{t}$ in \eqref{average} as well as the definitions of the stochastic gradient variations $\hbr_{t}$ and variable variations $\bbv_{t}$ in \eqref{chris} and \eqref{ball} we can rewrite \eqref{lazio} as
\begin{equation}\label{mancity}
 \hbG_{t} \bbv_{t} = \hbr_{t}.
\end{equation}
Invoking \eqref{liverpool} for the integrand in \eqref{average}, i.e., for $\hat{ \bbH}(\bbw,{\tbtheta}) = \hat{ \bbH}\big{(}\bbw_{t}+\tau(\bbw_{t+1}-\bbw_{t}),{\tbtheta} \big{)}$, it follows that for all vectors $\bbz$ the mean instantaneous Hessian $\hbG_{t}$ satisfies
\begin{equation}\label{arsenal}
   \tdm \|\bbz\|^{2} \leq \bbz^{T} \hbG_{t} \bbz \leq \tdM \|\bbz\|^{2}.
\end{equation}
{The claim in \eqref{claim233} follows from \eqref{mancity} and \eqref{arsenal}.} Indeed, consider the ratio of inner products $\hbr_{t}^{T}\bbv_{t}/\bbv_{t}^{T}\bbv_{t}$ and use \eqref{mancity} and the first inequality in \eqref{arsenal} to write
\begin{equation}\label{jose}
   \frac{\hbr_{t}^{T}\bbv_{t}}{\bbv_{t}^{T}\bbv_{t}}
       = \frac{{\bbv_{t}^{T} \hbG_{t} \bbv_{t} }}{\bbv_{t}^{T} \bbv_{t}} \geq \tdm.
\end{equation}
%
It follows that \eqref{claim233} is true for all times $t$. 

To prove \eqref{claim444} we operate \eqref{mancity} and \eqref{arsenal}. Considering the ratio of inner products ${\hbr_{t}^{T}\hbr_{t}}/{\hbr_{t}^{T}\bbv_{t}} $ and observing that \eqref{mancity} states $\hbG_{t}\bbv_{t}=\hbr_{t}$, we can write
\begin{equation}\label{fajr}
\frac{\hbr_{t}^{T}\hbr_{t}}{\hbr_{t}^{T}\bbv_{t}} = \frac{\bbv_{t}^{T}\hbG_{t}^{2}\bbv_{t}}{\bbv_{t}^{T}\hbG_{t}\bbv_{t}} 
\end{equation} 
Since the mean instantaneous Hessian $\hbG_{t}$ is positive definite according to \eqref{arsenal}, we can define $\bbz_{t}=\hbG_{t}^{1/2}\bbv_{t}$. Substituting this observation into \eqref{fajr} we can conclude
\begin{equation}\label{fajr2}
\frac{\hbr_{t}^{T}\hbr_{t}}{\hbr_{t}^{T}\bbv_{t}} = \frac{\bbz_{t}^{T}\hbG_{t}\bbz_{t}}{\bbz_{t}^{T}\bbz_{t}} .
\end{equation} 
Observing \eqref{fajr2} and the inequalities in \eqref{arsenal}, it follows that \eqref{claim444} is true. 

%
\section{Proof of Lemma \ref{lemma_determinant_and_trace_bounds}}\label{appx_lemma3}


We begin with the trace upper bound in \eqref{trace_bound_1}. Consider the recursive update formula for the Hessian approximation $\hbB_{t}$ as defined in \eqref{Hessian_appro_update}. To simplify notation we {define $s$ as a new index such that $s=t-\tau+u$. Introduce this simplified notation in \eqref{Hessian_appro_update} and compute the trace of both sides. Since traces are linear function of their arguments we obtain
\begin{equation}\label{trace_of_akbar}
   \tr\left(\hbB_{t,u+1}\right) =\tr\left( \hbB_{t,u}\right )
   - \tr\left({{\hbB_{t,u} \bbv_{s}\bbv_{s}^{T}{\hbB_{t,u}} }\over{\bbv_{s}^{T}\hbB_{t,u}\bbv_{s}}}\right) 
   + \tr\left( {{\hbr_{s}\hbr_{s}^{T}}\over{\bbv_{s}^{T}\hbr_{s}}}
\right).
\end{equation}
Recall that the trace of a matrix product is independent of the order of the factors to conclude that the second summand of \eqref{trace_of_akbar} can be simplified to 
\begin{equation}\label{trace_of_akbar_10}
   \tr\left( \hbB_{t,u} \bbv_{s}\bbv_{s}^{T}\hbB_{t,u}\right)
     \ =\ \tr\left(\bbv_{s}^{T}\hbB_{t,u} \hbB_{t,u} \bbv_{s}\right)
     \ =\ \bbv_{s}^{T}\hbB_{t,u} \hbB_{t,u} \bbv_{s}
     \ =\ \left\| \hbB_{t,u} \bbv_{s}\right\|^{2},
\end{equation}
where the second equality follows because $\bbv_{s}^{T}\hbB_{t,u} \hbB_{t,u} \bbv_{s}$ is a scalar and the second equality by observing that the term $\bbv_{s}^{T}\hbB_{t,u} \hbB_{t,u} \bbv_{s}$ is the inner product of the vector $\hbB_{t,u} \bbv_{s}$ with itself. Use the same procedure for  the last summand of \eqref{trace_of_akbar} so as to write $\tr(\hbr_{s}\hbr_{s}^{T})=\hbr_{s}^{T}\hbr_{s}=\|\hbr_{s}\|^{2}$. Substituting this latter observation as well as \eqref{trace_of_akbar_10} into \eqref{trace_of_akbar} we can simplify the trace of $\hbB_{t,u+1}$ to
\begin{equation}\label{trace_BFGS}
   \tr\left(\hbB_{t,u+1}\right) =\tr\left( \hbB_{t,u}\right )
      -{{\|\hbB_{t,u}\bbv_{s}\|^{2}}\over {\bbv_{s}^{T} \hbB_{t,u} \bbv_{s}}}
   + {{\| \hbr_{s} \|^{2}} \over {\hbr_{s}^{T}\bbv_{s}}}     .
\end{equation}
The second term in the right hand side of \eqref{trace_BFGS} is negative because, as we have already shown, the matrix $\hbB_{t,u}$ is positive definite. The third term is the one for which we have derived the bound that appears in \eqref{claim444} of Lemma \ref{lecce}. Using this two observations we can conclude that the trace of  $\hbB_{t,u+1}$ can be bounded as
\begin{equation}\label{trace_upper_bound}
   \tr\left(\hbB_{t,u+1}\right) \leq \tr\left(\hbB_{t,u}\right) + \tdM.
\end{equation}
By considering \eqref{trace_upper_bound} as a recursive expression for $u=0,\ldots\tau-1$, we can conclude that 
\begin{align}\label{trace_upper_bound2}
   \tr\left(\hbB_{t,u}\right) &\leq \tr\left(\hbB_{t,0}\right) + u\tdM.
\end{align}
To finalize the proof of \eqref{trace_bound_1} we need to find a bound for the initial trace $\tr(\hbB_{t,0})$. To do so we consider the definition $\hbB_{t,0}= \bbI/\hat{\gamma}_{t}$ with $\hat{\gamma}_{t}$ as given by \eqref{initial_matrix_update_2}. Using this definition of $\hbB_{t,0}$ as a scaled identity it follows that we can write the trace of $\hbB_{t,0}$ as
\begin{equation}\label{initial_Hessian_bound_trace}
   \tr\left(\hbB_{t,0}\right) \ =\  \tr\left(\frac{\bbI}{\hat{\gamma}_{t}}\right) 
                   \ =\  \frac{n}{\hat{\gamma}_{t}}  .  
\end{equation}
Substituting the definition of $\hat{\gamma}_{t}$ into the rightmost side of \eqref{initial_matrix_update_2} it follows that for all times $t\geq 1$,
\begin{equation}\label{polomahi}
  \tr\left(\hbB_{t,0}\right) \ = \  n\, \frac{\hbr_{t-1}^{T}{\hbr_{t-1}}}{\bbv_{t-1}^{T}\hbr_{t-1}}\, 
                    = \  n\, \frac{\|\hbr_{t-1}\|^2}{\bbv_{t-1}^{T}\hbr_{t-1}}\, .
 \end{equation}
The term $\|\hbr_{t-1}\|^2/\bbv_{t-1}^{T}\hbr_{t-1}$ in \eqref{polomahi} is of the same form of the rightmost term in \eqref{trace_BFGS}. We can then, as we did in going from \eqref{trace_BFGS} to \eqref{trace_upper_bound} apply the bound that we provide in \eqref{claim444} of Lemma \ref{lecce} to conclude that for all times $t\geq1$
 \begin{equation}\label{asgharzadeh}
       \tr\left(\hbB_{t,0}\right)\     \leq\ n\tdM.
\end{equation}
Substituting \eqref{asgharzadeh} into \eqref{trace_upper_bound2} and pulling common factors leads to the conclusion that for all times $t\geq1$ and indices $0\leq u\leq \tau$ it holds
\begin{equation}\label{desktop}
   \tr\left(\hbB_{t,u}\right)\ \leq\  (n+u) \tdM .
\end{equation}
The bound in \eqref{trace_bound_1} follows by making $u=\tau$ in \eqref{desktop} and recalling that, by definition, $\hbB_{t} = \hbB_{t,\tau}$. For time $t=0$ we have $\hat{\gamma}_{t}=\hat{\gamma}_{0}=1$ and \eqref{polomahi} reduces to $\tr(\hbB_{t,0}) =  n$ while \eqref{desktop} reduces to $\tr(\hbB_{t,\tau})\leq(1+\tau) \tdM$. Furthermore, for $t<\tau$ we make $\hbB_{t} = \hbB_{t,t}$ instead of $\hbB_{t} = \hbB_{t,\tau}$. In this case the bound in \eqref{desktop} can be tightened to $\tr(\hbB_{t,\tau})\leq  (n+t) \tdM$. Given that we are interested in an asymptotic convergence analysis, these bounds are inconsequential.

We consider now the determinant lower bound in \eqref{det_bound_1}. As we did in \eqref{trace_of_akbar} begin by considering the recursive update in \eqref{Hessian_appro_update} and define $s$ as a new index such that $s=t-\tau+u$ to simplify notation. Compute the determinant of both sides of \eqref{Hessian_appro_update}, factorize $\hbB_{t,u}$ on the right hand side, and use the fact that the determinant of a product is the product of the determinants to conclude that
\begin{align}\label{det_akbar}
  \det\left( \hbB_{t,u+1}\right)  =
  \det \left(\hbB_{t,u}\right) \det \Bigg{(}\bbI-
 {{\bbv_{s}(\hbB_{t,u}\bbv_{s})^{T}}\over
 {\bbv_{s}^{T}\hbB_{t,u}\bbv_{s}}} + {\hbB_{t,u}^{-1}\hbr_{s} {\hbr_{s}^{T}}\over{\hbr_{s}^{T}\bbv_{s}}}
\Bigg{ )}.
\end{align}
To simplify the right hand side of \eqref{det_akbar} we should first know that for any vectors $\bbu_1$, $\bbu_2$, $\bbu_3$ and $\bbu_4$, we can write $\det (\bbI + \bbu_1 \bbu_2^{T}+\bbu_3 \bbu_4^{T})=(1+\bbu_1^{T}\bbu_2)(1+\bbu_3^{T}\bbu_4)-(\bbu_1^{T}\bbu_4)(\bbu_2^{T}\bbu_3)$ -- see, e.g., \cite{Li}, Lemma $3.3$). Setting $\bbu_1=\bbv_{s}$, $\bbu_2= {{\hbB_{t,u}\bbv_{s}}/
 {\bbv_{s}^{T}\hbB_{t,u}\bbv_{s}}} $, $\bbu_3=\hbB_{t,u}^{-1}\hbr_{s} $ and $\bbu_4= {{\hbr_{s}}/{\hbr_{s}^{T}\bbv_{s}}}$, implies that $\det (\bbI + \bbu_1 \bbu_2^{T}+\bbu_3 \bbu_4^{T})$ is equivalent to the last term in the right hand side of  \eqref{det_akbar}. Applying these substitutions implies that $(1+\bbu_1^{T}\bbu_2)= 1-\bbv_{s}^{T} {{\hbB_{t,u}\bbv_{s}}/{\bbv_{s}\hbB_{t,u}\bbv_{s}}}=0$ and $\bbu_1^{T}\bbu_4 =-\bbv_{s}^{T} {{\hbr_{s}}/{\hbr_{s}^{T}\bbv_{s}}}=-1 $. Hence, the term $\det (\bbI + \bbu_1 \bbu_2^{T}+\bbu_3 \bbu_4^{T})$ can be simplified as $\bbu_{2}^{T}\bbu_{3}$. By this simplification we can write the right hand side of  \eqref{det_akbar} as
\begin{align}\label{det_kave2}
&\det\! \left[\bbI-
 {{\bbv_{s}(\hbB_{t,u}\bbv_{s})^{T}}\over
 {\bbv_{s}^{T}\hbB_{t,u}\bbv_{s}}} +  {\hbB_{t,u}^{-1}\hbr_{s}{\hbr_{s}^{T}}\over{\hbr_{s}^{T}\bbv_{s}}}
 \right]=   {{\left(\hbB_{t,u}\bbv_{s}\right)^{T}}\over
 {\bbv_{s}^{T}\hbB_{t,u}\bbv_{s}}} 
 \hbB_{t,u}^{-1}\hbr_{s}  .
\end{align}
To further simplify \eqref{det_kave2} write $(\hbB_{t,u}\bbv_{s})^{T} = \bbv_{s}^T\hbB_{t,u}^T$ and observer that since $\hbB_{t,u}$ is symmetric we have $\hbB_{t,u}^T \hbB_{t,u}^{-1} = \hbB_{t,u} \hbB_{t,u}^{-1} = \bbI$. Therefore,  
\begin{equation}\label{det_kave}
  \det\! \left[\bbI-
 {{\bbv_{s}(\hbB_{t,u}\bbv_{s})^{T}}\over
 {\bbv_{i}^{T}\hbB_{t,u}\bbv_{s}}} +  {\hbB_{t,u}^{-1}\hbr_{s}{\hbr_{s}^{T}}\over{\hbr_{s}^{T}\bbv_{s}}}
 \right]= 
  {{\hbr_{s}^{T}\bbv_{s}}\over 
{\bbv_{s}^{T}  \hbB_{t,u}  \bbv_{s}}}.
\end{equation}
Substitute the simplification in \eqref{det_kave} for the corresponding factor in \eqref{det_akbar}. Further multiply and divide the right hand side by the nonzero norm $\|\bbv_s\|$ and regroup terms to obtain
\begin{equation}\label{det_BFGS}
   \det\left(\hbB_{t,u+1}\right) = \det\left(\hbB_{t,u}\right) 
   \frac{\hbr_{s}^{T}\bbv_{s}}
        {\|\bbv_{s}\|}
   \frac{\|\bbv_{s}\|}
        {\bbv_{s}^{T}\hbB_{t,u}\bbv_{s}}.
\end{equation}
To bound the third factor in \eqref{det_BFGS} observe that the largest possible value for the normalized quadratic form ${\bbv_{s}^{T}\hbB_{t,u}\bbv_{s}}/\|\bbv_{s}\|^2$ occurs when $\bbv_{s}$ is an eigenvector of $\hbB_{t,u}$ associated with its largest eigenvalue. In such case the value attained is precisely the largest eigenvalue of $\hbB_{t,u}$ implying that we can write
\begin{equation}\label{eqn_det_bound_pf_70}
   \frac {\bbv_{s}^{T}\hbB_{t,u}\bbv_{s}}
         {\|\bbv_{s}\|}
   \leq \lambda_{\max}\left(\hbB_{t,u}\right).
\end{equation}
But to bound the largest eigenvalue $\lambda_{\max}(\hbB_{t,u})$ we can just use the fact that the trace of a matrix coincides with the sum of its eigenvalues. In particular, it must be that $\lambda_{\max}(\hbB_{t,u}) \leq \tr(\hbB_{t,u})$ because all the eigenvalues of the positive definite matrix $\hbB_{t,u}$ are positive. Combining this observation with the trace bound in \eqref{desktop} leads to
\begin{equation}\label{eqn_det_bound_pf_80}
   \frac {\bbv_{s}^{T}\hbB_{t,u}\bbv_{s}}
         {\|\bbv_{s}\|}
   \ \leq\ \tr\left(\hbB_{t,u}\right)
   \ \leq\  (n+u) \tdM .   
\end{equation}
We can also bound the second factor in the right hand side of \eqref{det_BFGS} if we reorder the inequality in \eqref{claim233} of Lemma \ref{lecce} to conclude that ${\hbr_{s}^{T}\bbv_{s}}/ {\|\bbv_{s}\|}\leq\tdm$. This bound, along with the inverse of the inequality in \eqref{eqn_det_bound_pf_80} substituted in \eqref{det_BFGS} leads to
\begin{align}\label{middle2}
   \det\left(\hbB_{t,u+1}\right) 
      \geq \frac{\tdm}{n\tdM+u\tdM}\det\left(\hbB_{t,u}\right) .
\end{align}
Apply \eqref{middle2} recursively between indexes $u=0$ and $u=\tau-1$ and further observing that $u\leq\tau$ in all of the resulting factors it follows that
\begin{align}\label{det_BFGS_bound}
   \det\left(\hbB_{t,\tau}\right)\  \geq\  
    \left[\frac{\tdm}{( n+ \tau) \tdM}\right]^{\tau}
    \det\left(\hbB_{t,0}\right) .
\end{align}
To finalize the derivation of \eqref{det_bound_1} we just need to bound the determinant of the initial curvature approximation matrix $\hbB_{t,0}$. To do so we consider, again, the definition $\hbB_{t,0}= \bbI/\hat{\gamma}_{t}$ with $\hat{\gamma}_{t}$ as given by \eqref{initial_matrix_update_2}. Using this definition of $\hbB_{t,0}$ as a scaled identity it follows that we can write the determinant of $\hbB_{t,0}$ as
\begin{equation}\label{initial_Hessian_bound_determinant}
   \det\left(\hbB_{t,0}\right) \ =\  \det\left(\frac{\bbI}{\hat{\gamma}_{t}}\right) 
                   \ =\  \frac{1}{\hat{\gamma}_{t}^n}  .  
\end{equation}
Substituting the definition of $\hat{\gamma}_{t}$ into the rightmost side of \eqref{initial_Hessian_bound_determinant} it follows that for all times $t\geq 1$,
\begin{equation}\label{polomahi}
   \det\left(\hbB_{t,0}\right) 
      \ = \ \left(\frac{\hbr_{t-1}^{T}{\hbr_{t-1}}}{\bbv_{t-1}^{T}\hbr_{t-1}}\right)^n
      \ = \ \left(\frac{\|\hbr_{t-1}\|^2}{\bbv_{t-1}^{T}\hbr_{t-1}}\right)^n .
 \end{equation}
The term $\|\hbr_{t-1}\|^2/\bbv_{t-1}^{T}\hbr_{t-1}$ has lower and upper bounds that we provide in \eqref{claim444} of Lemma \ref{lecce}. Using the lower bound in \eqref{claim444} it follows that the initial determinant must be such that
\begin{equation}\label{topoloo}
   \det\left(\hbB_{t,0}\right) \geq \tdm ^{n} .
\end{equation}
Substituting the upper bound in \eqref{topoloo} for the determinant of the initial curvature approximation matrix in \eqref{det_BFGS_bound} allows us to conclude that for all times $t\geq1$
\begin{align}\label{det_BFGS_bound2}
   \det\left(\hbB_{t,\tau}\right)\  \geq\  
    \tdm^n \left[\frac{\tdm}{( n+ \tau) \tdM}\right]^{\tau} .
\end{align}
The bound in \eqref{det_bound_1} follows by making $u=\tau$ in \eqref{det_BFGS_bound2} and recalling that, by definition, $\hbB_{t} = \hbB_{t,\tau}$. At time $t=0$ the initialization constant is set to $\hat{\gamma}_{t}=\hat{\gamma}_{0}=1$ and \eqref{topoloo} reduces to $\det(\hbB_{t,0}) =  1$ while \eqref{det_BFGS_bound2} reduces to $\det(\hbB_{t,\tau})\leq[\tdm/( 1+ \tau) \tdM]^{\tau}$. For $t<\tau$ we make $\hbB_{t} = \hbB_{t,t}$ instead of $\hbB_{t} = \hbB_{t,\tau}$. In this case the bound in \eqref{desktop} can be tightened to $\det(\hbB_{t,\tau})\leq\tdm[\tdm^n/( 1+ \tau) \tdM]^{\tau}$. As in the case of the trace, given that we are interested in an asymptotic convergence analysis, these bounds are inconsequential.

%
\section{Proof of Lemma \ref{norooz}}\label{appx_lemma4}

We first prove the upper bound inequality in \eqref{claim888}. Let us define $\lambda_{i}$ as the $i$th largest eigenvalue of matrix $\hbB_{t}$. Considering the result in Lemma \ref{lemma_determinant_and_trace_bounds} that $\tr(\hbB_{t})\leq(n+\tau) \tdM$ for all steps $t\geq1$, we obtain that the sum of eigenvalues of the Hessian approximation $\hbB_{t}$ satisfy
\begin{equation}\label{sum_eig_bound1}
   \sum_{i=1}^{n} \lambda_{i} \ =\ \tr\left(\hbB_{t}\right) \ \leq\ (n+\tau) \tdM.
\end{equation}
Considering the upper bound for the sum of eigenvalues in \eqref{sum_eig_bound1} and recalling that all the eigenvalues of the matrix $\hbB_{t}$ are positive because $\hbB_{t}$ is positive definite, we can conclude that each of the eigenvalues of $\hbB_{t}$ is less than the upper bound for their sum in \eqref{sum_eig_bound1}. We then have $\lambda_{i}\leq(n+\tau) \tdM$ for all $i$ from where the right inequality in \eqref{claim888} follows.

To prove the lower bound inequality in \eqref{claim888} consider the second result of Lemma \ref{lemma_determinant_and_trace_bounds} which provides a lower bound for the determinant of the Hessian approximation matrix $\hbB_{t}$. According to the fact that determinant of a matrix is the product of its eigenvalues, it follows that the product of the eigenvalues of $\hbB_{t}$ is bounded below by the lower bound in \eqref{det_bound_1}, or, equivalently, $\prod_{i=1}^{n} \lambda_{i} \geq {\tdm^{n+\tau}}/{[( n+ \tau) \tdM]^{\tau}} $. Hence, for any given eigenvalue of $\hbB_{t}$, say $\lambda_{j}$, we have
\begin{equation}\label{eigenvalues_product}
\lambda_{j}\ \geq\ 
	\frac{1}{\prod_{k=1,k\neq j}^{n} \lambda_{k} } \times \frac{\tdm^{n+\tau}}{\left[( n+ \tau) \tdM\right]^{\tau}}.
\end{equation}
But in the first part of this proof we have already showed that ${(n +\tau) \tdM}$ is a lower bound for the eigenvalues of $\hbB_{t}$. We can then conclude that the product of the $n-1$ eigenvalues $\prod_{k=1,k\neq j}^{n} \lambda_{k}$ is bounded above by $[{(n +\tau) \tdM}]^{n-1}$, i.e.,
\begin{equation}\label{khastam}
\prod_{k=1,k\neq j}^{n} \lambda_{k} \leq  \left[{(n +\tau) \tdM}\right]^{n-1}.
\end{equation} 
Combining the inequalities in \eqref{eigenvalues_product} and \eqref{khastam} we conclude that for any specific eigenvalue of $\hbB_{t}$ can be lower bounded as
\begin{equation}\label{lower_bound_for_eigenvalues}
\lambda_{j}\ \geq\ \frac{1}{ \left[{(n +\tau) \tdM}\right]^{n-1}} \times  \frac{\tdm^{n+\tau}}{\left[( n+ \tau) \tdM\right]^{\tau}}.
\end{equation}
Since inequality \eqref{lower_bound_for_eigenvalues} is true for all the eigenvalues of  $\hbB_{t}$, the left inequality \eqref{claim888} holds true.

\section{Proof of Lemma \ref{helpful}}\label{appx_lemma5}

The proof is standard in stochastic optimization and provided here for reference. As it follows from Assumption 1 the eigenvalues of the Hessian $\bbH(\bbw_{t})=\mbE_{\tbtheta}[\hbH(\bbw_{t},\tbtheta_{t})] = \nabla_{\bbw}^{2} F(\bbw_{t})$ are bounded between $0<m$ and $M<\infty$ as stated in \eqref{bbb}. Taking a Taylor's expansion of the function $F(\bbw)$ around $\bbw=\bbw_{t}$ and using the upper bound in the Hessian eigenvalues we can write
\begin{align}\label{taylor_upper_bound_22}
   F(\bbw_{t+1})\ \leq\ & F(\bbw_{t}) +\nabla F(\bbw_{t})^{T}(\bbw_{t+1}-\bbw_{t}) +  {{M}\over{2}}\|{\bbw_{t+1}-\bbw_{t}}\|^{2}.
\end{align}
From the definition of the oLBFGS update in \eqref{eqn_bfgs_descent} we can write the difference of two consecutive variables   $\bbw_{t+1}-\bbw_{t} $ as $ -\epsilon_{t}\hbB_{t}^{-1}  \hbs(\bbw_{t},\tbtheta_{t})$. Making this substitution in \eqref{taylor_upper_bound_22}, taking expectation with $\bbw_{t}$ given in both sides of the resulting inequality, and observing the fact that when $\bbw_{t}$ is given the Hessian approximation $\hbB_{t}^{-1}$ is deterministic we can write
\begin{equation}\label{pirloitalia}
   \E{F(\bbw_{t+1})\given \bbw_{t}}  \leq\  F(\bbw_{t})  
          - \epsilon_{t}\nabla F(\bbw_{t})^{T}\hbB_{t}^{-1}
                 \E{\hbs(\bbw_{t},\tbtheta_{t})\given\bbw_{t}} 
          + \frac{\epsilon^{2}M}{2} \ \!
              \E{\left\|\hbB_{t}^{-1}\hbs(\bbw_{t},\tbtheta_{t})\right\|^{2} 
                    \given  \bbw_{t}}. 
\end{equation}
We proceed to bound the third term in the right hand side of \eqref{pirloitalia}. Start by observing that the 2-norm of a product is not larger than the product of the 2-norms and that, as noted above, with $\bbw_{t}$ given the matrix $\hbB_{t}^{-1}$ is also given to write 
\begin{equation}\label{pedar_salavati}
    \E{\left\| \hbB_{t}^{-1}\hbs(\bbw_{t},\tbtheta_{t})\right\|^{2} 
        \given \bbw_{t}} 
    \ \leq\ \left\|\hbB_{t}^{-1} \right\|^{2}   \
           \E{\left\|  \hbs(\bbw_{t},\tbtheta_{t})\right\|^{2} \!\! \given \! \bbw_{t}}
\end{equation}
Notice that, as stated in \eqref{eqn_eigenvalue_critical_bounds}, $1/c$ is an upper bound for the eigenvalues of $\hbB_{t}^{-1}$. Further observe that the second moment of the norm of the stochastic gradient is bounded by $  \E{\| \hbs(\bbw_{t},\tbtheta_{t})\|^{2} \given \bbw_{t}} \leq S^{2}$, as stated in Assumption 2. These two upper bounds  substituted in \eqref{pedar_salavati} yield
\begin{equation}\label{dash_akol}
    \E{\left\| \hbB_{t}^{-1}\hbs(\bbw_{t},\tbtheta_{t})\right\|^{2} 
           \given \bbw_{t}}\leq \frac{S^2}{c^{2}}.
\end{equation}
Substituting the upper bound in \eqref{dash_akol} for the third term of \eqref{pirloitalia} and further using the fact that $ \E{\hbs(\bbw_{t},\tbtheta_{t})\given\bbw_{t}}=\nabla F(\bbw_{t})$ in the second term leads to
\begin{align}\label{lavashak}
    \E{F(\bbw_{t+1})\given \bbw_{t}}  
          \leq  F(\bbw_{t})   -  \epsilon_{t} \nabla F(\bbw_{t})^{T} \hbB_{t}^{-1}      \nabla F(\bbw_{t}) 
          +\frac{\epsilon_{t}^{2}MS^{2}}{{2c^{2}}}  .
\end{align}
We now find a lower bound for the second term in the right hand side of \eqref{lavashak}. As stated in \eqref{eqn_eigenvalue_critical_bounds}, $1/C$ is a lower bound for the eigenvalues of  $\hbB_{t}^{-1}  $. This lower bound implies that
\begin{equation}\label{inner_product_lower_bound}
   \nabla F(\bbw_{t})^{T} \hbB_{t}^{-1}\nabla F(\bbw_{t}) 
       \geq \frac{1}{C} \|\nabla F(\bbw_{t}) \|^{2} 
\end{equation}  
By substituting the lower bound in \eqref{inner_product_lower_bound} for the corresponding summand in \eqref{lavashak} the result in \eqref{pedarsag} follows.
%
%


\section{Proof of Theorem \ref{convg}}\label{appx_theorem_6}


The proof uses the relationship in the statement \eqref{pedarsag} of Lemma \ref{helpful} to build a supermartingale sequence. This is also a standard technique in stochastic optimization and provided here for reference. To construct the supermartingale sequence define the stochastic process $\alpha_t$ with values 
\begin{equation}\label{alp}
   \alpha_t := F(\bbw_{t})+ \frac{MS^{2}}{{2c^{2}}} \sum_{u=t}^{\infty}  {{\epsilon_{u}^{2}}}.
\end{equation}
Observe that $\alpha_t$ is well defined because the $\sum_{u=t}^{\infty}  {{\epsilon_{u}^{2}}}<\sum_{u=0}^{\infty}  {{\epsilon_{u}^{2}}}<\infty$ is summable. Further define the sequence $\beta_t$ with values
\begin{equation}\label{beta}
   \beta_t :=\  \frac{\epsilon_{t}}{C}\ \| \nabla F(\bbw_{t})\|^{2} .
\end{equation}
Let now $\ccalF_t$ be a sigma-algebra measuring $\alpha_t$, $\beta_t$, and $\bbw_t$. The conditional expectation of $\alpha_{t+1}$ given $\ccalF_t$ can be written as
\begin{equation}\label{eqn_theo_convg_pf_30}
   \E{\alpha_{t+1} \given \ccalF_t} 
       =  \E{F(\bbw_{t+1}) \given \ccalF_t}
             + \frac{MS^{2}}{{2c^{2}}}  \sum_{u=t+1}^{\infty}  {{\epsilon_{u}^{2}}},
\end{equation}
because the term $({MS^{2}}/{{2c^{2}}} ) \sum_{u=t+1}^{\infty}  {{\epsilon_{u}^{2}}}$ is just a deterministic constant. Substituting \eqref{pedarsag} of Lemma \ref{helpful} into \eqref{eqn_theo_convg_pf_30} and using the definitions of $\alpha_t$ in \eqref{alp} and $\beta_t$ in \eqref{beta} yields
\begin{equation}\label{martin}
   \E{\alpha_{t+1} \given \alpha_t} \ \leq\ \alpha_t - \beta_t
\end{equation}
Since the sequences $\alpha_t$ and $\beta_t$ are nonnegative it follows from \eqref{martin} that they satisfy the conditions of the supermartingale convergence theorem -- see e.g. (Theorem E$7.4$ in \cite{Solo}) . Therefore, we conclude that: (i) The sequence $\alpha_t$ converges almost surely. (ii) The sum $\sum_{t=0}^{\infty}\beta_t < \infty$ is almost surely finite. Using the explicit form of $\beta_t$ in \eqref{beta} we have that $\sum_{t=0}^{\infty}\beta_t < \infty$ is equivalent to
\begin{equation}\label{psg}
   \sum_{t=0}^{\infty}\frac{\epsilon_{t}}{C}\ \| \nabla F(\bbw_{t})\|^{2} < \infty, 
       \qquad\text{a.s.}
\end{equation}
Since the sequence of stepsizes is nonsummable, for \eqref{psg} to be true we need to have a vanishing subsequence embedded in $\| \nabla F(\bbw_{t})\|^{2}$. By definition, this implies that the limit infimum of the sequence $\| \nabla F(\bbw_{t})\|^{2}$ is null almost surely,
\begin{equation}\label{pppp}
   \liminf_{t \to \infty}   \| \nabla F(\bbw_{t})\|^{2}  =  0,  \qquad\text{a.s.}
\end{equation}
To transform the gradient bound in \eqref{pppp} into a bound pertaining to the squared distance to optimality  $\| \bbw_{t} -\bbw^{*} \|^2$ simply observe that the lower bound $m$ on the eigenvalues of $\textbf{H}(\bbw_{t})$ applied to a Taylor's expansion around the optimal argument $\bbw^*$ implies that
\begin{equation}\label{lower_bound_taylor_expansion}
F(\bbw^*)\ \geq\ F(\bbw_{t})+  \nabla F(\bbw_{t})^{T} (\bbw^*-\bbw_{t}) 
		+\ \frac{m}{2} \|\bbw^*-\bbw_{t}\|^2 .
\end{equation}
Observe now that since $\bbw^*$ is the minimizing argument of $F(\bbw)$ we must have  $F(\bbw^*) -\ F(\bbw_{t}) \leq 0$ for all $\bbw$. Using this fact and reordering terms we simplify \eqref{lower_bound_taylor_expansion} to
\begin{equation}\label{eqn_theo_convg_pf_80}
 \frac{m}{2}\  \|\bbw^*-\bbw_{t}\|^2\  \leq\ \nabla F(\bbw_{t})^{T} (\bbw_t-\bbw^*) .
\end{equation}   
Further observe that the Cauchy-Schwarz inequality implies that $\nabla F(\bbw_{t})^{T} (\bbw_{t} - \bbw^*)\leq\|\nabla F(\bbw_{t})\|\|\bbw_{t}-\bbw^{*}\|$. Substitution of this bound in \eqref{eqn_theo_convg_pf_80} and simplification of a $\|\bbw^*-\bbw_{t}\|$ factor yields
\begin{equation}\label{qqqq}
   \frac{m}{2}\| \bbw_{t} -\bbw^{*} \|\  \leq\   \|\nabla F(\bbw_{t})\|.
\end{equation}
Since the limit infimum of $\|\nabla F(\bbw_{t})\|$ is null as stated in \eqref{pppp} the result in \eqref{eqn_convg} follows from considering the bound in \eqref{qqqq} in the limit as the iteration index $t\to\infty$. 
%


\section{Proof of Theorem \ref{theo_convergence_rate}}\label{appx_theorem_7}

The proof follows along the lines of (\cite{AryanAleTSP}) and is presented here for completeness. Theorem \ref{theo_convergence_rate} claims that the sequence of expected objective values $\E {F(\bbw_{t})}$ approaches the optimal objective $F(\bbw^*)$ at a linear rate $O(1/t)$. Before proceeding to the proof of Theorem \ref{theo_convergence_rate} we repeat a technical lemma of (\cite{AryanAleTSP}) that provides a sufficient condition for a sequence $u_t$ to exhibit a linear convergence rate.

%
\begin{lemma}[\cite{AryanAleTSP}]\label{lecce22}
Let $a>1$, $b>0$ and $t_0 > 0$ be given constants and $u_{t}\geq 0$ be a nonnegative sequence that satisfies the inequality
\begin{equation}\label{claim23}
   u_{t+1} \leq \left( 1- \frac{a}{t+t_0} \right) u_{t} 
                 + \frac{b}{{(t+t_0)}^{2}}\ ,
\end{equation}
{for all times $t\geq0$}. The sequence $u_t$ is then bounded as
\begin{equation}\label{lemma3_claim}
u_{t} \leq\  \frac{Q}{t+t_{0}},
\end{equation}
for all times $t\geq0$, where the constant $Q$ is defined as
\begin{equation}\label{convergence_parameter}
   Q:=\max \left[\frac{b}{a-1},\ t_{0} u_{0}  \right] .
\end{equation}
\end{lemma}

%
\begin{proof}
We prove \eqref{lemma3_claim} using induction. To prove the claim for {$t=0$} simply observe that the definition of $Q$ in \eqref{convergence_parameter} implies that
\begin{equation}\label{eqn_lecce_pf_10}
   Q:=\max \left[\frac{b}{a-1},\ t_{0} u_{0}  \right] \geq\ t_{0} u_{0},
\end{equation}
because the maximum of two numbers is at least equal to both of them. By rearranging the terms in \eqref{eqn_lecce_pf_10} we can conclude that 
\begin{equation}\label{first_step_of_induction}
    u_0\ \leq\  \frac{Q}{t_{0}}.
\end{equation}
Comparing \eqref{first_step_of_induction} and \eqref{lemma3_claim} it follows that the latter inequality is true for $t=0$.

Introduce now the induction hypothesis that \eqref{lemma3_claim} is true for $t=s$. To show that this implies that \eqref{lemma3_claim} is also true for $t=s+1$ substitute the induction hypothesis $u_s\leq Q/(s+t_0)$ into the recursive relationship in \eqref{claim23}. This substitution shows that $u_{s+1}$ is bounded as
\begin{equation}\label{hassan_hassan}
    u_{s+1} \leq 
        \left( 1- \frac{a}{s+t_0} \right) \frac{Q}{s+t_{0}}
        +\frac{b}{{(s+t_0)}^{2}} \ .
\end{equation}
Observe now that according to the definition of $Q$ in \eqref{convergence_parameter}, we know that $b/(a-1) \leq Q$ because $Q$ is the maximum of $b/(a-1)$ and $t_{0}u_{0}$. Reorder this bound to show that $b\leq Q(a-1)$ and substitute into \eqref{hassan_hassan} to write
\begin{align}\label{eqn_lecce_pf_40}
   u_{s+1} \leq 
        \left( 1- \frac{a}{s+t_0} \right) \frac{Q}{s+t_{0}}  
        +\frac{(a-1) Q}{{(s+t_0)}^{2}}  \ .
\end{align}
Pulling out $Q/(s + t_0)^2$ as a common factor and simplifying and reordering terms it follows that \eqref{eqn_lecce_pf_40} is equivalent to
\begin{align}\label{jhn}
   u_{s+1} 
       \ \leq\  \frac{Q \big[s+t_0 -a +(a-1)\big]}{{(s+t_0)}^{2}} 
       \ =   \  \frac{s+t_0-1}{{(s+t_0)}^{2}} \ \! Q .
\end{align}
To complete the induction step use the difference of squares formula for $(s+t_0)^2 - 1$ to conclude that 
\begin{equation}\label{algebra}
   \big[ (s+t_0) - 1\big]     \big[ (s+t_0) + 1\big]  
       \ =   \ (s+t_0)^2 - 1 
       \ \leq\ (s+t_0)^2.
\end{equation}
Reordering terms in \eqref{algebra} it follows that $\big[ (s+t_0) - 1\big]/ (s+t_0)^2 \leq 1/\big[ (s+t_0) + 1\big]$, which upon substitution into \eqref{jhn} leads to the conclusion that 
\begin{align}\label{eqn_lecce_pf_80}
    u_{s+1} \leq  \frac{Q}{s+t_0 + 1}.
\end{align}
Eq. \eqref{eqn_lecce_pf_80} implies that the assumed validity of \eqref{lemma3_claim} for $t=s$ implies the validity of \eqref{lemma3_claim} for $t=s+1$. Combined with the validity of \eqref{lemma3_claim} for $t=0$, which was already proved, it follows that \eqref{lemma3_claim} is true for all times $t\geq0$.
 \end{proof}

%
Lemma \ref{lecce22} shows that satisfying \eqref{claim23} is sufficient for a sequence to have the linear rate of convergence specified in \eqref{lemma3_claim}. In the following proof of Theorem \ref{theo_convergence_rate} we show that if the stepsize sequence parameters $\epsilon_{0}$ and $T_{0}$ satisfy $2\epsilon_{0}T_{0}/C>1$ the sequence $\E{F(\bbw_{t})}-F(\bbw^*)$ of expected optimality gaps satisfies \eqref{claim23} with {$a=2\epsilon_{0}T_{0}/C$, $b=\epsilon_{0}^{2} T_{0}^{2}  {MS^{2}}/{{2c^{2}}}$ and $t_0=T_{0}$.} The result in \eqref{eqn_thm_cvg_rate_20} then follows as a direct consequence of Lemma \ref{lecce22}.

%
\medskip \noindent {\bf Proof of Theorem \ref{theo_convergence_rate}: }
Consider the result in \eqref{pedarsag} of Lemma \ref{helpful} and subtract the average function optimal value $F(\bbw^*)$ from both sides of the inequality to conclude that the sequence of optimality gaps in the RES algorithm satisfies
\begin{align}\label{taylor_objective}
 \E{F(\bbw_{t+1})\given \bbw_{t}} -\ F(\bbw^*)  
        \  \leq \ F(\bbw_{t}) -\ F(\bbw^*) 
          -  \frac{\epsilon_{t}}{C} \| \nabla F(\bbw_{t})\|^{2} 
          +  \frac{\epsilon_{t}^{2}MS^{2}}{{2c^{2}}}.
\end{align}
We proceed to find a lower bound for the gradient norm $\| \nabla F(\bbw_{t})\|$ in terms of the error of the objective value $F(\bbw_{t}) -\ F(\bbw^*) $ {-- this is a standard derivation which we include for completeness, see, e.g., \cite{Boyd}. As it follows from Assumption 1 the eigenvalues of the Hessian $\bbH(\bbw_{t})$ are bounded between $0<m$ and $M<\infty$ as stated in \eqref{bbb}. Taking a Taylor's expansion of the objective function $F(\bby)$ around $\bbw$ and using the lower bound in the Hessian eigenvalues we can write
\begin{equation}\label{taylor_lower_bound}
   F(\bby)\ \geq\ F(\bbw) +\nabla F(\bbw)^{T}(\bby-\bbw)   
   + {{m}\over{2}}\|{\bby - \bbw}\|^{2}.
\end{equation}
For fixed $\bbw$, the right hand side of \eqref{taylor_lower_bound} is a quadratic function of $\bby$ whose minimum argument we can find by setting its gradient to zero. Doing this yields the minimizing argument $\hby = \bbw- (1/m) \nabla  F(\bbw)$ implying that for all $\bby$ we must have
\begin{alignat}{2}\label{lower_bound_for_gradient}
F(\bby)\ \geq\ 
    &\ F(\bbw) +\nabla F(\bbw)^{T}(\hby-\bbw)   
   + {{m}\over{2}}\|{\hby - \bbw}\|^{2} \nonumber \\
   \ =\ 
         &\ F(\bbw) - \frac{1}{2m} \| \nabla F(\bbw)\|^{2} .
\end{alignat}
The bound in \eqref{lower_bound_for_gradient} is true for all $\bbw$ and $\bby$. In particular, for $\bby=\bbw^{*}$ and $\bbw=\bbw_{t}$ \eqref{lower_bound_for_gradient} yields
\begin{equation}\label{transition}
   F(\bbw^*)\ \geq \ F(\bbw_{t}) - \frac{1}{2m} \| \nabla F(\bbw_{t})\|^{2}.
\end{equation} 
Rearrange terms in \eqref{transition} to obtain a bound on the gradient norm squared $\| \nabla F(\bbw_{t})\|^{2}$. Further substitute the result in \eqref{taylor_objective} and regroup terms to obtain the bound
\begin{align}\label{simplified_taylor_expansion}
\E{F(\bbw_{t+1})\given \bbw_{t}} -\ F(\bbw^*)   
       \  \leq\ \left(1-\frac{2m\epsilon_{t}}{C} \right) \big{(}F(\bbw_{t}) -\ F(\bbw^*) \big{)} 
          +\frac{\epsilon_{t}^{2}MS^{2}}{{2c^{2}}}.
\end{align}
Take now expected values on both sides of \eqref{simplified_taylor_expansion}. The resulting double expectation in the left hand side simplifies to $\E{  \E {F(\bbw_{t+1})\given{\bbw_{t}}} } = \E {F(\bbw_{t+1})}$, which allow us to conclude that \eqref{simplified_taylor_expansion} implies that
\begin{align}\label{expectation_inequality}
 \E{F(\bbw_{t+1})} -\ F(\bbw^*)   
       \  \leq\ \left(1-\frac{2m\epsilon_{t}}{C} \right) \big{(} \E {F(\bbw_{t})} -\ F(\bbw^*) \big{)}
          +\frac{\epsilon_{t}^{2}MS^{2}}{{2c^{2}}}  .
\end{align}
Furhter substituting $\epsilon_{t} \!= \! \epsilon_{0}T_{0}/(T_{0}+t)$, which is the assumed form of the step size sequence by hypothesis, we can rewrite \eqref{expectation_inequality} as
\begin{align}\label{sass}
 \E{F(\bbw_{t+1})} -\ F(\bbw^*)  
         \ \leq\ \left(1- \frac{2m \epsilon_{0}T_{0} }{(T_{0}+t)C} \right) 
                     \Big( \E {F(\bbw_{t})} - F(\bbw^*) \Big)
             +\left(\frac{\epsilon_{0}T_{0}}{T_{0}+t}\right)^{2} \! \frac{MS^{2}}{{2c^{2}}}.
\end{align}
Given that the product $2m \epsilon_{0} T_{0} /C >1$ as per the hypothesis, the sequence $\E{F(\bbw_{t+1})} -\ F(\bbw^*)$ satisfies the hypotheses of Lemma \ref{lecce22} with {$a=2m\epsilon_{0}T_{0}/C$, $b=\epsilon_{0}^{2} T_{0}^{2}  {MS^{2}}/{{2c^{2}}}$.} It then follows from \eqref{lemma3_claim} and \eqref{convergence_parameter} that \eqref{eqn_thm_cvg_rate_20} is true for the $C_{0}$ constant defined in \eqref{eqn_thm_cvg_rate_30} upon identifying $u_t$ with $\E{F(\bbx_{t+1})} -\ F(\bbx^*)$, $C_{0}$ with $Q$, and substituting {$c=2m\epsilon_{0}T_{0}/C$, $b=\epsilon_{0}^{2} T_{0}^{2}  {MS^{2}}/{{2c^{2}}}$ and $t_0=T_{0}$} for their explicit values. \hfill \BlackBox

\bibliography{bmc_article}

\end{document}